\documentclass{article}

\usepackage[utf8]{inputenc}
\usepackage[T1]{fontenc}
\usepackage[french,english]{babel}
\usepackage{amsmath}
\usepackage{amssymb}
\usepackage{amsthm}
\usepackage{amsfonts}
\usepackage{enumitem}
\usepackage{mathtools}
\usepackage{geometry}
\usepackage{xparse}
\usepackage{tikz}
\usepackage{tikz-cd}
\usepackage[hide]{todo} 
\usepackage{mathrsfs}
\usepackage{csquotes}
\usepackage{verbatim}
\usepackage{setspace} 
\usepackage[backend=biber,style=alphabetic]{biblatex}
\DeclareFieldFormat{postnote}{#1} 
\usepackage{hyperref} 
\usepackage{cleveref} 

\newtheorem{introthm}{Theorem}

\newtheorem{thm}{Theorem}[section]
\newtheorem{lem}[thm]{Lemma}

\newtheorem{cor}[thm]{Corollary}
\newtheorem{prop}[thm]{Proposition}
\newtheorem*{prop*}{Proposition}
\newtheorem*{thm*}{Theorem}
\newtheorem{defi}[thm]{Definition}
\newtheorem*{defi*}{Definition}

\theoremstyle{definition}

\theoremstyle{remark}
\newtheorem*{rmk}{Remark}

\numberwithin{equation}{section}

\crefname{thm}{thm.}{thms.}
\Crefname{thm}{Theorem}{Theorems}
\crefname{lem}{lem.}{lemmas}
\Crefname{lem}{Lemma}{Lemmas}
\crefname{prop}{prop.}{props.}
\Crefname{prop}{Proposition}{Propositions}
\crefname{defi}{def.}{defs.}
\Crefname{defi}{Definition}{Definitions}
\crefname{section}{§}{§}
\Crefname{section}{Section}{Sections}
\crefname{subsection}{§}{§}
\Crefname{subsection}{Subsection}{Subsections}
\crefname{paragraph}{§}{§}
\Crefname{paragraph}{Paragraph}{Paragraphs}

\newcommand*{\sheafhom}{\mathscr{H}\kern -1.5pt om}
\renewcommand{\cal}{\mathcal}
\renewcommand{\ker}{\mathrm{ker}}

\newcommand{\RR}{\mathbb{R}}
\newcommand{\ZZ}{\mathbb{Z}}
\newcommand{\NN}{\mathbb{N}}
\newcommand{\fr}{\mathfrak}

\newcommand{\QQ}{\mathbb{Q}}
\newcommand{\GG}{\mathbb{G}}
\newcommand{\CC}{\mathbb{C}}
\DeclareMathOperator{\Spec}{Spec}
\DeclareMathOperator{\Hom}{Hom}

\DeclareMathOperator{\Ext}{Ext}

\DeclareMathOperator{\fib}{fib}
\DeclareMathOperator{\cofib}{cofib}
\DeclareMathOperator{\Gal}{Gal}
\DeclareMathOperator{\Ch}{Ch}
\DeclareMathOperator{\Ker}{Ker}

\DeclareMathOperator{\Sh}{Sh}
\DeclareMathOperator{\Disc}{Disc}
\DeclareMathOperator{\rank}{rank}
\DeclareMathOperator{\Pic}{Pic}
\DeclareMathOperator{\Lie}{Lie}
\DeclareMathOperator{\ord}{ord}

\DeclareMathOperator{\ind}{ind}

\tikzset{%
    symbol/.style={%
        draw=none,
        every to/.append style={%
            edge node={node [sloped, allow upside down, auto=false]{$#1$}}}
    }
} 

\title{Tori over number fields and special values at s=1}
\author{Adrien Morin\textsuperscript{\emph{a}}\\
		\textsuperscript{\emph{a}}Institut de Mathématiques de Bordeaux, Université de Bordeaux\\		
		351, cours de la Libération - 33 405 Talence, France\\
		Email : adrien.morin@math.u-bordeaux.fr
		}
\date{}

\begin{document}

\maketitle
\begin{abstract}
	We define a Weil-étale complex with compact support for duals (in the sense of the Bloch dualizing cycles complex $\ZZ^c$) of a large class of $\ZZ$-constructible sheaves on an integral $1$-dimensional proper arithmetic scheme flat over $\Spec(\ZZ)$. This complex can be thought of as computing Weil-étale homology. For those $\ZZ$-constructible sheaves that are moreover tamely ramified, we define an "additive" complex which we think of as the Lie algebra of the dual of the $\ZZ$-constructible sheaf. The product of the determinants of the additive and Weil-étale complex is called the fundamental line. We prove a duality theorem which implies that the fundamental line has a natural trivialization, giving a multiplicative Euler characteristic. We attach a natural $L$-function to the dual of a $\ZZ$-constructible sheaf; up to a finite number of factors, this $L$-function is an Artin $L$-function at $s+1$. Our main theorem contains a vanishing order formula at $s=0$ for the $L$-function and states that, in the tamely ramified case, the special value at $s=0$ is given up to sign by the Euler characteristic. This generalizes the analytic class number formula for the special value at $s=1$ of the Dedekind zeta function. In the function field case, this a theorem of Geisser--Suzuki.
\end{abstract}

\tableofcontents

\section{Introduction}

\subsection{Results}

\paragraph{A new \texorpdfstring{$L$}{L}-function.}
Let $\cal{O}$ be an order in a number field $K$ and $X=\Spec(\mathcal{O})$. Denote $\ZZ^c_X:=\ZZ^c_X(0)$ Bloch's cycle complex (with cohomological indexing), defined on the étale site of $X$; in particular if $X$ is regular we have $\ZZ_X^c=\GG_m[1]$. A sheaf of abelian groups $F$ on the étale site of $X$ is $\ZZ$-constructible if on a dense open it is locally constant associated to a finite type abelian group, and moreover the stalks at all geometric points are finite type abelian groups. If $F$ is a $\ZZ$-constructible sheaf, denote $F^D:=R\mathcal{H}om_X(F,\ZZ^c_X)$. The cohomology of $F^D$ is related to the compactly supported cohomology of $F$ by \emph{Artin-Verdier duality}; thus we think of it as computing homology with coefficients in $F$. To $F^D$, we associate an $L$-function:
\begin{defi*}
	For a complex of étale sheaves $M$, we put
	\[
	M \hat{\otimes}\QQ_l:= \left(R\lim_n (M \otimes^L \ZZ/l^n\ZZ)\right)\otimes \QQ
	\]
	computed on the proétale site. For each closed point $x$ of $X$, let $l=l_x$ be a prime number such that  $l\neq \mathrm{char}(\kappa(x))$ and $L_x$ the usual local factor defined using the geometric Frobenius $\varphi$: for a finite dimensional $\QQ_l$-representation $V$,
	\[
	L_x(V,s):=\det(1-\varphi N(x)^{-s}|V)^{-1}
	\]
	with $N(x)=\mathrm{card}(\kappa(x))$.
	We define the $L$-function of $F^D$ as
	\[
	L_X(F^D,s):=\prod_{x\in X_0} L_x((i_x^\ast F^D)\hat{\otimes} \QQ_{l_x},s).
	\]
\end{defi*}
We compute explicitly the local factors and show that they are well-defined and that the Euler product converges for $s>1$. In fact, denote $g:\Spec(K)\to X$ the inclusion of the generic point and $G_K=\Gal(K^{sep}/K)$. If $V$ is the rational representation of $G_K$ corresponding to $g^\ast F\otimes \QQ$ and $V^\vee$ is its linear dual, then the $L$-function of $F^D$ equals up to a finite number of factors the Artin $L$-function $L_K(V^\vee,s+1)$. In particular, the $L$-function extends to a meromorphic function on $\CC$. If $F$ is a $\ZZ$-constructible sheaf on $\Spec(K)$ associated to a finite type integral $G_K$-representation $M$, we have moreover $L_X((g_\ast F)^D,s)=L_K(\Hom_{Ab}(M,\QQ),s+1)$. We have as a special case $L_X(\ZZ^D,s)=\zeta_X(s+1)$ if $X$ is regular. On the other hand, let $i:x\to X$ be the inclusion of a closed point and $F$ is a  $\ZZ$-constructible sheaf on $x$. In \cite{AMorin21} we introduced $L$-functions for $\ZZ$-constructible sheaves on $X$; then $L_X((i_\ast F)^D,s)=\det(I-\varphi N(x)^{-s}|\cal{H}om_x(F,\ZZ)\otimes \QQ)^{-1}$ is the $L$-function of the $\ZZ$-constructible sheaf $i_\ast \cal{H}om_x(F,\ZZ)$. Finally, if $j:U\hookrightarrow X$ is the inclusion of a dense open subscheme, we have $L_X((j_!\ZZ)^D,s)=\zeta_U(s+1)\times \prod_{x\in X\backslash U} \frac{\zeta_x(s+1)}{\zeta_x(s)}$.

\paragraph{Weil-étale cohomology and special values.}
We apply the Weil-étale formalism of Flach--B. Morin \cite{Flach18} to give a special value formula at $s=0$ of this new $L$-function.
The idea of Weil-étale cohomology originates in \cite{Lichtenbaum05}, where Lichtenbaum constructs a Weil-étale topos for varieties over a finite field and links it to the special value at $s=0$ of zeta functions. Further work was done by Geisser over a finite field \cite{Geisser04}. For schemes over $\Spec(\ZZ)$, attempts at the definition of a Weil-étale topos were made by Lichtenbaum \cite{Lichtenbaum09} and Flach--B. Morin \cite{Flach10} but its cohomology doesn't behave well in high degree. Another approach was instigated by B. Morin in \cite{Morin14} and refined by Flach--B. Morin in \cite{Flach18}: instead of constructing the Weil-étale topos, one only constructs Weil-étale cohomology complexes in the derived category of abelian groups which fit into a certain distinguished triangle. This distinguished triangle comes heuristically from the pushforward from the Weil-étale topos to the étale topos\footnote{The pushforward had been computed by Geisser \cite{Geisser04} in the case over a finite field and by B. Morin in the case of the spectrum of a ring of integers in a number field \cite[§8, §9]{Morin11}.}.

Other relevant works on Weil-étale cohomology include Chiu's thesis \cite{Chiu}, Beshenov's thesis \cite{Beshenov21pI,Beshenov21pII}, Tran's article \cite{Tran16} and the author's article \cite{AMorin21}. Work related to the study of the Weil-étale cohomology of $F^D$ are Geisser--Suzuki's article \cite{Geisser21} and Tran's thesis \cite{Tran15}. As far as the author knows, Tran was the first to observe that for an integral representation $M$ of the Galois group of a number field $K$, the special value of the Artin $L$-function of $M\otimes \QQ$ at $s=1$ should be related to Weil-étale cohomology of the dual of the pushforward of $M$ to $\Spec(\cal{O}_K)$.

Following the formalism of Flach--B. Morin, we should construct for each $F$ a "multiplicative" complex\footnote{Here "multiplicative" suggests that the complex is linked to motivic cohomology: the latter involves the units, the Picard group, etc.}, the Weil-étale complex (with compact support) which we think of as "Weil-étale homology" with coefficients in $F$, and an "additive" complex\footnote{Here "additive" suggests that the complex is linked to coherent phenomena/ de Rham cohomology}, an analogue of Milne's correcting factor in special value formulas for zeta functions of varieties over finite fields. The right object to consider is then the fundamental line $\Delta_{X}(F^D)$, a free abelian group of rank $1$ which is defined as the product of the determinants of the additive and multiplicative complexes. In the general situation of an arithmetic scheme, contrary to the case over a finite field, the additive and multiplicative complexes are linked to each other through phenomena happening on complex points and so cannot be studied independently to get a special value formula; instead one of the fundamental insights of \cite{Flach18} is that one has to study them together through the fundamental line\footnote{This idea has its origin in the formulation of Fontaine--Perrin-Riou of the Bloch--Kato conjecture on special values of L-functions}. The fundamental line should have a canonical trivialization $\Delta_X(F^D)\otimes \RR \xrightarrow{\simeq} \RR$ after base change to $\RR$, which enables one to construct a multiplicative Euler characteristic as the covolume of $\Delta_X(F^D)$ inside $\Delta_X(F^D)\otimes\RR$, and this Euler characteristic should give the special value up to sign.

\paragraph{The Weil-étale complex}
We introduce compactly supported cohomology $R\Gamma_{c,B}(X,F^D)$ and Tate compactly supported cohomology $R\hat{\Gamma}_{c,B}(X,F^D)$.\footnote{The subscript $B$ refers to the way we correct the cohomology at infinity, which involves the Tate twist $\ZZ(1)=2i\pi\ZZ$} The two bear to each other the same kind of relationship as ordinary and Tate cohomology: the fiber of the canonical morphism $R\Gamma_{c,B}(X,F^D)\to R\hat{\Gamma}_{c,B}(X,F^D)$ computes some homology at the archimedean places. Moreover, we prove an Artin-Verdier duality statement for $R\hat{\Gamma}_{c,B}(X,F^D)$, which suggests that compactly supported cohomology of $F^D$ should be thought of as "étale homology" of $F$:
\begin{introthm}[Artin-Verdier duality for $F^D$, see \cref{subsec:AV}]
	There is a natural pairing $R\Gamma(X,F)\otimes^L R\hat{\Gamma}_{c,B}(X,F^D) \to \QQ/\ZZ[-2]$. It induces a map
	\[
	R\hat{\Gamma}_{c,B}(X,F^D) \to R\Hom(R\Gamma(X,F),\QQ/\ZZ[-2])
	\]
	If $F$ is $\ZZ$-constructible, the above map is an isomorphism in degree $\neq -1,0$, and an isomorphism after profinite completion of the left hand side in degree $-1,0$.
\end{introthm}
The proof proceeds by twisting the usual Artin-Verdier theorem \cite[II.3.1]{ADT} by a duality at archimedean places. With the Artin-Verdier-like duality theorem in our hands, we can construct a Weil-étale complex using the methodology of \cite{Morin14}: the cohomology groups with compact support of $F^D$ have a finite type part and a torsion of cofinite type part (i.e. the $\QQ/\ZZ$-dual of a finite type abelian group) and Artin-Verdier duality says that the torsion cofinite type part is the $\QQ/\ZZ$-dual of some étale cohomology of $F$. Taking inspiration from Geisser's and B. Morin's computation of the derived pushforward from the Weil-étale topos to the étale topos \cite{Geisser04, Morin11}, there should exist a fundamental distinguished triangle saying that Weil-étale cohomology with compact support of $F^D$ is obtained by replacing the torsion cofinite type part of cohomology with compact support of $F^D$ by a finite type part, using the short exact sequence $0 \to \ZZ \to \QQ \to \QQ/\ZZ \to 0$. This suggest the existence of a map $\Hom(H^{2-i}(X,F),\QQ) \to H^i_{c,B}(X,F^D)$ making the diagram 
\[\begin{tikzcd}[column sep = huge]
	{\Hom(H^{2-i}(X,F),\QQ)} & {H^i_{c,B}(X,F^D)} \\
	{\Hom(H^{2-i}(X,F),\QQ/\ZZ)} & {\hat{H}^i_{c,B}(X,F^D)}
	\arrow[from=1-1, to=2-1]
	\arrow[dashed, from=1-1, to=1-2]
	\arrow[from=1-2, to=2-2]
	\arrow["{\text{Artin-Verdier}}"', from=2-2, to=2-1]
\end{tikzcd}\]
commute; taking kernels and cokernels will then give something of finite type. We are thus led to consider the existence of a map $R\Hom(R\Gamma(X,F),\QQ[-2]) \xrightarrow{?} R\Gamma_{c,B}(X,F^D)$, the cone of which we want to name Weil-étale cohomology with compact support of $F^D$. We achieve the construction of such a map with good functoriality properties in $D(\ZZ)$ for a large class of $\ZZ$-constructible sheaves, which we name red sheaves and blue sheaves. 
\begin{defi*}
		Let $F$ be a $\ZZ$-constructible sheaf on $X$. We say that
		\begin{itemize}
			\item $F$ is red if $H^0_{c,B}(X,F^D)$ is torsion, hence finite.
			\item $F$ is blue if $H^1(X,F)$ is torsion, hence finite.
			\item A red-to-blue morphism is a morphism of sheaves $F\to G$ where either $F$ and $G$ are both blue, or are both red, or $F$ is red and $G$ is blue ; a red-to-blue short exact sequence is a short exact sequence with red-to-blue morphisms.
		\end{itemize}
\end{defi*}
Examples of red sheaves are extensions by zero of locally constant $\ZZ$-constructible sheaves on a regular open subscheme. Examples of blue sheaves are $\ZZ$-constructible sheaves supported on a finite closed subscheme. The important point is that there are "enough red and blue sheaves", meaning that any $\ZZ$-constructible sheaf can be put in a short exact sequence where the first term is red and the last is blue.\footnote{Namely, take the short exact localization sequence associated to a regular dense open subscheme on which the sheaf is locally constant.}
\begin{introthm}[Existence of the Weil-étale complex, see \cref{sec:construction_weil_etale}]
	For every red or blue sheaf $F$, there exists a Weil-étale complex with compact support $R\Gamma_{W,c}(X,F^D)\in D(\ZZ)$, well-defined up to unique isomorphism. It sits in a distinguished triangle
	\[
	R\Hom(R\Gamma(X,F),\QQ[-2]) \to R\Gamma_{c,B}(X,F^D)\to R\Gamma_{W,c}(X,F^D) \to
	\]
	It is a perfect complex, functorial in red-to-blue morphisms, and it yields a long exact cohomology sequence for red-to-blue short exact sequences. If $Y=\Spec(\cal{O}')$ is the spectrum of an order in a number field with a finite dominant morphism $\pi:Y\to X$, we have a canonical isomorphism $R\Gamma_{W,c}(X,(\pi_\ast F)^D)\simeq R\Gamma_{W,c}(Y,F^D)$.
\end{introthm}
The Weil-étale cohomology with compact support of $F^D$ is constructed from $R\Gamma_{c,B}(X,F^D)$, so it can be thought of as some "Weil-étale homology" of $F$. The idea of such a homology theory is not new: Geisser had defined "arithmetic homology" for curves over finite fields \cite{Geisser12}, and there is a tentative construction of a Weil-étale complex of $F^D$ over the spectrum of a ring of integers in a totally imaginary number field in Tran's thesis \cite{Tran15}. Finally, let $K$ be a function field associated to a smooth proper curve $C$ over a finite field. Geisser--Suzuki showed in \cite{Geisser20} that for a given torus over $K$, its connected Néron model $\cal{T}^\circ$ over $C$ is of the form $F^D$ for $F$ a complex defined in terms of the character group of $T$, and they linked the special value of the $L$-function of the torus at $s=1$ to Weil-étale cohomology of $\cal{T}^\circ$; this suggested that our approach might be fruitful (see \cref{par:comparison_GS} for a precise comparison).

Having an $\infty$-categorical construction of the map $R\Hom(R\Gamma(X,F),\QQ[-2]) \to R\Gamma_{c,B}(X,F^D)$ would have been better, but this construction has eluded us. Instead, we use a kind of miraculous vanishing of $\Ext^1$ groups, for red or blue sheaves, which enables to define the map by only specifying the maps it induces in cohomology. This vanishing is not true for arbitrary $\ZZ$-constructible sheaves, which explains our restriction. It will turn out that to define the Weil-étale Euler characteristic this is not a problem.

\paragraph{The additive complex}
We now have the multiplicative part of the fundamental line, so we turn to the additive part, which we dub the tangent space of $F^D$:
\begin{defi*}[The tangent space of $F^D$]
	Let $F$ be  a $\ZZ$-constructible sheaf on $X$ and denote $g:\Spec(K) \to X$ the inclusion of the generic point. We identify $g^\ast F$ with a discrete $G_K$-module $M$. Denote also $\cal{O}_K$ the ring of integers in $K$.
	\begin{itemize}
		\item We say that $F$ is tamely ramified if for each $x\in X_0$, the tame ramifiction group at $x$ acts trivially on $M$. Let $K^t$ be the maximal tamely ramified extension of $K$, and $G_K^t:=\Gal(K^t/K)$; if $F$ is tamely ramified, $M$ carries a natural action of $G_K^t$.
		\item Suppose that $F$ is tamely ramified. The tangent space of $F^D$ is the complex
		\[
		\Lie_X(F^D):=R\Hom_{G_K^t}(M, \cal{O}_{K^t}[1])
		\]
		\item Suppose that $F$ is tamely ramified and red or blue. The fundamental line is
		\[
		\Delta_X(F^D) := \det_\ZZ R\Gamma_{W,c}(X,F^D) \otimes \det_\ZZ \Lie_X(F^D)^{-1}
		\]
	\end{itemize}
\end{defi*}
Our definition of the tangent space is made so that in particular, when $X=\Spec(\cal{O}_K)$ is regular, the tangent space of $\ZZ^D=\ZZ^c_X=\GG_m[1]$ is $R\Gamma(X,\GG_a[1])=\cal{O}_K[1]$, and the tangent space of the dual of a sheaf supported on a finite closed subscheme is $0$. This is as expected from the $L$-function we introduce: indeed for $\ZZ^D$ it gives the zeta function at $s+1$ so its special value at $0$ is the special value of the Dedekind zeta function at $1$ and should involve the contribution from $\cal{O}_K$ (the discriminant); this was shown already (in the Weil-étale formalism) by Flach--B. Morin in the case of $\ZZ(1)=(\ZZ^D)[-2]$ for $X$ regular \cite{Flach18}. On the other hand, for the dual of a sheaf supported on a finite closed subscheme we find an $L$-function as in \cite[§ 6.4]{AMorin21}, for which there is no "additive" contribution for the special value at $0$. The restriction to tamely ramified sheaves is justified by a theorem of Noether \cite{Noether32}, which implies that $\cal{O}_{K^t}$ is cohomologically trivial. We are still investigating how to remove the tamely unramified hypothesis. The definition was inspired by \cite{Geisser20}, which uses the Lie algebra of the Néron model of a torus $T$ over $K$ to form the additive part of the fundamental line giving the special value at $s=1$ of the $L$-function of $T$ (see \cref{par:comparison_GS}). Our definition is a generalization of this construction in the tamely ramified case that includes finite groups of multiplicative type, as shows the following:
\begin{prop*}[{see \cref{lie_equals_lie}}]
	Suppose that $F$ is tamely ramified and that $g^\ast F$ is torsion-free. Denote $T$ the torus over $K$ with character group $g^\ast F$ and $\cal{T}$ its Néron model over $\Spec(\cal{O}_K)$. Then there is a canonical isomorphism
	\[
	\Lie_X(F^D)\simeq \Lie(\cal{T})[1]
	\]
\end{prop*}
\paragraph{Trivialization of the fundamental line}
Now that we have our fundamental line, we should seek a canonical trivialization to obtain the Euler characteristic. We propose a contravariant generalization of the complex $R\Gamma_c(X,\RR(1))$ of \cite{Flach18}\footnote{The complex of Flach--B. Morin is the mapping fiber of the Beilinson regulator between motivic cohomology (tensored with $\RR$) and (real) Deligne cohomology on the complex points, in weight $1$; see \cite[§ 2.1]{Flach18}}:
\begin{defi*}
	The map $\log|-|:\CC^\times \to \RR$ induces a natural map
	\[
	\mathrm{Log} : R\Gamma(X,F^D)_\RR \to R\Hom_{G_\RR,X(\CC)}(\alpha^\ast F,\RR[1]).
	\]
	For $F$ a $\ZZ$-constructible sheaf, we define the Deligne compactly supported complex with coefficients in $F^D$ by
		\[
		R\Gamma_{c,\cal{D}}(X,F^D_\RR):=\mathrm{fib}\left(R\Gamma(X,F^D)_\RR \xrightarrow{\mathrm{Log}} R\Hom_{G_\RR,X(\CC)}(\alpha^\ast F,\RR[1])\right)
		\]
\end{defi*}
Then, again following \cite{Flach18}, we introduce Weil-Arakelov complexes:
\begin{defi*}
	Let $F$ be a $\ZZ$-constructible sheaf on $X$. We define Weil-Arakelov complex of $F^D$ as the complex:
	\[
	R\Gamma_{ar,c}(X,F^D_\RR):=R\Gamma_{c,\cal{D}}(X,F^D_\RR)[-1] \oplus R\Gamma_{c,\cal{D}}(X,F^D_\RR)
	\]
\end{defi*}
The determinant of the Weil-Arakelov complex has a canonical trivialization; we will obtain the trivialization of the fundamental line by relating the fundamental line with the Weil-Arakelov complex through a duality theorem and the rational splitting of Weil-étale cohomology:
\begin{introthm}[Duality theorem for $\RR$-coefficients, see \cref{sec:deligne_and_duality}]
	There is a natural pairing
	\[
	(R\Gamma(X,F)\otimes\RR) \otimes^L_\RR R\Gamma_{c,\cal{D}}(X,F^D_\RR) \to\RR[0]
	\]
	It induces a map
	\[
	R\Gamma_{c,\cal{D}}(X,F^D_\RR)\to R\Hom(R\Gamma(X,F),\RR)
	\]
	which is an isomorphism for $F\in D^+(X_{et})$. If moreover $F$ is a bounded complex with $\ZZ$-constructible cohomology groups, both sides are perfect complexes of $\RR$-vector spaces.
\end{introthm}
\begin{prop*}[Rational splitting of Weil-étale cohomology, see \cref{prop:rational_splitting}]
	Let $F$ be a red or blue sheaf.	The defining distinguished triangle of Weil-étale cohomology splits rationally to give an isomorphism
	\[
	R\Gamma_{W,c}(X,F^D)\otimes \QQ \xrightarrow{\simeq} R\Hom(R\Gamma(X,F),\QQ[-1]) \oplus R\Gamma_{c,B}(X,F^D)\otimes\QQ
	\]
	natural in red-to-blue morphisms and red-to-blue short exact sequences, and compatible with finite dominant morphisms between spectra of orders in number fields.
\end{prop*}

The duality theorem and the rational splitting imply that there is a distinguished triangle
\[
R\Gamma_{ar,c}(X,F^D_\RR) \to R\Gamma_{W,c}(X,F^D)\otimes\RR \to \Lie_X(F^D)\otimes\RR
\]
which gives the natural trivialization
\begin{align*}
	\lambda:\Delta_X(F^D)_\RR = \det_\RR(R\Gamma_{W,c}(X,F^D)\otimes\RR) \otimes \det_\RR(\Lie_X(F^D)\otimes\RR)^{-1} &\xrightarrow{\simeq} \det_\RR(R\Gamma_{ar,c}(X,F^D_\RR) ) \xrightarrow{\simeq}\RR
\end{align*}

\paragraph{The Euler characteristic and the special value theorem}
\begin{defi*}
	\begin{itemize}
		\item[]
		\item  Let $F$ be a tamely ramified red or blue sheaf. The Weil-étale Euler characteristic of $F^D$ is the positive real number $\chi_X(F^D)$ such that
		\[
		\lambda(\Delta_X(F^D))=\chi_X(F^D)^{-1}\cdot\ZZ \hookrightarrow \RR
		\]
		\item  Let $F$ be a tamely ramified $\ZZ$-constructible sheaf. There exists a short exact sequence $0 \to F' \to F \to F''\to 0$ with $F'$ red and tamely ramified and $F''$ blue and tamely ramified; define
		\[
		\chi_X(F^D)=\chi_X((F')^D)\chi_X((F'')^D)
		\]
		It does not depend on the chosen sequence.
	\end{itemize}
\end{defi*}
The constructed Euler characteristic is multiplicative thanks to the functoriality properties of our constructions. We have an explicit computation:
\begin{prop*}[see \cref{explicit_characteristic}]
	Let $F$ be a tamely ramified $\ZZ$-constructible sheaf and suppose that $X$ is regular. Put
	We have
	\[
	\chi_X(F^D)=\frac{(2\pi)^{r_2(F)}2^{r_1(F)}[H^0(X,F)_{tor}][\Ext^1_X(F,\GG_m)_{tor}]R(F^D)}{[H^1(X,F)_{tor}][\Hom_X(F,\GG_m)_{tor}][\Ext^1_{G_K^t}(F_\eta,\cal{O}_{K^t})][N_2]\Disc(F)}
	\]
	where $r_1(F)$ and $r_2(F)$ are some positive integers, $N_2$ is a certain finite $2$-torsion abelian group, $R(F^D)$ is a regulator-type real number and $\Disc(F)$ is a square-root-of-discriminant-type real number.
\end{prop*}

Now that we have an Euler characteristic at our disposition, the general method following work of Tran \cite{Tran16} (also used in \cite{Geisser21} and \cite{AMorin21}) is to reduce the special value formula \emph{via} Artin induction to computations in specific cases. We obtain:

\begin{introthm}[Special value formula, see \cref{subsec:computations_special_value_thm}]
	Let $F$ be a $\ZZ$-constructible sheaf. We have the vanishing order formula
	\[
	\mathrm{ord}_{s=0}L_X(F^D,s)=\sum (-1)^i i\cdot \dim_\RR H^i_{ar,c}(X,F^D_\RR)
	\]
	If $F$ is  tamely ramified and red or blue, we have the special value formula
	\[
	\lambda^{-1}(L_X^\ast(F^D,0)^{-1}\cdot \ZZ)=\Delta_X(F^D)
	\]
	In general, if $F$ is a tamely ramified $\ZZ$-constructible sheaf, we have the special value formula
	\[
	L_X^\ast(F^D,0)=\pm \chi_X(F^D)
	\]
\end{introthm}
For $F=\ZZ$ we find the analytic class formula for the Dedekind zeta function, so this can be seen as a wide generalization of the analytic class number formula. In the singular case, it does not give a special value formula at $s=1$ of the zeta function: indeed if $Z$ is the singular locus and $\pi:Y\to X$ is the normalization, $L_X(\ZZ^D,s)=\zeta_{X\backslash Z}(s+1) \times \prod_{z\in Z} \prod_{\pi(y)=z} \frac{\zeta_y(s+1)}{\zeta_y(s)}$. Using a remark made by Chen in her thesis \cite{Chen17}, we observe that the case of a constructible sheaf (which states that $\chi_X(F^D)=1$) follows formally using \cite[Corollary 1]{Swan63} from the case of a sheaf supported on a closed subscheme and the multiplicativity properties of the Euler characteristic (see \cref{special_value_constructible}). Applying this remark to our construction in \cite{AMorin21}, this gives a quick proof of Tate's formula for Euler characteristics in global fields (see \cite[I.5.1, II.2.13]{ADT} and \cite[6.23]{AMorin21} for the reduction to Tate's formula).

\subsection{Comparison with other works}

\paragraph{\cite{Tran15}} Our work represents a significant improvement on Tran's thesis. The construction of Tran, which takes place over $X=\Spec(\cal{O}_K)$ in the totally imaginary case, involves a complex $D_{F^D}$ for so-called "strongly $\ZZ$-constructible sheaves" with a naturally attached real number $\chi_{Tran}(F^D)$. In the previous chapters of his thesis, Tran had constructed a complex $D_F$ with a naturally attached real number\footnote{We do not say Euler characteristic because the multiplicativity is not proven} $\chi_{Tran}(F)$ such that, for $M$ a finite type torsion-free discrete $G_K$-module corresponding to an étale sheaf $Y$ on $\Spec(K)$ and $g:\Spec(K)\to X$ the canonical morphism, $L_K^\ast(M,0)=\pm \chi_{Tran}(g_\ast Y)$. Let $d=\rank_\ZZ Y$; Tran computes
\[
\chi_{Tran}((g_\ast Y)^D)=\frac{\chi_{Tran}(F)(2\pi)^{nd/2}}{\sqrt{|\Delta_K|}^d}
\]
and essentially deduces from the functional equation the formula
\[
L_K^\ast(M,1)=\pm \frac{\chi_{Tran}((g_\ast Y)^D)}{N_{K/\QQ}(\mathfrak{f}(M))}
\]
with $\mathfrak{f}(M)$ the Artin conductor. The quantity $N_{K/\QQ}(\mathfrak{f}(M))$ is shown to be related to two natural integral structures on the Lie algebra $\Lie(D(M))$ of the torus $D(M)$ with character group $M$.\footnote{It seems to us that there is a mistake there, as Tran claims that $\Lie(D(M))=\Hom_\ZZ(M,K)$ while the correct formula should be $\Lie(D(M))=\Hom_{G_K}(M,K^{sep})$. It is not clear to us what impact this potential mistake has on his results.}

Tran's thesis only works out the case of a totally imaginary number field, which avoids the handling of factors of $2$ linked to real places. Our construction takes places over an order in an arbitrary number field; in particular we have to be careful with the $2$-torsion, which is achieved through our compactly supported cohomology, and with singularities, which are handled by using the dualizing complex $\ZZ^c$ instead of $\GG_m[1]$. Though we restrict ourselves to tamely ramified $\ZZ$-constructible coefficients, we obtain a multiplicative Euler characteristic, given \emph{on the nose} by the fundamental line, which describes the special value at $s=0$ of the $L$-function $L_X(F^D,s)$ (which includes as a special case Artin $L$-functions at $s=1$ for tamely ramified finite type discrete $G_K$-modules). We also have to place some restrictions on the $\ZZ$-constructible sheaves to obtain a construction of Weil-étale complexes, but these restrictions are significantly weaker than the "strongly $\ZZ$-constructible" condition of Tran and give a well-behaved complex that has good functorial properties.

\paragraph{\cite{Geisser20}}\label{par:comparison_GS} Our work can be seen as adapting and generalizing to the number field case a part of Geisser--Suzuki's work on special values of $1$-motives over a function field $K$, specifically the part about tori. The analogy between our work and theirs is as follows: let $K$ be a global field, let $T$ be a torus over $K$ with character group $M$, let $X$ be either $\Spec(\cal{O}_K)$ or the smooth complete curve with function field $K$ and let $g:\Spec(K)\to X$ be the canonical morphism. In the function field case, Geisser--Suzuki consider the connected Néron model $\cal{T}^\circ$ of $T$ over $X$. They prove that $\cal{T}^\circ\simeq R\cal{H}om_X(\tau^{\leq 1}Rg_\ast M,\GG_m)$ as an étale sheaf. The cup product with a generator $e\in H^1_W(X,\ZZ)\simeq \ZZ$ induces a trivialization on $\det_\RR (R\Gamma_{W}(X,\cal{T}^\circ)\otimes\RR)$ while the complex $R\Gamma_{Zar}(X,\cal{L}ie(\cal{T}^\circ))\otimes\RR$ is trivial; this induces a trivialization $\lambda_e : \left((\det_\ZZ R\Gamma_W(X,\cal{T}^\circ))^{-1}\otimes \det_\ZZ R\Gamma_{Zar}(X,\cal{L}ie(\cal{T}^\circ)\right)\otimes \RR \xrightarrow{\simeq} \RR$. After some reformulation, the special value formula for the $L$-function of $T$ is
\begin{thm*}[{\cite[thm. 4.6]{Geisser20}}]
	\[\lambda_e^{-1}(L^\ast(T,1)^{-1}\cdot \ZZ) =\left(\det_\ZZ R\Gamma_W(X,\cal{T}^\circ)\right)^{-1}\otimes \det_\ZZ R\Gamma_{Zar}(X,\cal{L}ie(\cal{T}^\circ))\]
\end{thm*}

In the number field case, denote $F:=\tau^{\leq 1}Rg_\ast M$ so that $F^D=R\cal{H}om_X(\tau^{\leq 1}Rg_\ast M,\GG_m)[1]=\cal{T}^\circ[1]$; the sheaf $g_\ast M$ is red and moreover $R^1g_\ast M$ is constructible, so the conclusions of both thm. A and thm. B hold without modification for $F$ and we have a Weil-étale complex with compact support $R\Gamma_{W,c}(X,\cal{T}^\circ):=R\Gamma_{W,c}(X,F^D)[-1]$. We have $L_X(F^D,s)=L(T,s+1)$. Finally when $M$ is tamely ramified, since $R^1g_\ast M$ is supported on a finite closed subscheme we have $\Lie_X(F^D)=\Lie_X((g_\ast M)^D)=\Lie(\cal{T})[1]=\Lie(\cal{T}^\circ)[1]=R\Gamma_{Zar}(X,\cal{L}ie(\cal{T}^\circ))[1]$ so our formula is the direct analogue of the previous one:
\begin{thm*}
	\[
	\lambda^{-1}(L^\ast(T,1)^{-1}\cdot \ZZ) = \left(\det_\ZZ R\Gamma_{W,c}(X,\cal{T}^\circ)\right)^{-1} \otimes \det_\ZZ R\Gamma_{Zar}(X,\cal{L}ie(\cal{T}^\circ))
	\]
\end{thm*}
Since we consider all coefficients $F^D$ for $F$ a $\ZZ$-constructible sheaf, we gain some flexibility which allows us to consider the case where $X$ is singular. The main technical difficulty in our situation seems to be the definition of Weil-étale cohomology, while in the function field case Weil-étale cohomology is a well-established construction by work of Lichtenbaum and Geisser \cite{Lichtenbaum05,Geisser04}.

\paragraph{\cite{Poonen20}} Jordan--Poonen proved an analytic class number formula for the zeta function of a $1$-dimensional affine reduced arithmetic scheme. Our work is related but different in nature, as the functions we consider are not the same: for $j:U\to X$ an open immersion, we have $L_X((j_!\ZZ)^D,s)=\zeta_U(s+1)$ only when $U=X$ and $X$ is regular. Let $X=\Spec(\cal{O})$ be the spectrum of an order in a number field $K$, let $j:U\to X$ be an open subscheme and let $S_f$ be the finite places missing in the normalization of $U$. Denote $A:=\cal{O}_U(U)$. Their special value formula for $U$ is
\[
\zeta_U^\ast(1)=\frac{2^{r_1}(2\pi)^{r_2}h(A)R(A) \prod_{x\in S_f}((1-N(x)^{-1})/\log N(x))}{\omega(A) \sqrt{|\Delta_A|}}
\]
with $r_1$, $r_2$ the number of real and complex places of $K$, $h(A):=[\Pic(U)]$, $R(A)$ is the covolume of $A^\times \subset \cal{O}_{K,S}^\times$ under the usual logarithmic embedding, $\omega(A)=[(A^\times)_{tor}]$ and $\Delta_A$ is the discriminant of $A$, that is $\det(\mathrm{Tr}(e_ie_j))$ for a $\ZZ$-basis $e_i$ of $A$.

Meanwhile our formula is computed explicitly as
\[
L_X^\ast((j_!\ZZ)^D,0)=\frac{2^{r_1}(2\pi)^{r_2}h_U R_U}{\omega \sqrt{|\Delta_K|}}
\]
with $h_U=[\mathrm{CH}_0(U)]$, $R_U$ the regulator from \cite{AMorin21} and $\omega$ the number of roots of unity in $K$.

\subsection{Notations}

\paragraph{Schemes}
The étale site of a scheme $Y$ will be denoted $Y_{et}$ and the pro-étale site $Y_{proet}$ \cite{Bhatt15}. There is morpism of topoi $\nu:\Sh(Y_{proet})\to \Sh(Y_{et})$ such that $\nu^\ast$ is fully faithful. A sheaf of abelian groups on $Y_{et}$ will be called an étale sheaf for short.

An arithmetic scheme is a scheme separated of finite type over $\Spec(\ZZ)$; an arithmetic curve is a dimension $1$ arithmetic scheme. In this paper, we will consider proper integral arithmetic curves that are flat over $\Spec(\ZZ)$. Such a scheme is the spectrum of an order $\mathcal{O}$ in a number field $K$. For the rest of the paper, we fix $X=\Spec(\cal{O})$. The generic point will be denoted $g:\eta=\Spec(K)\to X$. We write $G_K:=\Gal(K^{sep}/K)$ for the Galois group of $K$. If $v$ is a finite place (resp. an archimedean place) of $K$, we will denote $K_v$ the henselian local field (resp. complete local field) at $v$ and $g_v:\eta_v=\Spec(K_v)\to X$ the canonical morphism. If $x$ is a closed point of $X$, we will denote $i_x:x\to X$ the inclusion (or $i$ when the context is clear), $G_x=\Gal(\kappa(x)^{sep}/\kappa(x))$ the Galois group of the residue field at $x$ and $N(x)=[\kappa(x)]$ the cardinality of the residue field at $x$. When the context allows, we will abuse notation and write $v$ for a regular closed point of $X$ corresponding to a finite place $v$ of $K$. If $U$ is an open subscheme of $X$, we will denote $j:U\to X$ the corresponding morphism. We denote by $X_0$ the set of closed points of $X$.

For $F$ an étale sheaf on $X$, we denote $F_\eta:=g^\ast F$ and $F_x = i_x^\ast F$ for $x$ a closed point, and we identify those with a discrete $G_K$-module resp. discrete $G_x$-module. We will make no distinction between an étale sheaf on the spectrum of a field and the associated Galois module.

\paragraph{Dualizing complex}
Denote $z_0(X,i)$ the free abelian group generated by closed integral subschemes of $X\times \Delta^i$, of relative dimension $i$, which intersect all faces properly. The dualizing complex $\ZZ^c_X:=\ZZ^c_X(0)$ is the complex of étale sheaves with $z_0(-,-i)$ in degree $i$ and differentials the sum of face maps \cite{Geisser10}. Denote
\[
\GG_X:=\left[g_\ast \GG_m \to \bigoplus_{x\in X_0}i_{x,\ast} \ZZ\right].
\]
Deninger's dualizing complex \cite{Deninger87}. In our case we have $\ZZ^c_X\simeq \GG_X[1]$ \cite{Nart89}.

For $F$ an étale sheaf on $X$, we denote $F^D:=R\cal{H}om_X(F,\ZZ^c_X)$ the derived internal hom in the derived category of étale shaves on $X$.

\paragraph{Group cohomology}
If $G$ is a finite group, we denote $R\Gamma(G,-)$ the derived functor of $G$-invariants. Let $P^\bullet$ be the standard complete resolution of $\ZZ$. We denote furthermore $R\hat{\Gamma}(G,-)\coloneq R\Hom_G(P^\bullet,-)$. This latter functor computes Tate cohomology. 

Let $G$ be a profinite group and $H\subset G$ is an open subgroup. If $M$ is a discrete $H$-module, then the induction of $M$ to $G$ is the $G$-module $\ind_H^G M:=\mathrm{Cont}_{H}(G,M)$ of $H$-equivariant continuous maps $G\to M$. Let $L/K$ be a finite extension of fields and $F$ an étale sheaf on $\Spec(L)$ corresponding to a discrete $G_L$-module $M$. If $\pi$ denotes the map $\Spec(L)\to \Spec(K)$ then $\pi_\ast F$ corresponds to the discrete $G_K$-module $\ind_{G_L}^{G_K}M$.

If $M$ is an abelian group, we will denote $M^\vee:=\Hom_\ZZ(M,\ZZ)$ its linear dual, $M^\ast:=\Hom_\ZZ(M,\QQ/\ZZ)$ its Pontryagin dual, and $M^\dagger:=\Hom_\ZZ(M,\QQ)$\footnote{"The dagger kills torsion."}. If $V$ is a vector space over a field $E$, we will denote $V^\vee$ its linear dual.

\paragraph{Complex points}
We endow the complex points $X(\CC)$ with the analytic topology; in our case the complex points are the embeddings $K\to \CC$ and the topology is discrete. We denote $\Sh(G_\RR,X(\CC))$ the topos of $G_\RR$-equivariant sheaves on $X(\CC)$. There is a morphism of topoi $\alpha:\Sh(G_\RR,X(\CC))\to \Sh(X_{et})$. We denote $\ZZ(1):=2i\pi\ZZ\in \Sh(G_\RR,X(\CC)$, and for $M$ a $G_\RR$-equivariant sheaf of abelian groups on $X(\CC)$, $M(1):=M\otimes \ZZ(1)$ with diagonal action and
\begin{align*}
	& M^\vee := R\mathcal{H}om_{G_\RR,X(\CC)}(M,\ZZ)\\
	& M^\vee(1):=R\mathcal{H}om_{G_\RR,X(\CC)}(M,\ZZ(1)).
\end{align*}
The $G_\RR$-equivariant cohomology of a complex $C\in D(\Sh(G_\RR,X(\CC)))$ is defined as
\[
R\Gamma_{G_\RR}(X(\CC),C):=R\Gamma(G_\RR,R\Gamma(X(\CC),C))
\]
We also denote
\[
R\hat{\Gamma}_{G_\RR}(X(\CC),C):=R\hat{\Gamma}(G_\RR,R\Gamma(X(\CC),C))
\]
the Tate $G_\RR$-equivariant cohomology. The norm map $N$ induces a fiber sequence
\[
\ZZ\otimes_{\ZZ[G]}^L R\Gamma(X(\CC),C) \xrightarrow{N} R\Gamma_{G_\RR}(X(\CC),C) \to R\hat{\Gamma}_{G_\RR}(X(\CC),C)
\]
\begin{rmk}
	Since $X(\CC)$ is discrete, a $G_\RR$-equivariant sheaf on $X(\CC)$ is the data of
	\begin{itemize}
		\item For each embedding $\sigma:K\to \CC$, an abelian group $F_\sigma$
		\item An isomorphism $F_{\sigma} \to F_{\overline{\sigma}}$ for $\sigma$ complex and an action of $G_\RR$ on $F_{\sigma}$ for $\sigma$ real. 
	\end{itemize}
	Choose an embedding $\sigma_v$ for each archimedean place and put $F_v=F_{\sigma_v}$. We then have
	\[
	R\Gamma(X(\CC),F)=\prod_{v ~\text{real}} F_v \times \prod_{v~\text{complex}} \mathrm{ind}^{G_\RR} F_v
	\]
	as a $G_\RR$-module, and thus
	\[
	R\Gamma_{G_\RR}(X(\CC),F)=\prod_{v~\text{archimedean}} R\Gamma(G_{K_v},F_v)
	\]
	This breaks down if $X(\CC)$ is not discrete.
\end{rmk}

If $F$ is an étale sheaf on $X$, we will denote $F_{v}$ its pullback to $\Spec(K_v)$; thus $(\alpha^\ast F)_v=F_v$.

\paragraph{Determinants}
Let $A=\ZZ$ or $\RR$. A complex $C$ in the derived category $D(A)$ is perfect if it is bounded with finite type cohomology groups. We will use the determinant construction of Knudsen-Mumford \cite{detdiv}, and the subsequent work of Breuning, Burns and Knudsen, in particular \cite{Breuning08,Breuning11}. Denote $\mathrm{Proj}_A$ the exact category of projective finite type $A$-modules, $\mathrm{Gr}^b(\mathrm{Mod}_A^{ft})$ the bounded graded abelian category of finite type $A$-modules, $D_{perf}(A)$ the derived category of perfect complexes and $\cal{P}_A$ the Picard groupoid of graded $A$-lines. The usual determinant functor
\[
\det_A \in \mathrm{det}(\mathrm{Proj}_A, \mathcal{P}_A), ~ M  \mapsto (\Lambda^{\rank_A M} M,\rank_A M)
\]
extends to a determinant functor $g_A \in \in\mathrm{det}(\mathrm{Gr}^b(\mathrm{Mod}^{ft}_A), \mathcal{P}_A)$. Moreover, the graded cohomology functor $H:D_{perf}(A)\to \mathrm{Gr}^b(\mathrm{Mod}^{ft}_A)$ induces a functor $H^\ast: \mathrm{det}( \mathrm{Gr}^b(\mathrm{Mod}^{ft}_A),\mathcal{P}_A) \to \mathrm{det}( D_{perf}(A),\mathcal{P}_A)$, and we put $\det_A:=H^\ast g_A$. \footnote{This convention is chosen to simplify proofs later on, but of course all determinant functors extending $\det_A \in \mathrm{det}(\mathrm{Proj}_A$ are isomorphic} There is a canonical isomorphism $(\det_ZZ C)\otimes \RR \simeq \det_\RR (C\otimes \RR)$.

\paragraph{Derived $\infty$-categories}
We will use the theory of stable $\infty$-categories, see \cite{HA}. If $\cal{X}$ is a topos, we will denote $\cal{D}(\cal{X})$ the derived $\infty$-category associated to abelian objects in $\cal{X}$. It is a stable $\infty$-category whose homotopy category is the usual derived (1-)category $D(\cal{X})$. In particular we will denote $\cal{D}(X)$ the derived $\infty$-category of étale sheaves on $X$, $\cal{D}(\ZZ)$ the derived $\infty$-category of abelian groups and $\cal{D}(G_\RR,X(\CC))$ the derived $\infty$-category of $G_\RR$-equivariant sheaves on $X(\CC)$.

A stable $\infty$-category has all finite limits and colimits. Moreover, pushouts are pullbacks and reciprocally. The homotopy category of a stable $\infty$-category has a canonical structure of triangulated category. If $\cal{C}$ is a stable $\infty$-category and $A\in \cal{C}$, we denote $A[1]$ the shift of $A$, which is given by the following pushout:
\[\begin{tikzcd}
	A & 0 \\
	0 & {A[1]}
	\arrow[from=1-1, to=1-2]
	\arrow[from=1-1, to=2-1]
	\arrow[from=1-2, to=2-2]
	\arrow[from=2-1, to=2-2]
\end{tikzcd}\]
If $f:A \to B$ is a morphism in $\cal{C}$, we define $\fib(f)$ and $\cofib(f)$ by the following pullback and pushout diagrams:
\[
\begin{array}{lr}
\begin{tikzcd}
	{\fib(f)} & A \\
	0 & B
	\arrow[from=1-1, to=1-2]
	\arrow[from=1-1, to=2-1]
	\arrow["f", from=1-2, to=2-2]
	\arrow[from=2-1, to=2-2]
\end{tikzcd}&
\begin{tikzcd}
	A & B \\
	0 & {\cofib(f)}
	\arrow["f", from=1-1, to=1-2]
	\arrow[from=1-1, to=2-1]
	\arrow[from=1-2, to=2-2]
	\arrow[from=2-1, to=2-2]
\end{tikzcd}
\end{array}\]
A sequence $A\to B \to C$ is called a fiber sequence if $A=\fib(B\to C)$ (or equivalently $C=\cofib(A\to B)$). A fiber sequence induces a distinguished triangle in the homotopy category.

\subsection{Acknowledgments}
 First and foremost I would like to thank my advisor, Baptiste Morin. He introduced me to the beautiful subject of Weil-étale cohomology and progressively guided me towards independance. I would also like to thank Takashi Suzuki for stimulating exchanges about the Lie construction, and Matthias Flach for general discussions on Weil-étale cohomology and related topics. Finally, I would like to thank Thomas Geisser and Niranjan Ramachandran for helpful remarks and discussions.

This paper was written during my doctoral studies at the Institut de Mathématiques de Bordeaux, under a scholarship from the Ministère de l'Enseignement supérieur, de la Recherche et de l'Innovation. This research did not receive any specific grant from funding agencies in the public, commercial, or not-for-profit sectors. Declarations of interest: none.
\section{Cohomology with compact support of \texorpdfstring{$F^D$}{F D}}

\subsection{Definitions}
The $G_\RR$-equivariant sheaf $\alpha^\ast \ZZ^c_X=\alpha^\ast\GG_X[1]$ is $\overline{\QQ}^\times[1]$, with action of $G_\RR$ \emph{via} decomposition groups at archimedean places. We have a canonical morphism $\alpha^\ast \ZZ^c_X \to \CC^\times[1]$ and the short exact sequence $0 \to 2i\pi\ZZ \to \CC \to \CC^\times \to 0$ gives a composite arrow $\alpha^\ast \ZZ^c_X \to \CC^\times[1] \to \ZZ(1) [2]$. Let $F,G$ be abelian sheaves on $X_{et}$. The functor $\alpha^\ast$ is strict monoidal, hence from the arrow
\[
\alpha^\ast (R\cal{H}om_{X}(F,G) \otimes^L F) \to \alpha^\ast G
\]
obtained by applying $\alpha^\ast$ to the natural map, we obtain by adjunction a canonical map
\[
\alpha^\ast R\cal{H}om_X(F,G) \to R\cal{H}om_{G_\RR,X(\CC)}(\alpha^\ast F,\alpha^\ast G)
\]
In particular, there is a natural transformation 
\[
\alpha^\ast(F^D) \to R\mathcal{H}om_{G_\RR,X(\CC)}(\alpha^\ast F,\CC^\times[1]) \to (\alpha^\ast F)^\vee(1)[2]
\]
hence maps
\[
R\Gamma(X,F^D) \to R\Gamma_{G_\RR}(X(\CC),(\alpha^\ast F)^\vee(1)[2]) \to R\hat{\Gamma}_{G_\RR}(X(\CC),(\alpha^\ast F)^\vee(1)[2])
\]

\begin{defi}\label{def_cohomology_compact_support}
	The corrected cohomology with compact support of $F^D$ is the fiber of the first morphism:
	\[
	R\Gamma_{c,B}(X,F^D):=\fib\left(R\Gamma(X,F^D) \to R\Gamma_{G_\RR}(X(\CC),(\alpha^\ast F)^\vee(1)[2])\right)
	\]
	The Tate corrected cohomology with compact support of $F^D$ is the fiber of the composite morphism:
	\[
	R\hat{\Gamma}_{c,B}(X,F^D):=\fib\left(R\Gamma(X,F^D)\to R\hat{\Gamma}_{G_\RR}(X(\CC),(\alpha^\ast F)^\vee(1)[2])\right)
	\]
	We also recall the definition of Tate cohomology with compact support of $F$ from \cite[§ 2]{AMorin21}\footnote{The definition is not original but the terminology is, see \cite[II.2, Cohomology with compact support]{ADT}} as the fiber:
	\[
	R\hat{\Gamma}_c(X,F):=\fib\left(R\Gamma(X,F)\to R\hat{\Gamma}_{G_\RR}(X(\CC),\alpha^\ast F)\right)
	\]
\end{defi}

\begin{rmk}
		The above is only defined for sheaves of the specific form $F^D$. It is not a functor on the whole category of sheaves. We think of it as a contravariant functor in $F$. As we will see later, Tate corrected cohomology with compact support of $F^D$ is in an "Artin-Verdier-like" duality with étale cohomology of $F$, so we can think of it as a étale homology of $F$.
\end{rmk}

\begin{prop}\label{comparison_Tate_compact_cohomology}
	Let $F$ be an étale sheaf on $X$. The canonical map $\alpha^\ast \ZZ^c_X \to \to \ZZ(1) [2]$ induces an isomorphism
	\[
	R\hat{\Gamma}_c(X,F^D) \xrightarrow{\simeq} R\hat{\Gamma}_{c,B}(X,F^D)
	\]
\end{prop}

\begin{proof}
	We have to show that
	\[
	R\hat{\Gamma}(G_\RR,\alpha^\ast (F^D)) \xrightarrow{\simeq} R\hat{\Gamma}(G_\RR,(\alpha^\ast F)^\vee(1)[2])
	\]
	It is enough to show it at real places. For $v$ a real place, denote $g_v:\Spec(K_v)\to X$ the canonical map. As $g_v$ is proétale, we have $g_v^\ast (F^D)=R\cal{H}om_{G_\RR}(g_v^\ast F,\overline{\QQ}^\times[1])$, and we thus want to show
	\[
	R\hat{\Gamma}(G_\RR,R\cal{H}om_{G_\RR}(g_v^\ast F,\overline{\QQ}^\times[1])) \xrightarrow{\simeq} R\hat{\Gamma}(G_\RR,R\cal{H}om_{G_\RR}(g_v^\ast F,2i\pi\ZZ[2]))
	\]
	All terms now depend only on $M:=g_v^\ast F$ so we can reason by Artin induction on any finite type $G_\RR$-module $M$. 
	
	We first treat the case $M=\ZZ$: from the exact sequence $0\to 2i\pi\ZZ \to \CC \to \CC^\times \to 0$ we find an isomorphism
	\[
	R\hat{\Gamma}(G_\RR,\CC^\times[1]) \xrightarrow{\simeq} R\hat{\Gamma}(G_\RR,\ZZ(1)[2])
	\]
	Moreover, $\overline{\QQ}^\times$ and $\CC^\times$ have the same Tate cohomology: both are zero in odd degree and in even degree, both equal $\ZZ/2\ZZ$ \emph{via} the sign map because a real algebraic number which is a norm is the square of a real algebraic number, hence the norm of an algebraic number.
	
	If $M=\mathrm{ind}_{\{0\}}^{G_\RR} \ZZ$ is induced, we have that $R\cal{H}om_{G_\RR}(M,N)\simeq\mathrm{ind}_{\{0\}}^{G_\RR} N$ is induced for any $G_\RR$-module $N$, so it has trivial Tate cohomology.
	
	Finally, if $M$ is finite $R\cal{H}om_{G_\RR}(M,\overline{\QQ}^\times)$ and $R\cal{H}om_{G_\RR}(M,\CC^\times)$ are canonically isomorphic because $\overline{\QQ}^\times$ and $\CC^\times$ have the same torsion. Since $\CC$ is uniquely divisible and has no torsion, we have morover $R\cal{H}om_{G_\RR}(M,\CC^\times) \xrightarrow{\simeq} R\cal{H}om_{G_{\RR}}(M,2i\pi\ZZ[1])$. We thus obtain the required isomorphism on Tate cohomology.
	
	By Artin induction we obtain the required isomorphism for any finite type $G_\RR$-module $M$; by taking filtered colimits, we obtain the statement for any $G_\RR$-module $M$ since Tate cohomology, being a derived Hom, commutes with derived limits.
\end{proof}

\subsection{Artin-Verdier duality for \texorpdfstring{$F^D$}{F D}}\label{subsec:AV}

For a complex of abelian groups, we will denote $R\Hom(-,\QQ/\ZZ)=\Hom^\bullet(-,\QQ/\ZZ)$ by $(-)^\ast$. Artin-Verdier duality gives a map $R\Gamma(X,F^D)=R\Hom_X(F,\GG_X[1]) \xrightarrow{AV} R\hat{\Gamma}_c(X,F)^\ast [-2]$ which is "almost" an isomorphism. We want to modify this duality at the complex points to obtain an Artin-Verdier duality relating $R\hat{\Gamma}_{c,B}(X,F^D)$ and $R\Gamma(X,F)$.
We have fiber sequences
\begin{align*}
&R\hat{\Gamma}_{G_\RR}(X(\CC),(\alpha^\ast F)^\vee(1)[1]) \to R\hat{\Gamma}_{c,B}(X,F^D) \to R\Gamma(X,F^D)\\
&R\hat{\Gamma}_c(X,F) \to R\Gamma(X,F) \to R\hat{\Gamma}_{G_\RR}(X(\CC),\alpha^\ast F)\\
\end{align*}
To obtain an Artin-Verdier duality statement, we will construct in the next three paragraphs pairings
\begin{align*}
&R\hat{\Gamma}_{G_\RR}(X(\CC),(\alpha^\ast F)^\vee(1))\otimes^L R\hat{\Gamma}_{G_\RR}(X(\CC),\alpha^\ast F) \to \QQ/\ZZ[-3]\\
&R\Gamma(X,F)\otimes^L R\hat{\Gamma}_{c,B}(X,F^D) \to \QQ/\ZZ[-2]
\end{align*}
such that we have a morphism of fiber sequences given by the adjoint maps:
\begin{equation}\label{morphism_of_fiber_sequences}
\begin{tikzcd}
R\hat{\Gamma}_{G_\RR}(X(\CC),(\alpha^\ast F)^\vee(1)[1]) \rar \dar & R\hat{\Gamma}_{c,B}(X,F^D) \rar \dar & R\Gamma(X,F^D)  \dar["AV"]\\
R\hat{\Gamma}_{G_\RR}(X(\CC),\alpha^\ast F)^\ast[-2] \rar & R\Gamma(X,F)^\ast[-2] \rar & R\hat{\Gamma}_c(X,F)^\ast[-2]
\end{tikzcd}
\end{equation}

\paragraph{Construction of the pairing for Tate cohomology on $X(\CC)$.}
Let us construct the first pairing: Tate cohomology $R\hat{\Gamma}_{G_\RR}(X(\CC),-)=R\hat{\Gamma}_{G_\RR}(-)\circ R\Gamma(X(\CC),-)$ is lax-monoidal \cite[I.3.1]{Nikolaus18}, hence the natural evaluation map $(\alpha^\ast F)^\vee(1)[1] \otimes^L \alpha^\ast F \to 2i\pi\ZZ[1]$ gives a map
\[
R\hat{\Gamma}_{G_\RR}(X(\CC),(\alpha^\ast F)^\vee(1))\otimes^L R\hat{\Gamma}_{G_\RR}(X(\CC),\alpha^\ast F) \to R\hat{\Gamma}_{G_\RR}(X(\CC),2i\pi\ZZ)
\]
Since $\hat{H}^i_{G_\RR}(X(\CC),2i\pi\ZZ)=0$ for $i$ even, there is a canonical isomorphism\footnote{Indeed, there is always such an isomorphism (non-canonical) in the homotopy category $D(\ZZ)$, and there is a canonical one here because there are no $\Ext^1$ between two consecutive cohomology groups}
\[
R\hat{\Gamma}_{G_\RR}(X(\CC),2i\pi\ZZ)\simeq \bigoplus_i \hat{H}^i_{G_\RR}(X(\CC),2i\pi\ZZ)[-i]
\]
Let $r_1$ be the number of real places of the function field of $K$. There is an identification $\hat{H}^3(G_\RR,X(\CC),2i\pi\ZZ)=(\RR^\times/\RR^\times_{>0})^{r_1}$ comes from the natural map $\CC^\times\to 2i\pi\ZZ[1]$. We compose the previous pairing with a projection and a sum map to get the pairing
\begin{equation}\label{pairing_tate}
\begin{split}
R\hat{\Gamma}_{G_\RR}(X(\CC),(\alpha^\ast F)^\vee(1))\otimes^L R\hat{\Gamma}_{G_\RR}(X(\CC),\alpha^\ast F) &\to R\hat{\Gamma}_{G_\RR}(X(\CC),2i\pi\ZZ) \\&\xrightarrow{H^3} \hat{H}^3_{G_\RR}(X(\CC),2i\pi\ZZ)[-3]\\&=(\RR^\times/\RR^\times_{>0})^{r_1}[-3]\\&\xrightarrow{\Sigma} \QQ/\ZZ[-3].
\end{split}
\end{equation}

\paragraph{Construction of the pairing for $R\hat{\Gamma}_{c,B}(X,F^D)$.}
In the following diagrams, we will denote $R\Gamma$, $R\Gamma_{G_\RR}$, $R\hat{\Gamma}_{G_\RR}$, $R\hat{\Gamma}_c$, $R\hat{\Gamma}_{c,B}$ for $R\Gamma(X,-)$, $R\Gamma_{G_\RR}(X(\CC),-)$, $R\hat{\Gamma}_{G_\RR}(X(\CC),-)$, $R\hat{\Gamma}_c(X,-)$ and $R\hat{\Gamma}_{c,B}(X,-)$ respectively, and we will also write $\otimes$ for the derived tensor product. The natural lax monoidality pairings of $R\Gamma_{G_\RR}(X(\CC),-)$ and $R\hat{\Gamma}_{G_\RR}(X(\CC),-)$ are compatible with each other (\cite[I.3.1]{Nikolaus18}) and also with the natural pairing for $R\Gamma(X,-)$ via the map $R\Gamma(X,-)\to R\Gamma_{G_\RR}(X(\CC),\alpha^\ast(-))$,\footnote{Since the lax monoidal structure for $R\Gamma(X,-)$ and $R\Gamma_{G_\RR}(X(\CC),-)$ come from the strict monoidal structure on their left adjoint, this is a formal consequence of the commutative triangle of right adjoints with strict monoidal left adjoints
\[\begin{tikzcd}[ampersand replacement=\&]
{\cal{D}(G_\RR,X(\CC))} \& {} \& {\cal{D}(X_{et})} \\
\& {\cal{D}(\ZZ)}
\arrow["{R\Gamma_{G_\RR}(X(\CC),-)}"', from=1-1, to=2-2]
\arrow["{R\Gamma(X,-)}", from=1-3, to=2-2]
\arrow["{\alpha_\ast}", from=1-1, to=1-3]
\end{tikzcd}\]}\label{compatibility_pairings}
so the following diagram commutes:
\[\begin{tikzcd}
{R\Gamma(F)\otimes R\Gamma(F^D)} & {} & {R\Gamma(\ZZ^c_X)} \\
{R\Gamma(F)\otimes R\Gamma_{G_\RR}(F^\vee_\CC(1)[2])} & {R\Gamma_{G_\RR}(\alpha^\ast F)\otimes R\Gamma_{G_\RR}(F^\vee_\CC(1)[2])} & {R\Gamma_{G_\RR}(2i\pi\ZZ[2])} \\
{R\Gamma(F)\otimes R\hat{\Gamma}_{G_\RR}(F^\vee_\CC(1)[2])} & {R\hat{\Gamma}_{G_\RR}(\alpha^\ast F)\otimes R\hat{\Gamma}_{G_\RR}(F^\vee_\CC(1)[2])} & {R\hat{\Gamma}_{G_\RR}(2i\pi\ZZ[2])}
\arrow[from=1-1, to=2-1]
\arrow[from=2-1, to=3-1]
\arrow[from=3-1, to=3-2]
\arrow[from=2-1, to=2-2]
\arrow[from=1-1, to=1-3]
\arrow[from=2-2, to=3-2]
\arrow[from=3-2, to=3-3]
\arrow[from=1-3, to=2-3]
\arrow[from=2-3, to=3-3]
\arrow[from=2-2, to=2-3]
\end{tikzcd}\]
We thus obtain the left dotted map in the following commutative diagram where the top row is a fiber sequence:
\[\begin{tikzcd}
{R\Gamma(X,F)\otimes^LR\hat{\Gamma}_{c,B}(X,F^D)} & {R\Gamma(X,F)\otimes^LR\Gamma(X,F^D)} & {R\Gamma(X,F)\otimes^L R\hat{\Gamma}_{G_\RR}(X(\CC),F^\vee_\CC(1)[2])} \\
{R\hat{\Gamma}_{c,B}(X,\ZZ^D)} & {R\Gamma(X,\ZZ^c_X)} & {R\hat{\Gamma}_{G_\RR}(X(\CC),2i\pi\ZZ[2])}
\arrow[dashed, from=1-1, to=2-1]
\arrow[from=1-2, to=2-2]
\arrow[from=1-3, to=2-3]
\arrow[from=1-1, to=1-2]
\arrow[from=1-2, to=1-3]
\arrow[from=2-1, to=2-2]
\arrow[from=2-2, to=2-3]
\end{tikzcd}\]
We have
\[
\tau^{\geq 2} R\hat{\Gamma}_{c,B}(X,\ZZ^D)=\tau^{\geq 2} R\hat{\Gamma}_{c}(X,\ZZ^c_X)=\QQ/\ZZ[-2]
\]
by \cref{comparison_Tate_compact_cohomology} and \cite[II.2.6, II.6.1]{ADT}, hence we get a pairing
\begin{equation}\label{pairing_FD}
R\Gamma(X,F)\otimes^LR\hat{\Gamma}_{c,B}(X,F^D) \to R\hat{\Gamma}_{c,B}(X,\ZZ^D) \xrightarrow{\tau^{\geq 2}} \QQ/\ZZ[-2]
\end{equation}

\paragraph{The morphism of fiber sequences.}

The following cube is commutative by compatibility of the involved pairings:
\[\begin{tikzcd}[sep=small]
{R\Gamma(F)\otimes R\Gamma(F^D)} && {R\Gamma(F)\otimes R\hat{\Gamma}_{G_\RR}((\alpha^\ast F)^\vee(1)[2])} \\
& {R\Gamma(\ZZ^c_X)} & {} & {R\hat{\Gamma}_{G_\RR}(2i\pi\ZZ[2])} \\
{R\hat{\Gamma}_{G_\RR}(\alpha^\ast F)\otimes R\Gamma(F^D)} && {R\hat{\Gamma}_{G_\RR}(\alpha^\ast F)\otimes R\hat{\Gamma}_{G_\RR}((\alpha^\ast F)^\vee(1)[2])} \\
& {R\hat{\Gamma}_{G_\RR}(\alpha^\ast \ZZ^c_X)} && {R\hat{\Gamma}_{G_\RR}(2i\pi\ZZ[2])}
\arrow[from=1-1, to=3-1]
\arrow[from=3-1, to=3-3]
\arrow[from=1-1, to=1-3]
\arrow[from=1-3, to=3-3]
\arrow[from=2-2, to=4-2]
\arrow[from=4-2, to=4-4]
\arrow[from=2-4, to=4-4]
\arrow[from=2-2, to=2-4]
\arrow[from=1-1, to=2-2]
\arrow[from=1-3, to=2-4]
\arrow[from=3-1, to=4-2]
\arrow[from=3-3, to=4-4]
\end{tikzcd}\]

hence by adding the fibers of horizontal and vertical maps we deduce a morphism of $3\times 3$ diagrams:
\begin{equation}\label{compatibility_pairings_3x3}
\begin{tikzpicture}[baseline= (a).base]
\node[scale=.85] (a) at (0,0){
\begin{tikzcd}[sep=tiny]
{R\hat{\Gamma}_c(F)\otimes R\hat{\Gamma}_{c,B}(F^D)} && {R\hat{\Gamma}_c(F)\otimes R\Gamma(F^D)} && {R\hat{\Gamma}_c(F)\otimes R\hat{\Gamma}_{G_\RR}((\alpha^\ast F)^\vee(1)[2])} && {} \\
{} & {R\hat{\Gamma}_c(\ZZ^c_X)} && {R\hat{\Gamma}_c(\ZZ^c_X)} && 0 \\
{R\Gamma(F)\otimes R\hat{\Gamma}_{c,B}(F^D)} && {R\Gamma(F)\otimes R\Gamma(F^D)} & {} & {R\Gamma(F)\otimes R\hat{\Gamma}_{G_\RR}((\alpha^\ast F)^\vee(1)[2])} && {} \\
& {R\hat{\Gamma}_{c,B}(\ZZ^D)} & {} & {R\Gamma(\ZZ^c_X)} & {} & {R\hat{\Gamma}_{G_\RR}(2i\pi\ZZ[2])} \\
{R\hat{\Gamma}_{G_\RR}(\alpha^\ast F)\otimes R\hat{\Gamma}_{c,B}(F^D)} && {R\hat{\Gamma}_{G_\RR}(\alpha^\ast F)\otimes R\Gamma(F^D)} & {} & {R\hat{\Gamma}_{G_\RR}(\alpha^\ast F)\otimes R\hat{\Gamma}_{G_\RR}((\alpha^\ast F)^\vee(1)[2])} \\
& 0 && {R\hat{\Gamma}_{G_\RR}(\alpha^\ast \ZZ^c_X)} & {} & {R\hat{\Gamma}_{G_\RR}(2i\pi\ZZ[2])} \\
&& {}
\arrow[from=3-3, to=3-5]
\arrow[from=5-3, to=5-5]
\arrow[from=3-3, to=5-3]
\arrow[from=3-5, to=5-5]
\arrow[from=1-3, to=3-3]
\arrow[from=1-5, to=3-5]
\arrow[from=1-3, to=1-5]
\arrow[from=1-1, to=1-3]
\arrow[from=1-1, to=3-1]
\arrow[from=3-1, to=3-3]
\arrow[from=3-1, to=5-1]
\arrow[from=5-1, to=5-3]
\arrow[from=4-4, to=6-4]
\arrow[from=3-3, to=4-4]
\arrow[from=5-3, to=6-4]
\arrow["\simeq", from=6-4, to=6-6]
\arrow[from=3-5, to=4-6]
\arrow["\simeq", from=4-6, to=6-6]
\arrow[from=4-4, to=4-6]
\arrow[from=2-6, to=4-6]
\arrow[from=6-2, to=6-4]
\arrow[from=2-4, to=4-4]
\arrow[from=2-4, to=2-6]
\arrow[from=4-2, to=4-4]
\arrow["\simeq"{pos=0.7}, from=2-2, to=4-2]
\arrow[from=4-2, to=6-2]
\arrow[Rightarrow, no head, from=2-2, to=2-4]
\arrow[from=1-1, to=2-2]
\arrow[from=1-3, to=2-4]
\arrow[from=1-5, to=2-6]
\arrow[from=3-1, to=4-2]
\arrow[from=5-1, to=6-2]
\arrow[from=5-5, to=6-6]
\end{tikzcd}	
};
\end{tikzpicture}
\end{equation}
Doing so, we recover in the middle layer the pairing \eqref{pairing_FD}.

Let $A, A', B, B', C, C'$ be objects of $D(\mathrm{Ab})$, with maps $A\to A'$, $B' \to B$, $C\to C'$. Given two pairings $A\otimes^L B \to C$ and $A'\otimes^L B' \to C'$, the commutativity of the two following induced diagrams is equivalent:
\[\begin{tikzcd}
A & {A'} \\
{R\cal{H}om(B,C)} & {R\cal{H}om(B',C')}
\arrow[from=1-1, to=2-1]
\arrow[from=1-1, to=1-2]
\arrow[from=2-1, to=2-2]
\arrow[from=1-2, to=2-2]
\end{tikzcd}
\begin{tikzcd}
{A\otimes^L B'} & {A\otimes^L B} & C \\
{} & {A'\otimes^L B'} & {C'}
\arrow[from=1-3, to=2-3]
\arrow[from=1-2, to=1-3]
\arrow[from=2-2, to=2-3]
\arrow[from=1-1, to=1-2]
\arrow[from=1-1, to=2-2]
\end{tikzcd}\]

We now prove that the natural diagram \eqref{morphism_of_fiber_sequences} commutes. By the above, this reformulates to the commutativity of a subdiagram of \eqref{compatibility_pairings_3x3}, together with the following commutative diagram:
\[\begin{tikzcd}
{R\hat{\Gamma}(2i\pi\ZZ[1])=R\hat{\Gamma}(\CC^\times)} & {R\hat{\Gamma}_{c,B}(\ZZ^D)} & {R\hat{\Gamma}_c(\ZZ^c_X)} \\
{} & {\QQ/\ZZ[-2]}
\arrow["\simeq"', from=1-3, to=1-2]
\arrow["{\tau^{\geq 2}}", from=1-2, to=2-2]
\arrow[from=1-1, to=1-2]
\arrow["{\tau^{\geq 2}}", from=1-3, to=2-2]
\arrow["{\Sigma\circ H^2}"', from=1-1, to=2-2]
\end{tikzcd}\]
which follows from the proof of \cite[II.2.6]{ADT}

\begin{thm}[Artin-Verdier duality for $F^D$]\label{twisted_AV}
The pairing $ R\Gamma(X,F)\otimes^L R\hat{\Gamma}_{c,B}(X,F^D) \to \QQ/\ZZ[-2]$ induces a map
\[
R\hat{\Gamma}_{c,B}(X,F^D) \to R\Gamma(X,F)^\ast[-2]
\]
which is an isomorphism in degree $\neq -1,0$, and an isormophism after profinite completion of the left hand side in degree $-1,0$. In particular if $F$ is constructible then the map is an isomorphism.
\end{thm}

\begin{rmk}
The statement also holds more generally for a bounded complex $F\in \cal{D}^b(X)$ such that $H^0(F)$ is $\ZZ$-constructible and $H^i(F)$ is constructible for $i\neq 0$ (by filtering with the truncations). If we reformulate the theorem as the map $R\hat{\Gamma}_{c,B}(X,F^D)\otimes \hat{\ZZ} \to R\Gamma(X,F)^\ast[-2]$ being an isomorphism, this generalizes to bounded complexes with $\ZZ$-constructible cohomlogy groups.
\end{rmk}

\begin{proof}
	By the diagram \eqref{morphism_of_fiber_sequences} and Artin-Verdier duality for singular schemes \cite[II.6.2]{ADT}, it suffices to show that the pairing 
	\[
	R\hat{\Gamma}_{G_\RR}(X(\CC),(\alpha^\ast F)^\vee(1)[1])\otimes^L R\hat{\Gamma}_{G_\RR}(X(\CC),\alpha^\ast F) \to \QQ/\ZZ[-2]
	\]
	is a perfect pairing between complexes of abelian groups with finite cohomology groups. To prove that the pairing is perfect, it suffices to do it at every real place, that is for the Galois cohomology of $G_\RR$; thus it suffices to prove the next proposition.
\end{proof}

\begin{prop}
	The pairing
	\[
	R\hat{\Gamma}(G_\RR,M^\vee(1))\otimes^L R\hat{\Gamma}(G_\RR,M)\to \QQ/\ZZ[-3]
	\]
	is perfect for $M$ any discrete $G_\RR$-module of finite type.
\end{prop}

\begin{proof}
	We will reduce to \cref{duality_tate_finite_groups} for the finite group $G_\RR=\ZZ/2\ZZ$. Denote $\varepsilon$ the augmentation $\ZZ[G_{\RR}] \to \ZZ$ and $I_{G_\RR}=(s-1)\ZZ$ the augmentation ideal. We have $H^1(G_\RR,2i\pi\ZZ)=\Ext^1_{G_\RR}(\ZZ,2i\pi\ZZ)=\ZZ/2\ZZ$, and the non-zero class $t$ corresponds to a morphism $\ZZ\to 2i\pi\ZZ[1]$ in the derived category coming from the equivalence class of the non-split exact sequence
	\[
	0 \to 2i\pi\ZZ=I_{G_\RR} \to \ZZ[G_{\RR}] \xrightarrow{\varepsilon} \ZZ \to 0
	\]
	Since $I_{G_\RR}=2i\pi\ZZ$ as a $G_\RR$-module, we obtain an isomorphism
	\[
	R\hat{\Gamma}(G_\RR,M)\xrightarrow[(\mathrm{id}_M\otimes t)_\ast]{\simeq} R\hat{\Gamma}(G_\RR,M(1))[1]
	\]
	Note that $M^\vee(1)=(M^\vee)(1)$. From the above isomorphism we deduce an isomorphism of pairings
	\[\begin{tikzcd}
	{R\hat{\Gamma}(G_\RR,M^\vee)\otimes^L R\hat{\Gamma}(G_\RR,M)} & {R\hat{\Gamma}(G_\RR,\ZZ)} & {\hat{H}^2(G_\RR,\ZZ)[-2]} & {\QQ/\ZZ[-2]} \\
	{R\hat{\Gamma}(G_\RR,M^\vee(1))[1]\otimes^L R\hat{\Gamma}(G_\RR,M)} & {R\hat{\Gamma}(G_\RR,2i\pi\ZZ)[1]} & {\hat{H}^{3}(G_\RR,2i\pi\ZZ)[-2]} & {\QQ/\ZZ[-2]}
	\arrow["\simeq", from=1-1, to=2-1]
	\arrow[from=2-1, to=2-2]
	\arrow[from=2-2, to=2-3]
	\arrow[from=1-2, to=1-3]
	\arrow["\simeq", from=1-3, to=2-3]
	\arrow["\simeq", from=1-2, to=2-2]
	\arrow[from=1-1, to=1-2]
	\arrow[from=1-3, to=1-4]
	\arrow[from=2-3, to=2-4]
	\arrow[Rightarrow, no head, from=1-4, to=2-4]
	\end{tikzcd}\]
	which shows that it suffices to check that the natural pairing 
	\[
	R\hat{\Gamma}(G_\RR,M^\vee)\otimes^L R\hat{\Gamma}(G_\RR,M) \to R\hat{\Gamma}(G_\RR,\ZZ) \to \hat{H}^2(G_\RR,\ZZ)[-2] \to \QQ/\ZZ[-2]
	\]
	is perfect. In a similar way, the non-zero class $u\in H^2(G_\RR,\ZZ)$ corresponds to a map $\ZZ\to\ZZ[2]$ whose fiber comes from induced modules\footnote{This is just a reformulation of the $2$-periodicity of the Tate cohomology of cyclic groups}, and we also get an isomorphism of pairings
	\[\begin{tikzcd}
	{R\hat{\Gamma}(G_\RR,M^\vee)\otimes^L R\hat{\Gamma}(G_\RR,M)} & {R\hat{\Gamma}(G_\RR,\ZZ)} & {\hat{H}^0(G_\RR,\ZZ)[0]} & {\QQ/\ZZ[0]} \\
	{R\hat{\Gamma}(G_\RR,M^\vee)\otimes^L R\hat{\Gamma}(G_\RR,M)[2]} & {R\hat{\Gamma}(G_\RR,\ZZ)[2]} & {\hat{H}^{2}(G_\RR,\ZZ)[0]} & {\QQ/\ZZ[0]}
	\arrow[Rightarrow, no head, from=1-4, to=2-4]
	\arrow["\simeq", from=1-1, to=2-1]
	\arrow["\simeq", from=1-2, to=2-2]
	\arrow["\simeq", from=1-3, to=2-3]
	\arrow[from=1-1, to=1-2]
	\arrow[from=2-1, to=2-2]
	\arrow[from=2-2, to=2-3]
	\arrow[from=1-2, to=1-3]
	\arrow[from=1-3, to=1-4]
	\arrow[from=2-3, to=2-4]
	\end{tikzcd}\]
	so equivalently it suffices to check that the natural pairing 
	\[
	R\hat{\Gamma}(G_\RR,M^\vee)\otimes^L R\hat{\Gamma}(G_\RR,M) \to R\hat{\Gamma}(G_\RR,\ZZ) \to \hat{H}^0(G_\RR,\ZZ)[0] \to \QQ/\ZZ[0]
	\]
	is perfect; this is \cref{duality_tate_finite_groups}.
	\end{proof}

	We finish with a study of the behaviour of our corrected compactly supported cohomology with respect to finite dominant morphisms:
	\begin{prop}\label{cohomology_compact_pushforward}
		Let $Y=\Spec(\cal{O}')$ be the spectrum of an order in a number field with a finite dominant morphism $\pi:Y\to X$ and let $F$ be a $\ZZ$-constructible sheaf on $Y$. We have canonical isomorphisms
		\begin{align*}
		R\Gamma_{c,B}(X,(\pi_\ast F)^D)\xrightarrow{\simeq} R\Gamma_{c,B}(Y,F^D)\\
		R\hat{\Gamma}_{c,B}(X,(\pi_\ast F)^D)\xrightarrow{\simeq} R\hat{\Gamma}_{c,B}(Y,F^D)
		\end{align*}
		compatible with the arrows $R\Gamma_{c,B}(X,-) \to R\hat{\Gamma}_{c,B}(X,-)$ resp. $R\Gamma_{c,B}(Y,-) \to R\hat{\Gamma}_{c,B}(Y,-)$.
	\end{prop}
	\begin{proof}
		We prove it for $R\Gamma_{c,B}(X,(-)^D)$. The functor $\pi_\ast$ is exact and we have $R\pi^! \ZZ^c_X=R\pi^!\GG_X[1]=\GG_Y[1]=\ZZ^c_Y$ by the finite base change theorem, hence $(\pi_\ast F)^D=\pi_\ast (F^D)$. Denote by $\pi':\mathrm{Sh}(G_\RR,Y(\CC))\to \mathrm{Sh}(G_\RR,X(\CC))$ the morphism of topoi induced by $\pi$, and $\alpha'$ the morphism $\mathrm{Sh}(G_\RR,Y(\CC)) \to \mathrm{Sh}(Y_{et})$. Consider the commutative square
		\[\begin{tikzcd}[ampersand replacement=\&]
		{\mathrm{Sh}(G_\RR,Y(\CC))} \& { \mathrm{Sh}(Y_{et})} \\
		{\mathrm{Sh}(G_\RR,X(\CC))} \& { \mathrm{Sh}(X_{et})}
		\arrow["\pi"', from=1-2, to=2-2]
		\arrow["\alpha", from=2-1, to=2-2]
		\arrow["{\pi'}", from=1-1, to=2-1]
		\arrow["{\alpha'}"', from=1-1, to=1-2]
		\end{tikzcd}\]
		It induces a canonical map $\alpha^\ast \pi_\ast F \to \pi'_\ast \alpha'^\ast F$. We claim that it is an isomorphism; indeed it suffices to check it on points of $X(\CC)$, and then it follows from the computation of the stalks of a finite morphism in the étale case and in the topological case\footnote{In the topological case, by finite morphism we mean a universally closed separated continuous map with finite discrete fibers; we then use \cite[\href{https://stacks.math.columbia.edu/tag/09V4}{Tag 09V4}]{stacks-project}}.
		
		Since $X(\CC)$ and $Y(\CC)$ are finite discrete, we have $\pi'^!=\pi'^\ast$, so $R\pi'^!=\pi'^\ast$, $R\pi'^\ast 2i\pi\ZZ=2i\pi\ZZ$ and $R\pi'^\ast \CC^\times=\CC^\times$. The counit $\pi'_\ast\pi'^\ast \to \mathrm{id}$ is given on stalks by the sum map. Denote $\pi_x$ the base change of $\pi$ to a point $x\in X$ and $g:\eta \to X$, $g':\eta'\to Y$, $i_x:x\to X$, $i_y:y\to Y$ the inclusion of the generic points, resp. of closed points, of $X$ and $Y$. The counit $\pi_\ast \ZZ^c_Y \to \ZZ^c_X$ is given by the morphism of complexes (with left term in degree $-1$)
		\[\begin{tikzcd}[ampersand replacement=\&]
		{\pi_\ast g'_\ast \GG_m=g_\ast \pi_{\eta,\ast}\GG_m} \& {\bigoplus_{y\in Y_0} \pi_\ast i_{y,\ast \ZZ}=\bigoplus_{x\in X_0}i_{x,\ast} \left(\pi_{x,\ast} \ZZ\right)} \\
		{g_\ast\GG_m} \& {\bigoplus_{x\in X_0}i_{x,\ast} \ZZ}
		\arrow[from=1-1, to=2-1]
		\arrow["{\sum \mathrm{ord}_y}"', from=1-1, to=1-2]
		\arrow[from=1-2, to=2-2]
		\arrow["{\sum\mathrm{ord}_x}", from=2-1, to=2-2]
		\end{tikzcd}\]
		where the left arrow is obtained by applying $g_\ast$ to the counit of the adjunction $\pi_{\eta,\ast}\dashv \pi_\eta^\ast$, which is simply the sum map. Thus we have an identification $\alpha^\ast(\pi_\ast \ZZ^c_Y \to \ZZ^c_X)=\pi'_\ast \overline{\QQ}^\times[1] \to  \overline{\QQ}^\times[1]$.\footnote{That is, the pullback of the counit is the counit between the pullbacks} Combining this with the equality $\alpha^\ast \pi_\ast F = \pi'_\ast \alpha'^\ast F$ and the identification $R\pi'^!(\overline{\QQ}^\times \to 2i\pi\ZZ[2])=\overline{\QQ}^\times \to 2i\pi\ZZ [1]$, we obtain a commutative diagram\footnote{The commutation can be seen from the observations made by writing the adjonction arrows as the composition of applying the adjoint functor and postcomposing with the counit}:
		\[\begin{tikzcd}[ampersand replacement=\&]
		{R\Hom_Y(F,\ZZ^c_Y)} \& {R\Hom_{G_\RR,Y(\CC)}(\alpha^\ast F,\overline{\QQ}^\times[1])} \& {R\Hom_{G_\RR,Y(\CC)}(\alpha^\ast F,2i\pi\ZZ[2])} \\
		{R\Hom_X(\pi_\ast F,\ZZ^c_X)} \& {R\Hom_{G_\RR,X(\CC)}(\pi'_\ast \alpha^\ast F,\overline{\QQ}^\times[1])} \& {R\Hom_{G_\RR,X(\CC)}(\pi'_\ast \alpha^\ast F,2i\pi\ZZ[2])}
		\arrow["{\alpha'^\ast}", from=1-1, to=1-2]
		\arrow["{\alpha^\ast}"', from=2-1, to=2-2]
		\arrow["\simeq"', from=1-1, to=2-1]
		\arrow["\simeq", from=1-3, to=2-3]
		\arrow[from=1-2, to=1-3]
		\arrow[from=2-2, to=2-3]
		\arrow["\simeq", from=1-2, to=2-2]
		\end{tikzcd}\]
		 hence we obtain in the following diagram of fiber sequences the induced arrow, which is an isomorphism:
		\[\begin{tikzcd}
		{R\Gamma_{c,B}(Y,(-)^D)} & {R\Gamma(Y,(-)^D)} & {R\Gamma_{G_\RR}(Y(\CC),(-)_\CC^\vee(1)[2])} \\
		{R\Gamma_{c,B}(X,(\pi_\ast -)^D)} & {R\Gamma(X,(\pi_\ast -)^D)} & {R\Gamma_{G_\RR}(X(\CC),(\pi_\ast -)_\CC^\vee(1)[2])}
		\arrow["\simeq", dashed, from=1-1, to=2-1]
		\arrow[from=1-1, to=1-2]
		\arrow[from=2-1, to=2-2]
		\arrow[from=1-2, to=1-3]
		\arrow[from=2-2, to=2-3]
		\arrow["\simeq", from=1-2, to=2-2]
		\arrow["\simeq", from=1-3, to=2-3]
		\end{tikzcd}\]
	\end{proof}

	\begin{prop}\label{compatibility_pairings_finite}
		In the same setting, there is a commutative diagram of pairings:
		\[\begin{tikzcd}[ampersand replacement=\&]
		{R\Gamma(Y, F)\otimes^L R\hat{\Gamma}_{c,B}(Y,F^D)} \& {R\hat{\Gamma}_{c,B}(Y,\ZZ^D)} \\
		{R\Gamma(X,\pi_\ast F)\otimes^L R\hat{\Gamma}_{c,B}(X,(\pi_\ast F)^D)} \& {R\hat{\Gamma}_{c,B}(X,\ZZ^D)}
		\arrow["\simeq", from=1-1, to=2-1]
		\arrow[from=2-1, to=2-2]
		\arrow[from=1-1, to=1-2]
		\arrow[from=1-2, to=2-2]
		\end{tikzcd}\]
		where the arrow $R\hat{\Gamma}_{c,B}(Y,\ZZ^D) \to R\hat{\Gamma}_{c,B}(X,\ZZ^D)$ is induced from the map $R\Gamma(Y,F^D)\to R\Gamma(X,F^D)$ (coming from the counit $\varepsilon:\pi_\ast R\pi^!\ZZ^c_X=\pi_\ast \ZZ^D \to \ZZ^D$) and from the map $R\Gamma_{G_\RR}(Y(\CC),2i\pi\ZZ[2])\to R\Gamma_{G_\RR}(X(\CC),2i\pi\ZZ[2])$.
	\end{prop}
	\begin{proof}
		By compatibility of the pairings for $R\Gamma_{G_\RR}(X(\CC),-)$ and $R\hat{\Gamma}_{G_\RR}(X(\CC),-)$, we can reduce to checking the statement with $R\Gamma_{c,B}$ instead of $R\hat{\Gamma}_{c,B}$. This will follow formally (by taking the fiber) if we prove that the following cube is commutative:
		\begin{equation}\label{compatibility_finite_0}
		\begin{tikzpicture}[baseline= (a).base]
		\node[scale=.85] (a) at (0,0){\begin{tikzcd}[ampersand replacement=\&,sep=tiny]
		{R\Gamma(Y,F)\otimes R\Hom_Y(F,\ZZ^c_Y)} \& {} \& {R\Gamma(Y,F)\otimes R\Hom_{G_\RR,Y(\CC)}(\alpha^\ast F,2i\pi\ZZ[2])} \\
		{} \& {R\Gamma(Y,\ZZ^c_Y)} \&\& {R\Gamma_{G_\RR}(Y(\CC),2i\pi\ZZ[2])} \\
		{R\Gamma(X,\pi_\ast F)\otimes R\Hom_X(\pi_\ast F,\ZZ^c_X)} \&\& {R\Gamma(X,\pi_\ast F)\otimes R\Hom_{G_\RR,Y(\CC)}(\pi'_\ast \alpha^\ast F,2i\pi\ZZ[2])} \\
		\& {R\Gamma(X,\ZZ^c_X)} \&\& {R\Gamma_{G_\RR}(X(\CC),2i\pi\ZZ[2])}
		\arrow[from=1-1, to=2-2]
		\arrow[from=1-1, to=1-3]
		\arrow[from=3-1, to=3-3]
		\arrow["\simeq"{description}, shift right=5, from=1-3, to=3-3]
		\arrow["\simeq"{description}, shift right=5, from=1-1, to=3-1]
		\arrow[from=2-2, to=4-2]
		\arrow[from=2-2, to=2-4]
		\arrow[from=4-2, to=4-4]
		\arrow[from=2-4, to=4-4]
		\arrow[from=3-1, to=4-2]
		\arrow["\simeq"{description}, shift left=5, from=1-1, to=3-1]
		\arrow["\simeq"{description}, shift left=5, from=1-3, to=3-3]
		\arrow[from=1-3, to=2-4]
		\end{tikzcd}};
	\end{tikzpicture}\end{equation}
	
The lax monoidality of $\pi_\ast$, $\pi'_\ast$, $\alpha^\ast$, $\alpha'^\ast$ induce maps
\begin{align*}
& c_\alpha : \alpha^\ast R\cal{H}om_{X}(-,-)\to  R\cal{H}om_{G_\RR,X(\CC)}(\alpha^\ast -,\alpha^\ast -)\\
& c_{\alpha'} :  \alpha'^\ast R\cal{H}om_{Y}(-,-)\to  R\cal{H}om_{G_\RR,Y(\CC)}(\alpha'^\ast -,\alpha'^\ast -)\\
&c_{\pi} : \pi_\ast R\cal{H}om_{Y}(-,-)\to  R\cal{H}om_{X}(\pi_\ast -,\pi_\ast -)\\
&c_{\pi'} : \pi'_\ast R\cal{H}om_{G_\RR,Y(\CC)}(-,-)\to  R\cal{H}om_{G_\RR,X(\CC)}(\pi'_\ast -,\pi'_\ast -)
\end{align*}
The following diagram commutes by naturality; the composite rectangle is given by the adjunction isomorphism $\pi_\ast F^D\simeq (\pi_\ast F)^D$:
\begin{equation}\label{compatibility_finite_1}
\begin{tikzcd}
	{\pi_\ast F^D} & {R\alpha_\ast\alpha^\ast \pi_\ast F^D=R\alpha_\ast \pi'_\ast \alpha'^\ast F^D} & {} \\
	{R\cal{H}om_X(\pi_\ast F,\pi_\ast\ZZ^c_Y)} & {R\alpha_\ast\alpha^\ast R\cal{H}om_X(\pi_\ast F,\pi_\ast\ZZ^c_Y)} & {} \\
	{(\pi_\ast F)^D} & {R\alpha_\ast\alpha^\ast (\pi_\ast F)^D}
	\arrow["{\eta_\alpha}", from=1-1, to=1-2]
	\arrow["{c_\pi}"', from=1-1, to=2-1]
	\arrow["{R\alpha_\ast \alpha^\ast(c_\pi)}", from=1-2, to=2-2]
	\arrow["{\eta_\alpha}", from=2-1, to=2-2]
	\arrow["{\varepsilon_\ast}"', from=2-1, to=3-1]
	\arrow["{R\alpha_\ast \alpha^\ast(\varepsilon_\ast)}", from=2-2, to=3-2]
	\arrow["{\eta_\alpha}", from=3-1, to=3-2]
\end{tikzcd}
\end{equation}

We have $\pi_\ast \alpha^\ast = \alpha'^\ast \pi'_\ast$ so the following diagram commutes formally for any sheaf $F$ and $G$:
\[\begin{tikzcd}[ampersand replacement=\&]
	\& {\pi'_\ast\alpha'^\ast R\cal{H}om_Y(F,G)} \\
	{\alpha^\ast \pi_\ast R\cal{H}om_Y(F,G)} \&\& {\pi'_\ast R\cal{H}om_{G_\RR,Y(\CC)}(\alpha'^\ast F,\alpha'^\ast G)} \& {} \\
	{\alpha^\ast R\cal{H}om_X(\pi_\ast F,\pi_\ast G)} \& {} \& {R\cal{H}om_X(\pi'_\ast \alpha'^\ast F,\pi'_\ast \alpha'^\ast G)} \\
	\& {R\cal{H}om_X(\alpha^\ast\pi_\ast F,\alpha^\ast\pi_\ast G)} \& {}
	\arrow[Rightarrow, no head, from=2-1, to=1-2]
	\arrow["{\pi'_\ast(c_{\alpha'})}", from=1-2, to=2-3]
	\arrow["{\alpha_\ast(c_\pi)}"', from=2-1, to=3-1]
	\arrow["{c_\alpha}"', from=3-1, to=4-2]
	\arrow[Rightarrow, no head, from=4-2, to=3-3]
	\arrow["{c_{\pi'}}", from=2-3, to=3-3]
\end{tikzcd}\]
Thus the top square in the following diagram commutes:
\begin{equation}\label{compatibility_finite_2}
\begin{tikzcd}
	{} & {\alpha^\ast \pi_\ast F^D=R\alpha_\ast \pi'_\ast \alpha'^\ast F^D} & { \pi'_\ast R\cal{H}om_{G_\RR,Y(\CC)} (\alpha^\ast F, \overline{\QQ}^\times[1])} \\
	& {\alpha^\ast R\cal{H}om_X(\pi_\ast F,\pi_\ast\ZZ^c_Y)} & {R\cal{H}om_{G_\RR,X(\CC)}(\pi'_\ast \alpha^\ast F,\pi'_\ast \overline{\QQ}^\times[1])} \\
	& {\alpha^\ast R\cal{H}om_X(\pi_\ast F,\ZZ^c_X)} & {R\cal{H}om_{G_\RR,X(\CC)}(\pi'_\ast \alpha^\ast F,\overline{\QQ}^\times[1])}
	\arrow["{\alpha^\ast(c_\pi)}", from=1-2, to=2-2]
	\arrow["{\pi'_\ast(c_{\alpha'})}", from=1-2, to=1-3]
	\arrow["{c_{\alpha}}", from=2-2, to=2-3]
	\arrow["{ \alpha^\ast(\varepsilon_\ast)}", from=2-2, to=3-2]
	\arrow["{c_{\pi'}}"', from=1-3, to=2-3]
	\arrow["{\varepsilon_\ast}"', from=2-3, to=3-3]
	\arrow["{c_\alpha}", from=3-2, to=3-3]
\end{tikzcd}
\end{equation}
The bottom square also commutes, because $\alpha^\ast(\pi_\ast \ZZ^c_Y \to \ZZ^c_X)=\pi'_\ast \overline{\QQ}^\times [1] \to \overline{\QQ}^\times [1]$. The commutative square
\[\begin{tikzcd}
	{\pi'_\ast  \overline{\QQ}^\times[1]} & {\pi'_\ast 2i\pi\ZZ[2]} \\
	{ \overline{\QQ}^\times[1]} & {2i\pi\ZZ[2]}
	\arrow[from=2-1, to=2-2]
	\arrow[from=1-1, to=2-1]
	\arrow[from=1-1, to=1-2]
	\arrow[from=1-2, to=2-2]
\end{tikzcd}\]
implies that the following diagram is commutative:
\begin{equation}\label{compatibility_finite_3}
\begin{tikzcd}[ampersand replacement=\&]
	{} \&\& { \pi'_\ast R\cal{H}om_{G_\RR,Y(\CC)} (\alpha^\ast F, \overline{\QQ}^\times[1])} \& { \pi'_\ast R\cal{H}om_{G_\RR,Y(\CC)} (\alpha^\ast F, 2i\pi\ZZ[2])} \\
	\&\& {R\cal{H}om_{G_\RR,X(\CC)}(\pi'_\ast \alpha^\ast F,\pi'_\ast \overline{\QQ}^\times[1])} \& {R\cal{H}om_{G_\RR,X(\CC)}(\pi'_\ast \alpha^\ast F,\pi'_\ast 2i\pi\ZZ[2])} \\
	\&\& {R\cal{H}om_{G_\RR,X(\CC)}(\pi'_\ast \alpha^\ast F,\overline{\QQ}^\times[1])} \& {R\cal{H}om_{G_\RR,X(\CC)}(\pi'_\ast \alpha^\ast F,2i\pi\ZZ[2])}
	\arrow["{c_{\pi'}}"', from=1-3, to=2-3]
	\arrow["{\varepsilon_\ast}"', from=2-3, to=3-3]
	\arrow["{\varepsilon_\ast}", from=2-4, to=3-4]
	\arrow["{c_{\pi'}}", from=1-4, to=2-4]
	\arrow[from=1-3, to=1-4]
	\arrow[from=2-3, to=2-4]
	\arrow[from=3-3, to=3-4]
\end{tikzcd}
\end{equation}
Apply $R\alpha_\ast$ to diagrams \cref{compatibility_finite_2,compatibility_finite_3} and paste them next to diagram \cref{compatibility_finite_1} to obtain the following diagram:
\[\begin{tikzcd}[ampersand replacement=\&]
	{\pi_\ast F^D} \& {R\alpha_\ast \pi'_\ast(\alpha^\ast F)^\vee(1)[2]} \\
	\& {R\cal{H}om_X(\pi_\ast F,\pi_\ast\ZZ^c_Y)} \& {R\cal{H}om_{X}(\pi_\ast F,R\alpha_\ast\pi'_\ast 2i\pi\ZZ[2])} \\
	\& {R\cal{H}om_X(\pi_\ast F,\ZZ^c_X)} \& {R\cal{H}om_{X}(\pi_\ast F,R\alpha_\ast 2i\pi\ZZ[2])}
	\arrow["\simeq"{description, pos=0.2}, dashed, from=1-1, to=3-2]
	\arrow[from=1-1, to=2-2]
	\arrow[from=2-2, to=3-2]
	\arrow[from=2-2, to=2-3]
	\arrow[from=3-2, to=3-3]
	\arrow["\simeq"{description, pos=0.2}, dashed, from=1-2, to=3-3]
	\arrow[from=2-3, to=3-3]
	\arrow[from=1-2, to=2-3]
	\arrow[from=1-1, to=1-2]
\end{tikzcd}\]
Here we have rewritten the terms on the right using $R\alpha_\ast R\cal{H}om_{G_\RR,X(\CC)}(\alpha^\ast-,-)=R\cal{H}om(-,R\alpha_\ast -)$. The back and side faces are obtained by composing the top and front faces; by properties of adjunctions, the dotted maps are the adjunction isomorphisms for $\pi_\ast \dashv R\pi^!$ resp. $\pi'_\ast \dashv R\pi'^!$. By the $\otimes\dashv\Hom$-adjunction, the above diagram is equivalent to the following cube:
\[\begin{tikzcd}[ampersand replacement=\&]
	{\pi_\ast F \otimes\pi_\ast F^D} \& {} \& {\pi_\ast F \otimes R\alpha'_\ast (\alpha^\ast F)^\vee(1)[2]} \\
	{} \& {\pi_\ast \ZZ^c_Y} \&\& {R\alpha_\ast \pi'_\ast 2i\pi\ZZ[2]} \\
	{\pi_\ast F \otimes (\pi_\ast F)^D} \& {} \& {\pi_\ast F \otimes R\alpha_\ast (\pi'_\ast \alpha^\ast F)^\vee(1)[2]} \\
	\& {\ZZ^c_X} \&\& {R\alpha_\ast 2i\pi\ZZ[2]}
	\arrow[from=1-1, to=1-3]
	\arrow["\simeq"', from=1-1, to=3-1]
	\arrow[from=3-1, to=4-2]
	\arrow[from=2-2, to=4-2]
	\arrow[from=1-1, to=2-2]
	\arrow[from=2-2, to=2-4]
	\arrow[from=1-3, to=2-4]
	\arrow[from=4-2, to=4-4]
	\arrow["\simeq"'{pos=0.2}, from=1-3, to=3-3]
	\arrow[from=3-3, to=4-4]
	\arrow[from=3-1, to=3-3]
	\arrow[from=2-4, to=4-4]
\end{tikzcd}\]
 Finally, applying the lax monoidal functor $R\Gamma(X,-)$ recovers the sought-after cube.
	\end{proof}
\subsection{Computations}

\begin{prop}\label{cohomology_comparison}
	Let $F$ be a $\ZZ$-constructible sheaf. We have $H^i_{c,B}(X,F^D)\simeq \hat{H}^i_{c,B}(X,F^D)$ for $i\geq 1$ and $H^0_c(X,F^D)\to \hat{H}^0_c(X,F^D)$ is surjective. If $\alpha^\ast F$ is torsionfree, we have $H^i_{c,B}(X,F^D)\simeq \hat{H}^i_{c,B}(X,F^D)$ for $i\geq 0$ and $H^{-1}_c(X,F^D)\to \hat{H}^{-1}_c(X,F^D)$ is surjective.
\end{prop}
\begin{proof}
	Let $T(F^D)$ be the cofiber of $R\Gamma_{c,B}(X,F^D)\to R\hat{\Gamma}_{c,B}(X,F^D)$. Let $s$ be the generator of $G_\RR$ and denote $N=1+s$ the norm map. By the stable $3\times 3$ lemma (\cite[Lemma 2.2]{AMorin21}), we find that
	\begin{align*}
	T(F^D)=\mathrm{fib}\left(R\Gamma_{G_\RR}(X(\CC),(\alpha^\ast F)^\vee(1)[2])\xrightarrow{N} R\hat{\Gamma}_{G_\RR}(X(\CC),(\alpha^\ast F)^\vee(1)[2])\right)&=\ZZ\otimes^L_{\ZZ[G_\RR]}R\Gamma(X(\CC),(\alpha^\ast F)^\vee(1)[2])\\
		&=\prod_{v~\text{archimedean}} \ZZ\otimes^L_{\ZZ[G_{K_v}]}F_v^\vee(1)[2]
	\end{align*}
	computes homology at the archimedean places. Let $M$ be a $G_\RR$-module of finite type. Then $M^\vee(1)[2]=R\Hom_\ZZ(M,2i\pi\ZZ)[2]$ is concentrated in degree $-2$ if $M$ is torsion-free and in degree $[-2,-1]$ in general. It follows that $\ZZ\otimes^L_{\ZZ[G_\RR]}M^\vee(1)[2]$ is concentrated in degree $\leq -2$ if $M$ is torsion-free and in degree $\leq -1$ in general.
\end{proof}

\begin{prop}\label{cohomology_compact_support_tate}
	Let $F$ be a $\ZZ$-constructible sheaf. Then $R\hat{\Gamma}_{c,B}(X,F^D)$ is concentrated in degree $\leq 2$. If $F$ is constructible, the complex has finite cohomology groups. In general, we have
	\[\hat{H}^i_{c,B}(X,F^D)=\left\{
	\begin{array}{ll}
	\text{finite} & i\leq -2\\
	\text{finite type} & i=-1,0\\
	\text{torsion of cofinite type} & i=1,2\\
	\end{array}\right.
	\]
\end{prop}
\begin{proof}
	The vanishing in degree $>2$ comes from \cref{twisted_AV}. The remaining claims follow from the defining long exact cohomology sequence and \cite[II.3.6]{ADT}, because $R\hat{\Gamma}_{G_\RR}(X(\CC),(\alpha^\ast F)^\vee(1)[2])$ is a complex with finite $2$-torsion cohomology groups.
\end{proof}

\begin{prop}\label{cohomology_compact_support}
	Let $F$ be a $\ZZ$-constructible sheaf. Then $R\Gamma_{c,B}(X,F^D)$ is concentrated in degree $[-1,2]$. If $F$ is constructible, the complex is perfect. In general, we have
	\[H^i_{c,B}(X,F^D)=\left\{
	\begin{array}{ll}
	\text{finite type} & i=-1,0\\
	\text{torsion of cofinite type} & i=1,2\\
	\end{array}\right.
	\]
\end{prop}

\begin{proof}
	By \cref{cohomology_comparison} and the previous proposition, we have $H^i_{c,B}(X,F^D)=\hat{H}^i_{c,B}(X,F^D)=0$ for $i>2$ and $H^i_{c,B}(X,F^D)=\hat{H}^i_{c,B}(X,F^D)$ is torsion of cofinite type for $i=1,2$; finally for $i=-1,0$ the difference between $H^i_{c,B}(X,F^D)$ and $\hat{H}^i_{c,B}(X,F^D)$ is given by a group of finite type so $H^i_{c,B}(X,F^D)$ is finite type by the previous proposition. The vanishing in degree $<-1$ is clear.
\end{proof}

We now compute some special cases.
\begin{prop}\label{cohomology_Z}
	Suppose that $X=\Spec(\cal{O}_K)$ is \emph{regular}. We have
	\[H^i(X,\ZZ)=\left\{
	\begin{array}{ll}
	\ZZ & i=0\\
	0 & i=1\\
	\text{torsion} & i>1
	\end{array}\right.
	\]
	
	Let $j:U \to X$ be an open subscheme of $X$ with $U\neq X$. We have 
	\[H^i(X,j_!\ZZ)=\left\{
	\begin{array}{ll}
	0 & i=0\\
	\left(\prod_{v\in X\backslash U}\ZZ\right)/\ZZ & i=1\\
	\text{torsion} &i>1
	\end{array}\right.
	\]
\end{prop}
\begin{proof}
	The result for $i>1$ follows from \cite[II.2.10]{ADT}. Since $X$ is normal, $\pi_1^{proet}(X)=\pi_1^{et}(X)$ is profinite so $H^1(X,\ZZ)=\Hom_{cont}(\pi_1^{et}(X),\ZZ)=0$.
\end{proof}
\begin{prop}\label{cohomology_compact_support_ZD}
	Suppose that $X=\Spec(\cal{O}_K)$ is \emph{regular}.
	We have
	\[H^i_{c,B}(X,\ZZ^D)=\left\{
	\begin{array}{ll}
	\text{finite type of rank~} [K:\QQ]-1 & i=-1\\
	\text{finite} &i=0\\
	0 & i=1\\
	\QQ/\ZZ &i=2\\
	\end{array}\right.
	\]
	Moreover we have an exact sequence
	\[
	0 \to \ZZ^{r_2}\to H^{-1}_{c,B}(X,\ZZ^D) \to \cal{O}_K^\times \to (\ZZ/2\ZZ)^{r_1} \to H^0_{c,B}(X,\ZZ^D) \to \mathrm{Pic}(X) \to 0
	\]
	and $H^0_{c,B}(X,\ZZ^D)$ is the narrow ideal class group $\mathrm{Pic}^+(X)$.
\end{prop}

\begin{proof}
	The exact sequence comes from the long exact cohomology sequence of the defining fiber sequence. By \cref{cohomology_comparison}, we have
	\[
	\tau^{\geq 0} R\Gamma_{c,B}(X,\ZZ^D) \xrightarrow{\simeq} \tau^{\geq 0} R\hat{\Gamma}_{c,B}(X,\ZZ^D) \simeq \tau^{\geq 0}R\hat{\Gamma}_c(X,\GG_m[1])
	\]
	hence the result for $i\geq 0$ follows from \cite[II.2.6, II.2.8(a)]{ADT}. 
\end{proof}

\begin{rmk}
	\begin{itemize}
		\item[]
		\item We have $H^1_{G_\RR}(X(\CC),2i\pi\ZZ)\simeq \hat{H}^0_{G_\RR}(X(\CC),\CC^\times)=(\RR^\times/\RR^\times_+)^{r_1}$. Since the map $H^0(X,\GG_m) \to H^1_{G_\RR}(X(\CC),2i\pi\ZZ)$ factors through $H^0_{G_\RR}(X(\CC),\CC^\times)$ by definition, it is given by
		\[
		\cal{O}_K^\times \xrightarrow{(\mathrm{sign}_v)_v} \bigoplus_{v~\text{real}}\ZZ/2\ZZ
		\]
		Its kernel is $\cal{O}_{K,+}^{\times}$, the group of totally positive units; it has same rank as $\cal{O}_K^\times$. Thus we have a short exact sequence
		\[
		0 \to \ZZ^{r_2}\to H^{-1}_{c,B}(X,\ZZ^D)\to \cal{O}_{K,+}^{\times} \to 0
		\]
		\item We can also recover the result for $i\geq 1$ by Artin-Verdier duality (\cref{twisted_AV}) and \cref{cohomology_comparison} since $H^i(X,\ZZ)=0,\ZZ,0$ for $i<0$, $i=0$, $i=1$. On the other hand, from the result for $i=0$ we recover $H^2(X,\ZZ)\simeq \mathrm{Pic}^+(X)^\ast$.
		\item If $X$ is singular, the exact sequence becomes
		\[
		0 \to \ZZ^{r_2}\to H^{-1}_{c,B}(X,\ZZ^D) \to \mathrm{CH}_0(X,1) \to (\ZZ/2\ZZ)^{r_1}\to H^0_{c,B}(X,\ZZ^D) \to \mathrm{CH}_0(X) \to 0
		\]
		and $H^1_{c,B}(X,\ZZ^D)=H^1(X,\ZZ)^\ast$ can be non-zero and non-finite, depending on the singularities of $X$. Moreover, $H^0_{c,B}(X,\ZZ^D)$ can be seen as a "narrow Chow group".
	\end{itemize}
\end{rmk}

\begin{prop}\label{cohomology_compact_support_jZD}
	Suppose that $X=\Spec(\cal{O}_K)$ is \emph{regular} and let $j:U\to X$ be an open subscheme with $U\neq X$, say $U=\Spec(\cal{O}_{K,S})$ with $S$ a set of places containing the archimedean places and at least one finite place. Denote $S_f$ the set of finite places in $S$ and $s_f$ its cardinality.
	We have
	\[H^i_{c,B}(X,(j_!\ZZ)^D)=\left\{
	\begin{array}{ll}
	\text{finite type of rank~} [K:\QQ]-1 & i=-1\\
	\text{finite} &i=0\\
	\Ker\left(\bigoplus_{v\in S_f}\mathrm{Br}(K_v)\xrightarrow{\Sigma} \QQ/\ZZ\right)\simeq (\QQ/\ZZ)^{s_f-1}  & i=1\\
	0 & i=2
	\end{array}\right.
	\]
	Moreover, we have an exact sequence
	\[
	0 \to \ZZ^{r_2} \to H^{-1}_{c,B}(X,(j_!\ZZ)^D) \to \cal{O}_{K,S}^\times \to (\ZZ/2\ZZ)^{r_1} \to H^0_{c,B}(X,(j_!\ZZ)^D) \to \mathrm{Pic}(U) \to 0
	\]
	and $H^0_{c,B}(X,(j_!\ZZ)^D)$ is the narrow $S$-class group $\Pic^+(U)$.
\end{prop}

\begin{proof}
 	By \cref{cohomology_comparison} and Artin-Verdier duality we obtain the result for $i=1,2$ from the computation of $H^i(X,j_!\ZZ)$. We have $(j_!\ZZ)^D=Rj_\ast \ZZ^c_U=j_\ast \GG_m[1]$ \cite[II.1.4]{ADT} and $(j_!\ZZ)_\CC^\vee(1)=2i\pi\ZZ$ so the exact sequence comes from the long exact cohomology sequence associated to the defining fiber sequence. 
	
	We have $R\hat{\Gamma}_{c,B}(X,(j_!\ZZ)^D)\xleftarrow{\simeq} R\hat{\Gamma}_c(X,j_\ast \GG_m[1])$ by \cref{comparison_Tate_compact_cohomology}. The divisor short exact sequence
	\[
	0 \to j_\ast \GG_m \to g_\ast \GG_m \to \bigoplus_{v\in U_0} i_{v,\ast} \ZZ \to 0
	\]
	and \cref{cohomology_comparison} gives, as in \cite[II.2.8(a)]{ADT}, the identification of $H^0_{c,B}(X,(j_!\ZZ)^D)\simeq \hat{H}^0_{c,B}(X,(j_!\ZZ)^D)\simeq \hat{H}^0_{c}(X,j_\ast \GG_m[1])$ with the narrow $S$-class group.
\end{proof}

\begin{rmk}
	\begin{itemize}
		\item[]
		\item The map $H^0(U,\GG_m) \to H^1_{G_\RR}(X(\CC),2i\pi\ZZ)$ is given by
		\[
		\cal{O}_{K,S}^\times \xrightarrow{(\mathrm{sign}_v)_{v~\text{real}}} \bigoplus_{v~\text{real}}\ZZ/2\ZZ
		\]
		Its kernel is $\cal{O}_{K,S,+}^{\times}$, the group of totally positive $S$-units; it has same rank as $\cal{O}_{K,S}^\times$. Thus there is a short exact sequence
		\[
		0 \to \ZZ^{r_2} \to H^{-1}_c(X,(j_!\ZZ)^D) \to \cal{O}_{K,S}^{\times,+} \to 0
		\]
		\item For $i=1,2$, we can also prove the result directly by considering the following snake diagram:
		\[
		\begin{tikzcd}
		&
		& 
		& 0 \ar{r} \arrow[d, hook]
		& H^{1}_{c,B}(X,(j_!\ZZ))^D \arrow[d, hook]   
		&
		&\\
		& 0 \ar{r}
		& 0 \ar{r} \ar{dd}
		& H^2(U,\GG_m) \arrow[r, equal] \ar{dd}
		& H^2(U,\GG_m) \ar{r} \ar{dd}
		&  0
		& ~\\
		&
		&
		& ~
		&
		&
		\ar[r, phantom, ""{coordinate, name=Y}]
		&~\\
		~& \ar[l, phantom, ""{coordinate, name=Z}] 0 \ar{r}
		& \prod_{v\in S_f}\mathrm{Br}(K_v) \ar{r} \arrow[d, equal]
		& \prod_{v\in S}\mathrm{Br}(K_v) \ar{r} \arrow[d, two heads,"\Sigma"]
		& \prod_{v~\text{archimedean}}\mathrm{Br}(K_v) \ar{r} \arrow[d, two heads,]
		& 0
		& \\
		&
		& \ar[from=uuuurr, crossing over, rounded corners,
		to path=
		{ -- ([xshift=2ex]\tikztostart.east)
			-| (Y) [near end]\tikztonodes
			-| (Z) [near end]\tikztonodes
			|- ([xshift=-2ex]\tikztotarget.west)
			-- (\tikztotarget)}
		]\prod_{v\in S_f}\mathrm{Br}(K_v) \ar{r}{\Sigma}
		& \QQ/\ZZ \ar{r}
		& H^2_{c,B}(X,(j_!\ZZ)^D) \ar{r} 
		& 0
		&
		\end{tikzcd}
		\]
		\item We recover from Artin-Vertier duality an identification $H^2(X,j_!\ZZ)\simeq \Pic^+(U)^\ast$.
	\end{itemize}
\end{rmk}

Let $i:x\to X$ be a closed point of $X$ and $M$ a discrete $G_x$-module of finite type. Since $Ri^!\ZZ^c_X=\ZZ[0]$, we have $(i_\ast M)^D=i_\ast M^\vee$, while $(i_\ast M)_\CC=0$. Thus $R\Gamma_{c,B}(X,(i_\ast M)^D)=R\Gamma(G_x,M^\vee)$ and we obtain
\begin{prop}\label{cohomology_compact_support_point}
	We have
	\[H^i_{c,B}(X,(i_\ast M)^D)=\left\{
	\begin{array}{ll}
	0 & i=-1\\
	\text{finite type} & i=0\\
	\text{finite} & i=1\\
	\text{torsion of cofinite type} &i=2\\
	\end{array}\right.
	\]
	If $M$ is finite, $R\Gamma_{c,B}(X,(i_\ast M)^D)$ has finite cohomology groups and $H^0_{c,B}(X,(i_\ast M)^D)=0$.
\end{prop}
\begin{proof}
	If $M$ is finite, we have $M^\vee=M^\ast[-1]$. If $M$ is torsion-free, we have $M^\vee=\Hom_{Ab}(M,\ZZ)$. Thus in both cases we reduce to cohomology of discrete $G_x\simeq \hat{\ZZ}$-modules of finite type, for which the result is known. For arbitrary $M$, we conclude with the short exact sequence $0 \to M_{tor} \to M \to M/{tor} \to 0$.
\end{proof}

Let $D=R\Hom(R\Gamma(X,-),\QQ[-2])$, and denote $(-)^\dagger=\Hom(-,\QQ)$.
\begin{prop}\label{cohomology_dagger}
	Let $F$ be a $\ZZ$-constructible sheaf. Then $D_F$ is concentrated in degree $[1,2]$.
\end{prop}

\begin{proof}
	We have $H^i(D_F)=H^{2-i}(X,F)^\dagger$. We thus have to show that $H^i(X,F)$ is torsion for $i\neq 0,1$. But $H^i(X,F)$ differs from $\hat{H}^i_c(X,F)$ by a finite group since $X$ is proper, and the latter is torsion for $i\neq 0,1$ \cite[II.6.2]{ADT}.
\end{proof}
\section{Weil-étale cohomology with compact support of \texorpdfstring{$F^D$}{FD}}\label{sec:construction_weil_etale}

\subsection{Construction of the Weil-étale complex}

Following \cite{Flach18}, the Weil-étale complex with compact support should be the cone of a map $\alpha_F$ making the following diagram commute
\[\begin{tikzcd}
{R\Hom(R\Gamma(X,F),\QQ[-2])} & {R\Gamma_{c,B}(X,F^D)} \\
& {R\hat{\Gamma}_{c,B}(X,F^D)} \\
& {R\Hom(R\Gamma(X,F),\QQ/\ZZ[-2])}
\arrow[from=1-1, to=3-2]
\arrow[from=1-2, to=2-2]
\arrow[from=2-2, to=3-2]
\arrow["{\alpha_F}"{description}, dashed, from=1-1, to=1-2]
\end{tikzcd}\]
Let $F$, $G$ be two $\ZZ$-constructible sheaves on $X$ and let us compute $\Hom_{D(\ZZ)}(D_G,R\Gamma_{c,B}(X,F^D))$. Recall the Verdier spectral sequence \cite[III.4.6.10]{theseverdier}for $R\Hom_{D(\ZZ)}$:
\begin{equation}\label{Verdier_ss}
E_2^{p,q}=\prod_{i\in \ZZ} \Ext^p_\ZZ(H^i(K),H^{i+q}(L))\Rightarrow \Ext^n_{D(\ZZ)}(K,L)
\end{equation}
 Using the vanishing results of the previous section, the above degenerates to a short exact sequence
\begin{equation}\label{Verdier_ses}
\begin{split}
0 \to \prod_{i=1,2}\Ext^1(H^{2-i}(X,G)^\dagger, H^{i-1}_{c,B}(X,F^D)) &\to \Hom_{D(\ZZ)}(D_G,R\Gamma_{c,B}(X,F^D)) \\
&\to \prod_{i=1,2} \Hom(H^{2-i}(X,G)^\dagger,H^i_{c,B}(X,F^D)) \to 0
\end{split}
\end{equation}

Since $H^1_{c,B}(X,F^D)$ is torsion and $H^0(X,G)^\dagger$ is a $\QQ$-vector space, we have $\Ext^1(H^0(X,G)^\dagger,H^1_{c,B}(X,F^D))=0$. Therefore the left term is $\Ext^1(H^{1}(X,G)^\dagger, H^{0}_{c,B}(X,F^D))$. Similarly to our approach in \cite{AMorin21}, this motivates the following definition:
\begin{defi}
	Let $F$ be a $\ZZ$-constructible sheaf on $X$. We say that
	\begin{itemize}
		\item $F$ is red if $H^0_{c,B}(X,F^D)$ is torsion, hence finite.
		\item $F$ is blue if $H^1(X,F)$ is torsion, hence finite; this happens if and only if $\hat{H}^1_{c,B}(X,F^D)=H^1_{c,B}(X,F^D)$ is finite (\cref{twisted_AV}).
		\item A red-to-blue morphism is a morphism of sheaves $F\to G$ where either $F$ and $G$ are both blue, or are both red, or $F$ is red and $G$ is blue ; a red-to-blue short exact sequence is a short exact sequence with red-to-blue morphisms.
	\end{itemize}
\end{defi}

\begin{rmk}
	\begin{itemize}
		\item[]
		\item If $F$ is constructible, $F$ is blue and red.
		\item The constant sheaf $\ZZ$ is red; this follows from the finiteness of $\mathrm{CH}_0(X)$. \footnote{see the remark after \cref{cohomology_compact_support_ZD} and the remark after \cite[Corollary 7.5]{AMorin21}} If $X$ is regular or more generally unibranch (so that each point has a singleton preimage in the normalization of $X$), the sheaf $\ZZ$ is also blue. 
		\item If $j:U\to X$ is an open inclusion with $U\neq X$, the sheaf $j_!\ZZ$ is red, and it is blue if $X$ is unibranch and $|X\backslash U|=1$.
		\item If $F$ is $\ZZ$-constructible and supported on a closed subscheme, then $F$ is blue. This reduces to the case of a single discrete finite type $\hat{\ZZ}$-module, where it follows from $H^1(\hat{\ZZ},\ZZ)=0$ and the Hochschild-Serre spectral sequence for a normal open subgroup acting trivially.
	\end{itemize}
\end{rmk}

\begin{prop}
	Let $j:U\to X$ be an open immersion such that $U$ is regular and $G$ a locally constant $\ZZ$-constructible sheaf on $U$. Then $j_!G$ is red. 
\end{prop}
\begin{proof}
	Indeed, the result is true if $G$ is torsion so we can suppose that $G$ is torsion-free. Let $g:\eta=\Spec(K) \to U$ be the generic point. Since $G$ is locally constant, we have $G=g_\ast g^\ast G$.\footnote{The regularity hypothesis appears here: in general $g_\ast \ZZ\neq \ZZ$.} By Artin induction for the $G_K$-module $g^\ast G$, we find normal open subgroups $H_k$, $H_l$ of $G_K$, an integer $n\in \NN^\times$ and a finite $G_K$-module $N$ such that we have a short exact sequence
	\[
	0 \to (g^\ast G)^n \oplus \bigoplus \mathrm{ind}^{H_k}_{G_K} \ZZ \to \bigoplus \mathrm{ind}^{H_l}_{G_K} \ZZ \to N \to 0
	\]
	Denote $\pi'_k:V_k\to U,\pi'_l:V_l\to U$ the normalizations of $U$ in $(K^{sep})^{H_k}, (K^{sep})^{H_l}$. By applying $g_\ast$, we find a short exact sequence
	\[
	0 \to G^n \oplus \bigoplus \pi'_{k,\ast} \ZZ \to \bigoplus \pi'_{l,\ast} \ZZ \to Q \to 0
	\]
	where $Q$ is a subsheaf of the constructible sheaf $g_\ast N$ and hence is constructible. Denote by $\pi_k:Y_k\to X$ (resp. $\pi_l:Y_l\to X$) the normalization of $X$ in $(K^{sep})^{H_k}$ (resp. $(K^{sep})^{H_l}$) and $j_k:V_k\to Y_k$ the open inclusion (resp. $j_l:V_l \to Y_l$). The points in $Y_k$ (resp. $Y_l$) above $U$ are exactly the points of $V_k$ (resp. $V_l$) so 
	\[
	\begin{array}{lr}
	j_!\pi'_{k,\ast} \ZZ=\pi_{k,\ast}j_{k,!} \ZZ, & j_!\pi'_{l,\ast} \ZZ=\pi_{l,\ast}j_{l,!} \ZZ.
	\end{array}
	\]
	By \cref{cohomology_compact_pushforward} we conclude with the preceding remark.
\end{proof}

\begin{rmk}
It follows that there are "enough red-and-blues"; by this, we mean that any $\ZZ$-constructible sheaf $F$ sits in a short exact sequence
\[
0 \to R \to F \to B \to 0
\]
with $R$ red and $B$ blue. Indeed, it suffices to take $R=j_!F_{|U}$ and $B=i_\ast i^\ast F$ for $j:U\to X$ a regular open subscheme such that $F_{|U}$ is locally constant and $i:Z\to X$ its closed complement. As in \cite{AMorin21}, this remark will allow us to bootstrap the construction of the Weil-étale Euler characteristic from red or blue sheaves to arbitrary $\ZZ$-constructible sheaves by enforcing multiplicativity.
\end{rmk}

If $F$ is red or $G$ is blue, the left term of the short exact sequence \eqref{Verdier_ses} vanishes and we obtain:
\begin{prop}\label{hom_via_cohomology}
	Suppose that $F$ is red or that $G$ is blue. We have
	\[
		\Hom(D_G,R\Gamma_{c,B}(X,F^D))=\prod_{i=1,2} \Hom(H^{2-i}(X,G)^\dagger,H^i_{c,B}(X,F^D))
	\]
\end{prop}
In particular, suppose that $F$ is red or blue. For $i=1,2$, we have by \cref{twisted_AV} isomorphisms \[H^i_{c,B}(X,G^D) \xrightarrow{\simeq} \hat{H}^i_{c,B}(X,G^D) \xrightarrow{\simeq} H^{2-i}(X,G)^\ast\]
Hence we obtain a canonical element $\alpha_F\in \Hom(D_F,R\Gamma_{c,B}(X,G^D))$, given in cohomology in degree $i=1,2$ by 
\[H^i(\alpha_F):H^{2-i}(X,F)^\dagger \to H^{2-i}(X,F)^\ast \xleftarrow{\simeq} H^i_{c,B}(X,F^D).\]
\begin{defi}
	Let $F$ be red or blue. We define 
	\[
	R\Gamma_{W,c}(X,F^D):=\mathrm{Cone}(\alpha_F)
	\]
	the Weil-étale complex with compact support with coefficients in $F^D$.
\end{defi}

\begin{rmk}
	We want to emphasize that our construction is only done in the homotopy category, as a mapping cone. Nevertheless for red or blue sheaves we will obtain sufficient functoriality properties; this is unusual.
\end{rmk}

\subsection{Computation of Weil-étale cohomology}

\begin{prop}\label{weil-etale_perfect}
	Let $F$ be a red or blue sheaf. Then $R\Gamma_{W,c}(X,F^D)$ is a perfect complex of abelian groups concentrated in degree $[-1,2]$. Moreover, we have
	\[H^i_{W,c}(X,F^D)=\left\{
	\begin{array}{ll}
	H^{-1}_c(X,F^D) & i=-1\\
	(H^0(X,F)_{tor})^\ast & i=2\\
	\end{array}\right.
	\]
	and short exact sequences
	\[
	0 \to H^0_{c,B}(X,F^D) \to H^0_{W,c}(X,F^D) \to \Hom(H^1(X,F),\ZZ) \to 0
	\]
	\[
	0 \to (H^1(X,F)_{tor})^\ast \to H^1_{W,c}(X, F^D) \to \Hom(H^0(X,F),\ZZ) \to 0
	\]
\end{prop}
\begin{proof}
	The groups $H^i(X,F)$ are of finite type for $i=0,1$: they differ from $\hat{H}^i_c(X,F)$ by a finite group since $X$ is proper, and the latter are of finite type for $i=0,1$\cite[II.3.1, II.6.2]{ADT}. The claim then follows from the distinguished triangle
	\[
	R\Hom(R\Gamma(X,F),\QQ[-2]) \to R\Gamma_{c,B}(X,F^D) \to	R\Gamma_{W,c}(X,F^D) \to
	\]
	and \cref{cohomology_compact_support,cohomology_dagger}.
\end{proof}

We compute some special cases, using \cref{cohomology_compact_support_ZD,cohomology_compact_support_jZD}.
\begin{prop}\label{weil_etale_Z}
	Suppose that $X=\Spec(\cal{O}_K)$ is regular. We have
	\[H^i_{W,c}(X,\ZZ^D)=\left\{
	\begin{array}{ll}
	\ZZ & i=1\\
	0 & i=2
	\end{array}\right.
	\]
	and an exact sequence
	\[
	0 \to \ZZ^{r_2}\to H^{-1}_{W,c}(X,\ZZ^D) \to \cal{O}_K^\times \to (\ZZ/2\ZZ)^{r_1} \to H^0_{W,c}(X,\ZZ^D) \to \mathrm{Pic}(X) \to 0
	\]
	Moreover $H^0_{W,c}(X,\ZZ^D)=\mathrm{Pic}^+(X)$.
\end{prop}
\begin{rmk}
	If $X$ is singular, we still have $H^2_{W,c}(X,\ZZ^D)=0$; the abelian group $H^1_{W,c}(X, \ZZ^D)$ is of rank $1$ but it can have torsion coming from $H^1(X,\ZZ)$ if $X$ is not unibranch.
\end{rmk}

\begin{prop}\label{weil_etale_jZ}
	Suppose that $X=\Spec(\cal{O}_K)$ is regular and let $j:U\to X$ be an open immersion with $U\neq X$; we keep the notations from \cref{cohomology_compact_support_jZD}. Then $H^1_{W,c}(X,(j_!\ZZ)^D)=H^2_{W,c}(X,(j_!\ZZ)^D)=0$ and we have short exact sequences
	\[
	0 \to \ZZ^{r_2} \to H^{-1}_{W,c}(X,(j_!\ZZ)^D) \to \cal{O}_{K,S,+}^\times \to 0
	\]
	and
	\[
	0 \to \mathrm{Pic}^+(U) \to H^0_{W,c}(X,(j_!\ZZ)^D) \to \mathrm{ker}(\ZZ^{s_f}\xrightarrow{\Sigma}\ZZ) \to 0
	\]
\end{prop}

\begin{prop}
	Let $i:x\to X$ be a closed point and $M$ a finite type discrete $G_x$-module. Then
	\[H^i_{W,c}(X,(i_\ast M)^D)=\left\{
	\begin{array}{ll}
	0 & i=-1\\
	H^0(G_x,M^\vee)=\Hom_{G_x}(M,\ZZ) & i=0\\
	(H^0(G_x,M)_{tor})^\ast & i=2
	\end{array}\right.
	\]
	and we have a short exact sequence
	\[
	0 \to H^1(G_x,M)^\ast \to H^1_{W,c}(X,(i_\ast M)^D) \to \Hom(H^0(G_x,M),\ZZ) \to 0
	\]
\end{prop}

\begin{prop}
	If $F$ is constructible then
	\[
	R\Gamma_{c,B}(X,F^D) \xrightarrow{\simeq} R\Gamma_{W,c}(X,F^D)
	\]
\end{prop}

\subsection{Functoriality properties}

\begin{thm}\label{weil_etale_well_defined}
	Let $F$ be red or blue. The complex $R\Gamma_{W,c}(X,F^D)$ is well-defined up to unique isomorphism, is functorial in red-to-blue morphisms, and gives long exact cohomology sequences for red-to-blue short exact sequences.
\end{thm}

\begin{proof}
	For each red or blue sheaf $F$ we fixe a choice $R\Gamma_{W,c}(X,F^D)$ of cone of $\alpha_F$. Let $f:F\to G$ be a red-to-blue morphism. Let us consider the following diagram with distinguished rows:
	\[\begin{tikzcd}
	{D_F} & {R\Gamma_{c,B}(X,F^D)} & {R\Gamma_{W,c}(X,F^D)} & {~} \\
	{D_G} & {R\Gamma_{c,B}(X,G^D)} & {R\Gamma_{W,c}(X,G^D)} & {~}
	\arrow["{f^\ast}", from=2-1, to=1-1]
	\arrow["{f^\ast}", from=2-2, to=1-2]
	\arrow["{\alpha_F}", from=1-1, to=1-2]
	\arrow["{\alpha_G}", from=2-1, to=2-2]
	\arrow[from=1-2, to=1-3]
	\arrow[from=1-3, to=1-4]
	\arrow[from=2-3, to=2-4]
	\arrow[from=2-2, to=2-3]
	\end{tikzcd}\]
	We claim that the left square commutes. Indeed, by \cref{hom_via_cohomology} it suffices to check it in cohomology in degree $1$ and $2$, in which case it follows from the functoriality of the maps $D_{(-)} \to R\Hom_\ZZ(R\Gamma(X,-),\QQ/\ZZ[-2])$, $R\hat{\Gamma}_{c,B}(X,(-)^D) \to R\Hom_\ZZ(R\Gamma(X,-),\QQ/\ZZ[-2])$ and $R\Gamma_{c,B}(X,(-)^D)\to R\hat{\Gamma}_{c,B}(X,(-)^D)$. We obtain an induced morphism $f^\ast : R\Gamma_{W,c}(X,G^D) \to R\Gamma_{W,c}(X,F^D)$ in $D(\ZZ)$ completing the diagram to a morphism of distinguished triangles. Let us show that this induced morphism is uniquely determined by the left square. It suffices to show that the natural map
	\[
	\Hom(R\Gamma_{W,c}(X,G^D),R\Gamma_{W,c}(X,F^D)) \to \Hom(R\Gamma_{c,B}(X,G^D),R\Gamma_{W,c}(X, F^D))
	\]
	is injective. We have an exact sequence
	\[
	\Hom(D_G[1],R\Gamma_{W,c}(X,F^D)) \to\Hom(R\Gamma_{W,c}(X,G^D),R\Gamma_{W,c}(X,F^D)) \to \Hom(R\Gamma_{c,B}(X,G^D),R\Gamma_{W,c}(X, F^D))
	\]
	Since $D_G$ is a complex of $\QQ$-vector spaces, multiplication by $n$ on $D_G$ is a quasi-isomorphism for every integer $n$ so the same holds for $\Hom(D_G[1],R\Gamma_{W,c}(X,F^D))$ by functoriality. The latter must thus be a $\QQ$-vector space. On the other hand, $\Hom(R\Gamma_{W,c}(X,G^D),R\Gamma_{W,c}(X,F^D))$ is a finite type abelian group \footnote{using \cref{weil-etale_perfect,Verdier_ss}, for instance.}, so the image of $\Hom(D_G[1],R\Gamma_{W,c}(X,F^D))$ in $\Hom(R\Gamma_{W,c}(X,G^D),R\Gamma_{W,c}(X,F^D))$ must be $0$ and the claim follows. This proves the functoriality in red-to-blue morphisms; by running the same argument with the morphism $\mathrm{id}:F\to F$ for a choice of two different cones of $\alpha_F$, we obtain the uniqueness up to unique isomorphism.
	
	Let $0\to F \to G \to H \to 0$ be a red-to-blue short exact sequence. Denote $u: H \to F[1]$ the induced map. By the previous argument, we obtain a diagram
	\[\begin{tikzcd}
	{D_F} & {R\Gamma_{c,B}(X,F^D)} & {R\Gamma_{W,c}(X,F^D)} & ~ \\
	{D_G} & {R\Gamma_{c,B}(X,G^D)} & {R\Gamma_{W,c}(X,G^D)} &  ~\\
	{D_H} & {R\Gamma_{c,B}(X,H^D)} & {R\Gamma_{W,c}(X,H^D)} & ~ \\
	{D_F[-1]} & {R\Gamma_{c,B}(X,F^D)[-1]} & {R\Gamma_{W,c}(X,F^D)[-1]} & ~
	\arrow["{f^\ast}", from=2-1, to=1-1]
	\arrow["{f^\ast}", from=2-2, to=1-2]
	\arrow["{\alpha_F}", from=1-1, to=1-2]
	\arrow["{\alpha_G}", from=2-1, to=2-2]
	\arrow[from=1-2, to=1-3]
	\arrow[from=1-3, to=1-4]
	\arrow[from=2-3, to=2-4]
	\arrow[from=2-2, to=2-3]
	\arrow["{u^\ast}", from=4-1, to=3-1]
	\arrow["{g^\ast}", from=3-1, to=2-1]
	\arrow["{\alpha_H}", from=3-1, to=3-2]
	\arrow["{\alpha_F[-1]}", from=4-1, to=4-2]
	\arrow[from=3-2, to=3-3]
	\arrow[from=4-2, to=4-3]
	\arrow["{g^\ast}", from=3-2, to=2-2]
	\arrow["{u^\ast}", from=4-2, to=3-2]
	\arrow["{f^\ast}", from=2-3, to=1-3]
	\arrow["{g^\ast}", from=3-3, to=2-3]
	\arrow[from=4-3, to=4-4]
	\arrow[from=3-3, to=3-4]
	\arrow["{u^\ast}", dashed, from=4-3, to=3-3]
	\end{tikzcd}\]
	To construct the unique dotted arrow $u^\ast$ making the diagram commute, repeat the previous argument using \cref{cohomology_dagger,Verdier_ss}.
	
	The right column provides a triangle that is not necessarily a distinguished. We will show that it induces a long exact cohomology sequence. Fix a prime $p$ and an integer $n\in \NN^\times$. Observe that $D_F$ is a complex of $\QQ$-vector spaces, so that $D_F\otimes^L \ZZ/p^n\ZZ=0$. It follows that
	\[
	R\Gamma_{c,B}(X,F^D)\otimes^L \ZZ/p^n\ZZ \xrightarrow{\simeq} R\Gamma_{W,c}(X,F^D)\otimes^L \ZZ/p^n\ZZ
	\]
	Denote by $(-)^{\wedge}_p=R\lim (-\otimes^L \ZZ/p^n\ZZ)$ the derived $p$-completion. By passing to the derived limit, we find
	\[
	R\Gamma_{c,B}(X,F^D)^{\wedge}_p \xrightarrow{\simeq} R\Gamma_{W,c}(X,F^D)^{\wedge}_p = R\Gamma_{W,c}(X,F^D)\otimes^L \ZZ_p
	\]
	naturally in $F$ where the right equality holds because $R\Gamma_{W,c}(X,F^D)$ is a perfect complex\footnote{This follows from \cite[\href{https://stacks.math.columbia.edu/tag/0EEU}{Tag 0EEU}]{stacks-project} and \cite[\href{https://stacks.math.columbia.edu/tag/00MA}{Tag 00MA}]{stacks-project} by filtering with the truncations.}. Since derived $p$-completion is an exact functor, it follows that $R\Gamma_{W,c}(X,(-)^D)\otimes^L \ZZ_p$ is an exact functor too. Now $\ZZ_p$ is flat and the family $(\ZZ_p)_{p~ \text{prime}}$ is faithfully flat so it suffices to show that the long cohomology sequence is exact after tensoring with $\ZZ_p$, for each prime $p$; this follows from the exactness of $R\Gamma_{W,c}(X,(-)^D)\otimes^L \ZZ_p$.
\end{proof}

\begin{prop}
		Let $Y=\Spec(\cal{O}')$ be the spectrum of an order in a number field with a finite dominant morphism $\pi:Y\to X$ and let $F$ be a red or blue $\ZZ$-constructible sheaf on $Y$. Then $\pi_\ast F$ is red or blue and we have a unique isomorphism
		\[
		R\Gamma_{W,c}(Y,F^D) \xrightarrow{\simeq} R\Gamma_{W,c}(X,(\pi_\ast F)^D)
		\]
\end{prop}

\begin{proof}
We can run an argument similar to the previous proof: it suffices to check that the square
\[\begin{tikzcd}
{R\Gamma(Y,F)^\dagger[-2]} & {R\Gamma_{c,B}(Y,F^D)} & {} \\
{R\Gamma(X,\pi_\ast F)^\dagger[-2]} & {R\Gamma_{c,B}(X,(\pi_\ast F)^D)}
\arrow["{\alpha_F}", from=1-1, to=1-2]
\arrow["\simeq"', from=2-1, to=1-1]
\arrow["\simeq"', from=2-2, to=1-2]
\arrow["{\alpha_{\pi_\ast F}}"', from=2-1, to=2-2]
\end{tikzcd}\]
is commutative. This can be checked in cohomology in degree $2$ and $3$, whence from the definition of $\alpha_F$ we reduce to checking that
\[\begin{tikzcd}
{\hat{H}^i_{c,B}(Y,F^D)} & {H^{2-i}(Y,F)^\ast} & {} \\
{\hat{H}^i_{c,B}(X,(\pi_\ast F)^D)} & {H^{2-i}(X,\pi_\ast F)^\ast}
\arrow[from=1-1, to=1-2]
\arrow[from=2-1, to=2-2]
\arrow["\simeq"', from=2-1, to=1-1]
\arrow["\simeq"', from=2-2, to=1-2]
\end{tikzcd}\]
is commutative for $i=2,3$. This would be implied by the commutativity of the following diagram:
\[\begin{tikzcd}[ampersand replacement=\&]
{R\Gamma(Y, F)\otimes^L R\hat{\Gamma}_{c,B}(Y,F^D)} \& {R\hat{\Gamma}_{c,B}(Y,\ZZ^D)} \& {\QQ/\ZZ[-2]} \\
{R\Gamma(X,\pi_\ast F)\otimes^L R\hat{\Gamma}_{c,B}(X,(\pi_\ast F)^D)} \& {R\hat{\Gamma}_{c,B}(X,\ZZ^D)} \& {\QQ/\ZZ[-2]}
\arrow["\simeq", from=1-1, to=2-1]
\arrow[from=2-1, to=2-2]
\arrow[from=1-1, to=1-2]
\arrow[from=1-2, to=2-2]
\arrow[Rightarrow, no head, from=1-3, to=2-3]
\arrow["{\tau^{\geq 2}}"', from=1-2, to=1-3]
\arrow["{\tau^{\geq 2}}"', from=2-2, to=2-3]
\end{tikzcd}\]
By \cref{compatibility_pairings_finite}, the left square is commutative. Moreover, the arrow $R\hat{\Gamma}_{c,B}(Y,\ZZ^D) \to R\hat{\Gamma}_{c,B}(X,\ZZ^D)$ is the composition 
\[
R\hat{\Gamma}_{c,B}(Y,\ZZ^D)\simeq R\hat{\Gamma}_c(Y,\ZZ^c_Y)\simeq R\hat{\Gamma}_c(X,\pi_\ast R\pi^!\ZZ^c_X) \to R\hat{\Gamma}_c(X,\ZZ^c_X)\simeq R\hat{\Gamma}_{c,B}(X,\ZZ^D)
\]
coming from the counit $\varepsilon:\pi_\ast R\pi^! \to \mathrm{id}$. It induces in cohomology in degree $2$ the identity $\QQ/\ZZ \to \QQ/\ZZ$ \cite[II.3.10]{ADT},\cite[8.9]{Harari20}, so the right square is commutative.
\end{proof}

The Weil-étale complex splits rationally:
\begin{prop}\label{prop:rational_splitting}
	Let $F$ be a red or blue sheaf.	There is an isomorphism
	\[
	R\Gamma_{W,c}(X,F^D)_\QQ \xrightarrow{\simeq} R\Hom(R\Gamma(X,F),\QQ[-1])  \oplus R\Gamma_{c,B}(X,F^D)_\QQ
	\]
	natural in red-to-blue morphisms and red-to-blue short exact sequences.
\end{prop}

\begin{proof}
	Consider the distinguished triangle
	\[
	R\Gamma_{c,B}(X,(-)^D)_\QQ \to R\Gamma_{W,c}(X,(-)^D)_\QQ \xrightarrow{p} D_{(-)}[1]
	\]
	It suffices to show that $p$ has a section; but if $F$ is a red sheaf or $G$ is a blue sheaf, considerations on the Verdier spectral sequence \cref{Verdier_ss} show that 
	\[
	\begin{array}{lr}
	\Hom_{D(\ZZ)}(D_F[1],R\Gamma_{c,B}(X,G^D)_\QQ)=0, & \Hom_{D(\ZZ)}(D_F[1],R\Gamma_{c,B}(X,G^D)_\QQ[1])=0.
	\end{array}
	\]
	so that composition with $p_G$ induces an isomorphism
	\[
	\Hom_{D(\ZZ)}(D_F[1],R\Gamma_{W,c}(X,G^D)_\QQ) \xrightarrow[p_{G,\ast}]{\simeq} \Hom_{D(\ZZ)}(D_F[1],D_G[1])
	\]
	We then put $s_F=p_{F,\ast}^{-1}(\mathrm{id})$; the functoriality is easily checked from the above isomorphism.
\end{proof}

\begin{prop}
	Let $Y=\Spec(\cal{O}')$ be the spectrum of an order in a number field with a finite dominant morphism $\pi:Y\to X$ and let $F$ be a red or blue sheaf on $Y$. Then the rational splitting of Weil-étale cohomology is compatible with $\pi_\ast$, i.e. the following square commutes
	\[\begin{tikzcd}[ampersand replacement=\&]
		{R\Gamma_{W,c}(Y,F^D)_\QQ} \& { D_F[1] \oplus R\Gamma_c(Y,F^D)_\QQ } \\
		{R\Gamma_{W,c}(X,(\pi_\ast F)^D)_\QQ} \& { D_{\pi_\ast F}[1] \oplus R\Gamma_c(X,(\pi_\ast F)^D)_\QQ}
		\arrow["\simeq", from=1-1, to=2-1]
		\arrow["\simeq"', from=1-1, to=1-2]
		\arrow["\simeq", from=2-1, to=2-2]
		\arrow["\simeq"', from=1-2, to=2-2]
	\end{tikzcd}\]
\end{prop}
\begin{proof}
	It suffices to prove that in the following square, the square with the sections commutes:
	\[\begin{tikzcd}[ampersand replacement=\&]
		{R\Gamma_{W,c}(Y,F^D)_\QQ} \& {D_F[1]} \\
		{R\Gamma_{W,c}(X,(\pi_\ast F)^D)_\QQ} \& {D_{\pi_\ast F}[1]}
		\arrow["\gamma", from=1-1, to=2-1]
		\arrow["{p_F}"', shift right=2, from=1-1, to=1-2]
		\arrow["{p_{\pi_\ast F}}", shift left=2, from=2-1, to=2-2]
		\arrow["\delta"', from=1-2, to=2-2]
		\arrow["{s_F}"', shift right=2, from=1-2, to=1-1]
		\arrow["{s_{\pi_\ast F}}", shift left=2, from=2-2, to=2-1]
	\end{tikzcd}\]
	As in the previous proof, there is an isomorphism 
	\[
	\Hom_{D(\ZZ)}(D_F[1],R\Gamma_{W,c}(X,(\pi_\ast F)^D)_\QQ) \xrightarrow[p_{\pi_\ast F,\ast}]{\simeq} \Hom_{D(\ZZ)}(D_F[1],D_{\pi_\ast F}[1])
	\]
	We compute 
	\[p_{\pi_\ast F}\gamma s_F=\delta p_F s_F=\delta=p_{\pi_\ast F} s_{\pi_\ast F}\delta.\]
\end{proof}
\section{The tangent space and the fundamental line}

Denote $g:\eta=\Spec(K) \to X$ the inclusion of the generic point, $G_K$ the absolute Galois group of $K$, $K^t$ the maximal tamely ramified extension of $K$, $G_K^t=\mathrm{Gal}(K^t/K)$ and $\cal{O}_{K^t}$ the ring of integers of $K^t$. For each finite place $v$ of $K$, denote $P_v\subset I_v \subset D_v \subset G_K$ the wild ramification, inertia, decomposition subgroups of $G_K$ at $v$~\footnote{coming from a choice of embedding $K_v \hookrightarrow K^{sep}$}. Denote $N$ the smallest closed normal subgroup of $G_K$ containing $P_v$ for each finite place $v$. We have $G_K^t=G_K/N$ by \cref{tame_galois_group_and_wild_ramification}. For a sheaf $F$ on $X$, denote $F_\eta:=g^\ast F$, seen as a discrete $G_K$-module.
\begin{defi}
	Let $F$ be a $\ZZ$-constructible sheaf on $X$.
	\begin{itemize}
		\item We say that $F$ is tamely ramified if for each finite place $v$ of $K$, $P_v$ acts trivially on $F_\eta$; then $F_\eta$ carries a natural action of $G_K^t$.
		\item Suppose that $F$ is tamely ramified. The tangent space of $F^D$ is the complex
		\[
		\Lie_X(F^D):=R\Hom_{G_K^t}(F_\eta, \cal{O}_{K^t}[1])
		\]
	\end{itemize}
\end{defi}

\begin{rmk}
	\begin{itemize}
		\item Let us motivate the above definition. Following the work of Flach-Morin \cite{Flach18}, $\Lie_X(F^D)$ should be the "additive complex" giving the additive part of the fundamental line. It should behave like Milne's correcting factor in the special value formula for $L_X(F^D,s)$ at $s=0$. Since $L_X(\ZZ^D,s)=\zeta_X(s+1)$ when $X=\Spec(\cal{O}_K)$ is regular, we must have $\Lie_X(\GG_m)=\cal{O}_K$: this is the additive complex for $\ZZ(1)[1]$ on $X$. We want $\Lie_X$ to be an exact functor\footnote{A way to avoid the tame ramification hypothesis may be to relax this condition and only ask that $\det_\ZZ(\Lie_X((-)^D)$ is multiplicative with respect to fiber sequences)}. Recall the $L$-function of a $\ZZ$-constructible sheaf $L_X(-,s)$ introduced in \cite[§ 6.4]{AMorin21}. If $i:x\to X$ is the inclusion of a closed point and $M$ is a finite type $G_x$-module, we have $L_X((i_\ast M)^D,s)=L_X(i_\ast M^\vee,s)$. Since the special value at $s=0$ of $L_X(i_\ast M^\vee,s)$ doesn't involve an additive complex, it is expected that $\Lie_X((i_\ast M)^D)=0$. Therefore, $\Lie_X(F^D)$ should only depend on $F_\eta$. In \cite{Geisser20}, Geisser-Suzuki consider a torus $T$ over a function field $K$ associated to a proper smooth curve over a finite field $C$. They write down a special value formula for the $L$-function of $T$ (defined in terms of the rational $l$-adic Tate module) at $s=1$ in terms of Weil-étale cohomology and of the Lie algebra of the connected Néron model $\cal{T}^0$ on $C$ of $T$. Moreover, they show that on the étale site, $\cal{T}^0=R\cal{H}om(\widetilde{\cal{Y}},\GG_m)$ where $\widetilde{\cal{Y}}$ is a bounded complex with $\ZZ$-constructible cohomology sheaves related to the character group of $T$. This inspired our construction.
		\item If $F$ is $\ZZ$-constructible and $X=\Spec(\cal{O}_K)$ is regular, $F$ is tamely ramified if and only if for each $U\subset X$ such that $F_{|U}$ is locally constant, there exist a finite extension $L/K$ such that the normalization $\pi :Y \to X$ of $X$ in $L$ is étale over $U$ with $(\pi^\ast F)_{|\pi^{-1}(U)}$ constant, and $\pi$ is tamely ramified above points $x\in X\backslash U$.
		\item If $X$ is not regular we nonetheless have $\Lie_X(\ZZ^D)=\cal{O}_K[1]$. This seems strange to us, but as we explained $\Lie_X(F^D)$ depends only on $F_\eta$ which does not see singular points so it is also more or less expected.
		\item We have $\Lie_X(F^D)_\QQ=\Hom_{G_K}(F_\eta,K^{sep})[1]$.
	\end{itemize}
\end{rmk}

The restriction to tamely ramified sheaves is justified by the following theorem, essentially due to Noether \cite{Noether32}\footnote{See also \cite[Theorem 1.1]{Kock04}}:
\begin{thm}[{\cite[6.1.10]{NeukirchCohomologyNumberFields}}]\label{noether_theorem}
	Let $K$ be a number field or a $p$-adic field. The discrete $G_K^t$-module $\cal{O}_{K^t}$ is cohomologically trivial. In particular
	\[
	R\Gamma(G_K^t, \cal{O}_{K^t})=\cal{O}_K[0]
	\]
\end{thm}

\begin{cor}\label{cor_noether}
	Let $L/K$ be  finite tamely ramified Galois extension of number fields or $p$-adic fields with Galois group $G$. For $i\geq 1$ denote $\cal{O}_L[[t]]_i:= 1+t^i \cal{O}_L[[t]]$. Then $\cal{O}_L[[t]]_i$ is a cohomologically trivial $G$-module.
\end{cor}
\begin{proof}
	We have an isomorphism $\cal{O}_L[[t]]_i/\cal{O}_L[[t]]_{i+1}\simeq \cal{O}_L$ and 
	\[
	\cal{O}_L[[t]]_i=\lim_{k>i} \cal{O}_L[[t]]_i/\cal{O}_L[[t]]_k
	\]
	The transition maps are surjective, and the terms are cohomologically trivial by induction using \cref{noether_theorem}. Moreover, the kernel of $\cal{O}_L[[t]]_i/\cal{O}_L[[t]]_{k+1} \to \cal{O}_L[[t]]_i/\cal{O}_L[[t]]_k$ is $\cal{O}_L[[t]]_k//\cal{O}_L[[t]]_{k+1}\simeq \cal{O}_L$ which is cohomologically trivial, so for every subgroup $H$ of $G$ the maps $(\cal{O}_L[[t]]_i/\cal{O}_L[[t]]_{k+1})^H \to (\cal{O}_L[[t]]_i/\cal{O}_L[[t]]_k)^H $ are surjective. 
	
	Put $A_k=\cal{O}_L[[t]]_i/\cal{O}_L[[t]]_k$ and let $H$ be a subgroup of $G$. We compute:
	\begin{align*}
	R\Gamma(H,\cal{O}_L[[t]]_i)=R\Gamma(H,\lim A_k)=R\Gamma(H,R\lim A_k)&= R\lim R\Gamma(H,A_k)\\
	&=R\lim (A_k^H[0])\\
	&=(\lim A_k^H) [0]
	\end{align*}
	where each limit is a derived limit because the transition maps are surjective.
\end{proof}

The following lemma together with the above theorem will enable us to compute the tangent space of $F^D$ more explicitely:
\begin{lem}\label{lem_CTtorsionfree}
	Let $G$ be a profinite group and let $\cal{O}$ be a torsion-free cohomologically trivial discrete $G$-module. Let $M$ be a finite type discrete $G$-module. The complex $R\Hom_{G}(M, \cal{O})$ is cohomologically concentrated in degree $[0,1]$. Moreover, if $M$ is finite it is concentrated in degree $1$ and if $M$ is torsion-free it is concentrated in degree $0$.
\end{lem}
\begin{proof}
	 To show the vanishing result, we can use the short exact sequence $0\to M_{tor} \to M \to M/tor \to 0$ to reduce to the finite and torsion-free cases; thus it suffices to study those cases:
	\begin{itemize}
		\item Suppose first that $M$ is finite. We then have $\Hom_{G}(M,\cal{O})=0$ because $\cal{O}$ is torsion-free. As $\Ext^p_{\ZZ}(M,\cal{O})=0$ for $p\neq 1$, we obtain for $i \geq 1$:
		\[
		\Ext^i_{G}(M,\cal{O})=H^{i-1}(G,\Ext^1_\ZZ(M,\cal{O})).
		\]
		Let us now show the vanishing for $i>1$. We can suppose that $M$ is $p$-primary, and even $p$-torsion.  Let $H$ be any $p$-Sylow of $G$. As an $H$-module, $M$ has a composition series with quotients isomorphic to $\ZZ/p\ZZ$ with its trivial action \cite[IX, Thm 1]{Serrelocaux}. We have
		\[
		H^i(H,\Ext^1_\ZZ(\ZZ/p\ZZ,\cal{O}))=H^i(H,\cal{O}/p\cal{O})=0
		\]
		for $i\geq 1$ because $\cal{O}$ is cohomologically trivial. It follows that for any $i\geq 1$ we also have $H^i(H,\Ext^1_\ZZ(M,\cal{O}))=0$. Finally, the following triangle commutes:
		\[\begin{tikzcd}[ampersand replacement=\&]
			{H^i(G,\Ext^1_\ZZ(M,\cal{O})} \& {} \& {H^i(G,\Ext^1_\ZZ(M,\cal{O})} \\
			\& {H^i(H,\Ext^1_\ZZ(M,\cal{O})=0}
			\arrow["{\mathrm{Res}}"', from=1-1, to=2-2]
			\arrow["{\mathrm{Cores}}"', from=2-2, to=1-3]
			\arrow["{[G:H]}", from=1-1, to=1-3]
			\arrow["\simeq"', draw=none, from=1-1, to=1-3]
		\end{tikzcd}\]
		whence it follows that $H^i(G,\Ext^1_\ZZ(M,\cal{O}))=0$ for $i\geq 1$.
		
		\item Suppose now that $M$ is torsion-free, and let us show that $\Ext^i_{G}(M,\cal{O})=0$ for $i>0$. By Artin induction and the previous case it suffices to consider the case of $M=\mathrm{ind}^{G}_H \ZZ$ for $H$ an open subgroup. Since $\cal{O}$ is cohomologically trivial we have
		\[
		R\Hom_{G}(\mathrm{ind}^{G}_H \ZZ,\cal{O})=R\Gamma(H,\cal{O})=\cal{O}^H[0]
		\]
		which is concentrated in degree $0$.
	\end{itemize}
\end{proof}

\begin{cor}
	Let $M$ be a finite type discrete $G_K^t$-module. Then $R\Hom_{G_K^t}(M, \cal{O}_{K^t}[1])$ is a perfect complex cohomologically concentrated in degree $[-1,0]$. 
	More precisely, $\Ext^{-1}_{G_K^t}(M,\cal{O}_{K^t}[1])$ is free of finite rank and $\Ext^0_{G_K^t}(M,\cal{O}_{K^t}[1])$ is finite.
\end{cor}
\begin{proof}
	We apply first \cref{lem_CTtorsionfree} to obtain the vanishing outside $[-1,0]$. Now let $H$ be an open normal subgroup of $G_K^t$ acting trivially on $M$ and corresponding to a finite tamely ramified Galois extension $L/K$, and put $G:= G_K^t/H=\Gal(L/K)$. There is a natural $G$-action on $M$. The functor $R\Gamma(H,-)$ is right adjoint to the forgetful functor $G\text{-}\mathrm{Mod}\to G_K^t\text{-}\mathrm{Mod}$, so we obtain
	\[
	R\Hom_{G}(M,\cal{O}_{K^t}[1])=R\Hom_G(M,R\Gamma(H,\cal{O}_{K^t}[1]))=R\Hom_G(M,\cal{O}_L[1])
	\]
	by \cref{noether_theorem}.
	
	Since $\cal{O}_L$ is torsion-free, $\Hom_G(M,\cal{O}_L)$ is free of finite rank. On the other hand, we can use the short-exact sequence $0 \to M_{tor} \to M \to M/tor \to 0$, the vanishing $\Ext^p_\ZZ(M_{tor},\cal{O}_L)=0$ for $p\neq 1$ and \cref{lem_CTtorsionfree} (applied to $G$ and $\cal{O}_L$) to compute
	\[
	\Ext^1_G(M,\cal{O}_L)\xrightarrow{\simeq} \Ext^1_G(M_{tor},\cal{O}_L)=H^0(G,\Ext^1_\ZZ(M_{tor},\cal{O}_L))
	\]
	which is finite.
\end{proof}

\begin{cor}\label{lie_perfect}
	Let $F$ be a tamely ramified $\ZZ$-constructible sheaf. Then $\Lie_X(F^D)$ is a perfect complex concentrated in degree $[-1,0]$. Moreover $H^{-1}\Lie_X(F^D)$ is free of finite rank and $H^0\Lie_X(F^D)$ is finite.
\end{cor}

\begin{cor}\label{lie_equals_lie}
	Suppose that $M$ is a torsion-free discrete $G_K^t$-module and let $T$ be the torus over $K$ with character group $M$. Denote $\cal{T}$ the lft Néron model of $T$ over $\Spec(\cal{O}_K)$. We have a canonical, functorial isomorphism
	\[
	\Lie(\cal{T})\xrightarrow{\simeq}\Hom_{G_K^t}(M,\cal{O}_{K^t})
	\]
	where the left term is the Lie algebra of $\cal{T}$.
\end{cor}
\begin{proof}
	Both terms are lattices inside $\Lie(T)=\Hom_{G_K}(M,K^{sep})$. It suffices to show that they are canonically isomorphic after tensoring with $\cal{O}_{K_v}$ for all non-archimedean places $v$. On the one hand we have
	\[
	\Lie(\cal{T})\otimes_{\cal{O}_K} \cal{O}_{K_v}=\Lie(\cal{T}_{\cal{O}_{K_v}})=\Lie(\cal{N}(T_{K_v}))
	\]
	where $\cal{N}$ denotes the lft Néron model. On the other hand let $G=\Gal(L/K)$ be a quotient of $G_K^t$ through which the action of $G_K^t$ on $M$ factors. We have
	\[
	\cal{O}_L\otimes_{\cal{O}_K}\cal{O}_{K_v}\simeq \prod_{w|v} \cal{O}_{L_w}
	\]
	and thus
	\begin{align*}
	\Hom_{G_K^t}(M,\cal{O}_{K^t})\otimes_{\cal{O}_K}\cal{O}_{K_v}=\Hom_G(M,\cal{O}_L)\otimes_{\cal{O}_K}\cal{O}_{K_v}&=\Hom_G(M,\cal{O}_L\otimes_{\cal{O}_K}\cal{O}_{K_v})\\
	&=\Hom_G(M,\prod_{w|v} \cal{O}_{L_w})\\
	&= \Hom_{D_{w_0/v}}(M,\cal{O}_{L_{w_0}}
	\end{align*}
	where $w_0$ is a fixed place of $L$ above $v$, so we can reduce to the case where $K$ is a $p$-adic field. By \cite[A1.7]{Chai01} we have a canonical functorial isomorphism
	\[
	\Lie(\cal{T})\xrightarrow{\simeq}\{v\in \Hom_G(M,\cal{O}_L) ~|~v ~\text{lifts to}~ \Hom_G(M,\cal{O}_L[[t]]_1)\}
	\]
	It thus suffices to show that $\Hom_G(M,\cal{O}_L[[t]]_1) \to \Hom_G(M,\cal{O}_L)$ is surjective. We have an exact sequence
	\[
	\Hom_G(M,\cal{O}_L[[t]]_1) \to \Hom_G(M,\cal{O}_L) \to \Ext^1_G(M,\cal{O}_L[[t]]_2)
	\]
	Now $\cal{O}_L[[t]]_2$ is torsion-free and cohomologically trivial by \cref{cor_noether} and $M$ is a torsion-free finite type $G$-module, so we can apply \cref{lem_CTtorsionfree} to obtain $\Ext^i_G(M,\cal{O}_L[[t]]_2)=0$ for $i\geq 1$.
\end{proof}

\begin{rmk}
	As a corollary, a short exact sequence of tori that are split by a tamely ramified extension gives rise to a short exact sequence of the Lie algebras of the respective lft Néron models.
\end{rmk}

In the following, denote $(-)_\RR:=-\otimes^L \RR$.
\begin{prop}\label{lie_tensor_R}
	Let $F$ be a tamely ramified $\ZZ$-constructible sheaf. We have
	\[
	\Lie_X(F^D)_\RR = R\Hom_{G_\RR,X(\CC)}(\alpha^\ast F,\CC[1])
	\]
\end{prop}

\begin{proof}
	Let $L/K$ be a finite tamely ramified Galois extension such that $G_L$ acts trivially on $F_\eta$, and denote $G:=G_K^t/G_L^t=\Gal(L/K)$. We have:
	\begin{align*}
		\Lie_X(F^D)_\RR = R\Hom_G(F_\eta,\cal{O}_L[1])_\RR &\underset{(\ast)}{=} R\Hom_G(F_\eta,\cal{O}_L \otimes_\ZZ \RR[1]) \\
		&=R\Hom_G(F_\eta, \prod_{v~\text{archimedean}} \prod_{w|v} L_w[1])\\
		& = \prod_{v~\text{archimedean}} R\Hom_G(F_\eta,\prod_{w|v} L_w[1]).
	\end{align*}
	where $(\ast)$ holds because $F_\eta$ is a finite type abelian group, hence a finite presentation $\ZZ[G]$-module. For each archimedean place $v$, choose a place $w_0$ in $L$ above $v$. The group $G$ acts transitively on the places above $v$, the fields $L_w$ are pairwise isomorphic, and the stabilizer $D_{w_0}$ of $w_0$ identifies with $\Gal(L_{w_0}/K_v)$. Thus $\prod_{w|v} L_w = \ind^{G}_{D_{w_0}} L_{w_0}$ and we find
	\[
	\Lie_X(F^D)_\RR = \prod_{v~\text{archimedean}} R\Hom_{D_{w_0}}(F_\eta,L_{w_0}[1])
	\]
	On the other hand, for $v$ an archimedean place we have $F_v=F_\eta$ with the restricted action to $G_{K_v}:=\Gal(\CC/K_v)$. Therefore the action of $G_{K_v}$ factors through $\Gal(L_{w_0}/K_v)$; as $R\Gamma(\Gal(\CC/L_{w_0}),-)$ is right adjoint to the forgetful functor $D_{w_0}\text{-}\mathrm{Mod} \to G_{K_v}\text{-}\mathrm{Mod}$, we find
	\begin{align*}
		R\Hom_{G_\RR,X(\CC)}(\alpha^\ast F,\CC[1])&=R\Gamma_{G_\RR}(X(\CC),R\cal{H}om_{X(\CC)}(\alpha^\ast F,\CC[1]))\\
		&=\prod_{v~\text{archimedean}} R\Gamma(G_{K_v},R\Hom_\ZZ(F_v,\CC[1]))\\
		&=\prod_{v~\text{archimedean}}R\Hom_{G_{K_v}}(F_v,\CC[1])\\
		&=\prod_{v~\text{archimedean}}R\Hom_{D_{w_0}}(F_v,L_{w_0}[1])\\
		&=\Lie_X(F^D)_\RR
\end{align*}

\end{proof}

\begin{prop}
	Let $Y=\Spec(\cal{O}')$ be the spectrum of an order in a number field with a finite dominant morphism $\pi:Y\to X$. If $K(Y)/K$ is tamely ramified and $F$ is a tamely ramified $\ZZ$-constructible sheaf on $Y$, we have
	\[
	\Lie_{X}((\pi_\ast F)^D)=\Lie_{Y}(F^D)
	\]
\end{prop}
\begin{proof}
	Denote $g':\Spec(L)\to Y$ the inclusion of the generic point and $\pi':L\to K$ the map induced by base change of $\pi$ by $g:\Spec(K)\to X$. By finite base change, we have $g^\ast \pi_\ast F=\pi'_\ast g'^\ast F$. If $K\subset H \subset G$ are normal subgroups of a profinite group $G$ and $M$ is a discrete $H/K$-module, we have
	\[
	\ind_{H/K}^{G/K} M = \mathrm{Cont}_{H/K}(G/K,M)\simeq \mathrm{Cont}(G/H,M)\simeq \mathrm{Cont}_H(G,M) = \ind_H^G M
	\]
	where the left term is seen as a $G$-module, and on the right $M$ has the natural $H$-module structure.	Denote $U:G_K^t\text{-}\mathrm{Mod}\to G_K\text{-}\mathrm{Mod}$ resp. $U':G_L^t\text{-}\mathrm{Mod}\to G_L\text{-}\mathrm{Mod}$ the forgetful functors. The functor $\pi'_\ast$ identifies with induction $\mathrm{ind}_{G_L}^{G_K}$. We have $K^t=L^t$ because $L/K$ is tamely ramified so we can apply the above to $\Gal{K^{sep}/K^t}\subset G_L \subset G_K$. We obtain the formula $U\circ \mathrm{ind}_{G_L^t}^{G_K^t}=\mathrm{ind}_{G_L}^{G_K}\circ U'$, so that
	\[
	\Lie_Y((\pi_\ast F)^D):=R\Hom_{G_K^t}(\mathrm{ind}_{G_L^t}^{G_K^t} g'^\ast F, \cal{O}_{K^t}[1])=R\Hom_{G_L^t}(g'^\ast F, \cal{O}_{K^t}[1])=\Lie_X(F^D)
	\]
	since induction is right \emph{and} left adjoint to the forgetful functor by the finiteness of $L/K$.
\end{proof}

\begin{defi}
	Let $F$ be a tamely ramified red or blue sheaf. The fundamental line is the determinant
	\[
	\Delta_{X}(F^D)=\det_\ZZ(R\Gamma_{W,c}(X,F^D))\otimes \det_\ZZ(\Lie_X(F^D))^{-1}
	\]
\end{defi}

\section{Deligne compactly supported cohomology and the duality theorem}\label{sec:deligne_and_duality}

There is a natural $G_\RR$-equivariant map of sheaves on $X(\CC)$ given by the logarithm of the absolute value $\log|\cdot|:\CC^\times \to \RR$. Let $F$ be a sheaf on $X$. We denote by $\mathrm{Log}$ the natural map
\begin{align*}
\mathrm{Log} : R\Hom(F,\ZZ^c_X)\xrightarrow{\alpha^\ast} R\Hom_{G_\RR,X(\CC)}(\alpha^\ast F,\overline{\QQ}^\times[1]) &\to R\Hom_{G_\RR,X(\CC)}(\alpha^\ast F,\CC^\times[1])\\ 
&\xrightarrow{(\log|\cdot|)_\ast[1]} R\Hom_{G_\RR,X(\CC)}(\alpha^\ast F,\RR[1])
\end{align*}
Since the target is a sheaf of $\RR$-vector spaces, the map factors as $\mathrm{Log} : R\Gamma(X,F^D)_\RR=R\Hom(F,\ZZ^c_X)_\RR \to R\Hom_{G_RR,X(\CC)}(\alpha^\ast F,\RR[1])$.
\begin{defi}\label{defi_deligne_compact}
	Let $F$ be a $\ZZ$-constructible sheaf. We define the Deligne compactly supported cohomology with coefficients in $F^D$ by
	\[
	R\Gamma_{c,\cal{D}}(X,F^D_\RR):=\mathrm{fib}\left(R\Gamma(X,F^D)_\RR \xrightarrow{\mathrm{Log}} R\Hom_{G_\RR,X(\CC)}(\alpha^\ast F,\RR[1])\right)
	\]
\end{defi}

\begin{rmk}
	We have chosen to name the different cohomologies with compact support by what happens at the archimedean places; here the cohomology at the archimedean places is replaced by a contravariant version of Deligne cohomology with real coefficients $\RR(1)_{\cal{D}}=[2i\pi\RR \to \CC]\xrightarrow{\simeq} \RR[-1]$.
\end{rmk}

 The commutative diagram with exact rows of $G_\RR$-equivariant sheaves
\[\begin{tikzcd}
0 & 2i\pi\ZZ & \CC & {\CC^\times} & 0 \\
0 & 2i\pi\RR & \CC & \RR & 0
\arrow[from=2-1, to=2-2]
\arrow[from=2-2, to=2-3]
\arrow[from=2-3, to=2-4]
\arrow[from=2-4, to=2-5]
\arrow["\log|\cdot|", from=1-4, to=2-4]
\arrow[Rightarrow, no head, from=1-3, to=2-3]
\arrow[hook, from=1-2, to=2-2]
\arrow[from=1-1, to=1-2]
\arrow[from=1-2, to=1-3]
\arrow[from=1-3, to=1-4]
\arrow[from=1-4, to=1-5]
\end{tikzcd}\]
gives a commutative square
\begin{equation}\label{square_infinity}
\begin{tikzcd}
{} && {\CC^\times} & {2i\pi\ZZ[1]} \\
&& \RR & {2i\pi\RR[1]}
\arrow[from=2-3, to=2-4]
\arrow["\log|\cdot|", from=1-3, to=2-3]
\arrow[from=1-3, to=1-4]
\arrow[from=1-4, to=2-4]
\end{tikzcd}\end{equation}
Consider the following diagram
\[\begin{tikzcd}[ampersand replacement=\&]
{R\Gamma_{c,B}(X,F^D)_\RR} \& {R\Gamma(X,F^D)_\RR} \& {R\Hom_{G_\RR,X(\CC)}(\alpha^\ast F,2i\pi\RR)[2]} \\
{R\Hom_{G_\RR,X(\CC)}(\alpha^\ast F,\CC[1])} \& {R\Hom_{G_\RR,X(\CC)}(\alpha^\ast F,\RR[1])} \& {R\Hom_{G_\RR,X(\CC)}(\alpha^\ast F,2i\pi\RR)[2]}
\arrow[from=1-1, to=1-2]
\arrow[from=1-2, to=1-3]
\arrow[from=2-2, to=2-3]
\arrow[from=2-1, to=2-2]
\arrow["{\mathrm{Log}}", dashed, from=1-1, to=2-1]
\arrow[Rightarrow, no head, from=1-3, to=2-3]
\arrow["{\mathrm{Log}}", from=1-2, to=2-2]
\end{tikzcd}\]
where the top fiber sequence is the defining fiber sequence of $R\Gamma_{c,B}(X,F^D)$ tensored with $\RR$ and the bottom fiber sequence is induced from the short exact sequence $0 \to 2i\pi\RR \to \CC \to \RR \to 0$. We claim that the right square is commutative, hence induces the left map making the whole diagram commute. This follows from the following commutative diagram, induced by the commutative square \eqref{square_infinity}, by using the universal property of base change to $\RR$-coefficients:
\[\begin{tikzcd}[ampersand replacement=\&]
	\& {R\Gamma(X,F^D)} \& {R\Hom_{G_\RR,X(\CC)}(\alpha^\ast F,\overline{\QQ}^\times[1])} \\
	{} \& {} \& {R\Hom_{G_\RR,X(\CC)}(\alpha^\ast F,\CC^\times[1])} \& {R\Hom_{G_\RR,X(\CC)}(\alpha^\ast F,2i\pi\ZZ[2])} \\
	\&\& {R\Hom_{G_\RR,X(\CC)}(\alpha^\ast F,\RR[1])} \& {R\Hom_{G_\RR,X(\CC)}(\alpha^\ast F,2i\pi\RR[2])}
	\arrow[Rightarrow, no head, from=2-4, to=3-4]
	\arrow["{(\log|\cdot|)_\ast[1]}", from=2-3, to=3-3]
	\arrow[from=3-3, to=3-4]
	\arrow[from=2-3, to=2-4]
	\arrow["{\alpha^\ast}", from=1-2, to=1-3]
	\arrow[from=1-3, to=2-3]
	\arrow["\mathrm{Log}"', bend right = 30, from=1-2, to=3-3]
\end{tikzcd}\]
By the $\infty$-categorical nine lemma\footnote{see \cite[lem. 2.2]{AMorin21}} we obtain
\[
R\Gamma_{c,\cal{D}}(X,F^D_\RR)=\mathrm{fib}\left(R\Gamma_{c,B}(X,F^D)_\RR \xrightarrow{\mathrm{Log}} R\Hom_{G_\RR,X(\CC)}(\alpha^\ast F,\CC[1])\right)
\]

We now want to prove a duality theorem relating $R\Gamma(X,F)_\RR$ and $R\Gamma_{c,\cal{D}}(X,F^D_\RR)$. Let us construct a natural pairing
\[
R\Gamma(X,F)_\RR \otimes^L R\Gamma_{c,\cal{D}}(X,F^D_\RR) \to \RR[0]
\]
Consider the following diagram, with the right rectangle commutative:
\begin{equation}\label{pairing_deligne}\begin{tikzpicture}[baseline= (a).base]
\node[scale=.85] (a) at (0,0){
\begin{tikzcd}[ampersand replacement=\&]
{R\Gamma(X,F)_\RR \otimes^L R\Gamma_{c,\cal{D}}(X,F^D_\RR)} \& {R\Gamma(X,F)_\RR \otimes^L R\Gamma(X,F^D)_\RR} \& {R\Gamma(X,F)_\RR \otimes^L  R\Hom_{G_\RR,X(\CC)}(\alpha^\ast F,\RR[1])} \\
\& {R\Gamma(X,F)_\RR \otimes^L R\Hom(F,\ZZ^c_X)_\RR} \& {R\Gamma_{G_\RR}(X(\CC),\alpha^\ast F)_\RR\otimes^LR\Hom_{G_\RR,X(\CC)}(\alpha^\ast F,\RR[1])} \\
{R\Gamma_{c,\cal{D}}(X,\ZZ^D_\RR)} \& {R\Gamma(X,\ZZ^c_X)_\RR} \& {R\Gamma_{G_\RR}(X(\CC),\RR[1])}
\arrow[Rightarrow, no head, from=1-2, to=2-2]
\arrow[from=1-1, to=1-2]
\arrow[from=1-2, to=1-3]
\arrow[from=2-2, to=3-2]
\arrow[dashed, from=1-1, to=3-1]
\arrow[from=3-1, to=3-2]
\arrow[from=3-2, to=3-3]
\arrow[from=1-3, to=2-3]
\arrow[from=2-3, to=3-3]
\end{tikzcd}};
\end{tikzpicture}\end{equation}
The top and bottom rows are fiber sequences, so we obtain an induced dotted map making the diagram commute. To obtain the desired pairing, it remains to compute $R\Gamma_{c,\cal{D}}(X,\ZZ^D_\RR)$. The complex $R\Gamma(X,\ZZ^c_X)=R\Gamma(X,\GG_X)$ is torsion in degree $i\geq 1$ \cite[II.6.2]{ADT}, and $R\Gamma_{G_\RR}(X(\CC),\RR[1])=(\displaystyle\prod_{\sigma : K \to \CC} \RR)^{G_\RR} [1]$. The cohomology in degree $-1$ and $0$ is given by the exact sequence
\[
0 \to H^{-1}_{c,\cal{D}}(X,\ZZ^D_\RR) \to \mathrm{CH}_0(X,1)_\RR \xrightarrow{\mathrm{Log}} (\displaystyle\prod_{\sigma : K \to \CC} \RR)^{G_\RR} \to H^0_{c,\cal{D}}(X,\ZZ^D_\RR) \to \mathrm{CH}_0(X)_\RR=0
\]
Denote $f:X\to \Spec(\ZZ)$ the structure map of $X$, $Z$ its singular locus and $U$ its regular locus. By \cite[thm. 4.4]{AMorin21}\footnote{Which also holds for singular schemes; this can be seen either by applying it on $\Spec(\ZZ)$ to the direct image with compact support along the structural morphism, or by modifying slightly the proof to reduce to the regular case by removing the singular points}, we have\footnote{This is essentially the finiteness of the class number of a number field}
\[
\mathrm{CH}_0(X)_\RR=\Hom_X(\ZZ,\ZZ^c_X)_\RR\xrightarrow{\simeq} \Hom(H^0_c(X,\ZZ)_\RR,\RR)=0
\]
Moreover, \cite[Theorem 4.4]{AMorin21} and the computation of $H^1(X,\ZZ)$ gives a short exact sequence
\[
0 \to \mathrm{CH}_0(X,1)_\RR \xrightarrow{\mathrm{Log}} (\displaystyle\prod_{\sigma : K \to \CC} \RR)^{G_\RR} \times \prod_{v\in Z}\prod_{\pi(w)=v} \RR \xrightarrow{\Sigma \times \prod_{v\in Z}\Sigma} \RR \times \prod_{v\in Z} \RR \to 0
\]
where $\Sigma$ denote sum maps. Consider the following snake diagram:
\[
\begin{tikzcd}
&
& 
& 0 \ar{r} \arrow[d, hook]
& H^{-1}_{c,\cal{D}}(X,\ZZ^D_\RR) \arrow[d, hook]   
&
&\\
& 0 \ar{r}
& 0 \ar{r} \ar{dd}
& \mathrm{CH}_0(X,1)_\RR \arrow[r, equal] \ar{dd}
& \mathrm{CH}_0(X,1)_\RR \ar{r} \ar{dd}
&  0
& ~\\
&
&
& ~
&
&
\ar[r, phantom, ""{coordinate, name=Y}]
&~\\
~& \ar[l, phantom, ""{coordinate, name=Z}] 0 \ar{r}
& \displaystyle\prod_{v\in Z}\prod_{\pi(w)=v} \RR \ar{r} \arrow[d, equal]
& \displaystyle(\displaystyle\prod_{\sigma : K \to \CC} \RR)^{G_\RR} \times \displaystyle\prod_{v\in Z}\displaystyle\prod_{\pi(w)=v} \RR \ar{r} \arrow[d, two heads,"\Sigma \times \prod_{v\in Z}\Sigma"]
& \displaystyle(\displaystyle\prod_{\sigma : K \to \CC} \RR)^{G_\RR} \ar{r} \arrow[d, two heads,]
& 0
& \\
&
& \ar[from=uuuurr, crossing over, rounded corners,
to path=
{ -- ([xshift=2ex]\tikztostart.east)
	-| (Y) [near end]\tikztonodes
	-| (Z) [near end]\tikztonodes
	|- ([xshift=-2ex]\tikztotarget.west)
	-- (\tikztotarget)}
]\displaystyle\prod_{v\in Z}\displaystyle\prod_{\pi(w)=v} \RR \ar{r}{(0,\prod_{v\in Z}\Sigma)}
& \RR \times \displaystyle\prod_{v\in Z} \RR \ar{r}
& H^0_{c,\cal{D}}(X,\ZZ^D_\RR) \ar{r} 
& 0
&
\end{tikzcd}
\]
If we denote by $(\prod_I \RR)^\Sigma$ the kernel of the sum map $\prod_I \RR \to \RR$, we obtain identifications
\[
H^{-1}_{c,\cal{D}}(X,\ZZ^D_\RR)= \prod_{v\in Z}(\prod_{\pi(w)=v}\RR)^\Sigma,~~ H^0_{c,\cal{D}}(X,\ZZ^D_\RR)= \RR
\]
and under the latter identification the map $(\displaystyle\prod_{\sigma : K \to \CC} \RR)^{G_\RR} \to H^0_{c,\cal{D}}(X,\ZZ^D_\RR)$ is simply the restriction of $\Sigma$. If $X$ is unibranch, and in particular if it is regular, we have $H^{-1}_{c,\cal{D}}(X,\ZZ^D_\RR)=0$.

The desired pairing is now
\[
R\Gamma(X,F)_\RR \otimes^L R\Gamma_{c,\cal{D}}(X,F^D_\RR) \to R\Gamma_{c,\cal{D}}(X,\ZZ^D_\RR)\xrightarrow{\tau^\geq 0} \RR[0]
\]

\begin{thm}\label{duality_theorem}
	Let $F\in D^b(X_{et})$ be a bounded complex with $\ZZ$-constructible cohomology groups. The pairing
	\[
	R\Gamma(X,F)_\RR \otimes^L R\Gamma_{c,\cal{D}}(X,F^D_\RR) \to \RR[0]
	\]
	is a perfect pairing of perfect complexes of $\RR$-vector spaces. If $F\in D^+(X_{et})$, the map
	\[
	R\Gamma_{c,\cal{D}}(X,F^D_\RR)\to R\Hom(R\Gamma(X,F)_\RR,\RR)
	\]
	is an isomorphism.
\end{thm}
\begin{proof}
	The map $R\Gamma_{c,\cal{D}}(X,(-)^D_\RR)\to R\Hom(R\Gamma(X,-)_\RR,\RR)$ is a natural transformation between exact functors which commute with filtered colimits. Any complex $F\in D^+(X_{et})$ is a filtered colimit of bounded complexes; a bounded complex has a filtration by truncations with graded pieces shifts of sheaves; and any sheaf is a filtered colimit of $\ZZ$-constructible sheaves. We thus reduce to the case of a single $\ZZ$-constructible sheaf. The groups $H^i(X,F)$ differ from $\hat{H}^i_c(X,F)$ by a finite group since $X$ is proper, hence are finite type for $i=0,1$ and torsion otherwise, so $R\Gamma(X,F)_\RR$ is a perfect complex of $\RR$-vector spaces.
	
	We now prove the perfectness of the complex $R\Gamma_{c,\cal{D}}(X,F^D_\RR)$ and of the pairing by Artin induction.
	\begin{itemize}
		\item This is trivial if $F$ is constructible.
		
		\item If $X$ is regular and $F=\ZZ$, we have $R\Gamma(X,\ZZ)_\RR=\RR[0]$ and $R\Gamma_{c,\cal{D}}(X,\ZZ^D_\RR)=\RR[0]$ so the pairing is of the form $\RR \otimes \RR \to \RR$, concentrated in degree $0$. By construction, we have a commutative square
		\[\begin{tikzcd}[ampersand replacement=\&]
		{R\Gamma(X,\ZZ)_\RR\otimes^L R\Gamma_{G_\RR}(X(\CC),\RR)} \& {R\Gamma(X,\ZZ)_\RR\otimes^LR\Gamma_{c,\cal{D}}(X,\ZZ^D_\RR)} \\
		{R\Gamma_{G_\RR}(X(\CC),\RR)} \& {R\Gamma_{c,\cal{D}}(X,\ZZ^D_\RR)}
		\arrow[from=1-1, to=2-1]
		\arrow[from=1-1, to=1-2]
		\arrow[from=2-1, to=2-2]
		\arrow[from=1-2, to=2-2]
		\end{tikzcd}\]
		so we get a commutative diagram
		\[\begin{tikzcd}[ampersand replacement=\&]
		{\RR\otimes(\displaystyle\prod_{\sigma : K \to \CC} \RR)^{G_\RR}} \& {\RR\otimes \RR} \\
		{(\displaystyle\prod_{\sigma : K \to \CC} \RR)^{G_\RR}} \& \RR
		\arrow[from=1-1, to=2-1]
		\arrow["{\mathrm{id}\otimes \Sigma}", two heads, from=1-1, to=1-2]
		\arrow["{\Sigma}", two heads, from=2-1, to=2-2]
		\arrow[from=1-2, to=2-2]
		\end{tikzcd}\]
		Since the pairing on the left is multiplication, the pairing on the right is also multiplication, so it is perfect.
		
		\item If $F$ is supported on a closed subscheme, without loss of generality we can suppose that there is a closed point $i\colon x \to X$ and a finite type $G_x$-module $M$ such that $F=i_\ast M$. We have
		\begin{align*}
			R\Gamma(X,i_\ast M)&\simeq R\Gamma(G_x,M)\\
			R\Gamma_{c,\cal{D}}(X,(i_\ast M)^D_\RR)\xrightarrow{\simeq} R\Gamma(X,(i_\ast M)^D)_\RR&=R\Gamma(X,i_\ast M^\vee)_\RR \simeq R\Gamma(G_x,M^\vee)_\RR
		\end{align*}
		The counit $i_\ast Ri^!\ZZ^c_X=i_\ast \ZZ \to \ZZ^c_X$ induces a map $R\Gamma(G_x,\ZZ)=R\Gamma(X,(i_\ast \ZZ)^D)\to R\Gamma(X,\ZZ^c_X)$. Since $\alpha^\ast i_\ast =0$, this induces a map $R\Gamma(G_x,\ZZ)_\RR \to R\Gamma_{c,\cal{D}}(X,\ZZ^D_\RR)$. Upon identifying $\ZZ$ with $\ZZ^\vee$, this map is the map $R\Gamma_{c,\cal{D}}(X,(i_\ast \ZZ)^D_\RR) \to R\Gamma_{c,\cal{D}}(X,\ZZ^D_\RR)$ induced by $\ZZ \to i_\ast \ZZ$. Denote $j:U:=X\backslash x \to X$ the open immersion; the short exact sequence $0\to j_!\ZZ \to \ZZ \to i_\ast\ZZ \to 0$ gives a $3\times 3$ diagram:
		\[\begin{tikzcd}[ampersand replacement=\&]
			{} \& {R\Gamma_{c,\cal{D}}(X,(j_! \ZZ)^D_\RR[-1])} \& {R\Gamma_{c,\cal{D}}(X,(i_\ast\ZZ)_\RR^D)} \& {R\Gamma_{c,\cal{D}}(X,\ZZ^D_\RR)} \\
			{} \& {R\Gamma(U,\ZZ^c_U[-1])_\RR} \& {R\Gamma(X,i_\ast\ZZ)_\RR} \& {R\Gamma(X,\ZZ^c_X)_\RR} \\
			{} \& {R\Gamma_{G_\RR}(X(\CC),\RR)} \& 0 \& {R\Gamma_{G_\RR}(X(\CC),\RR[1])}
			\arrow[Rightarrow, no head, from=1-3, to=2-3]
			\arrow[from=1-2, to=2-2]
			\arrow[from=1-2, to=1-3]
			\arrow[from=2-2, to=2-3]
			\arrow[from=2-3, to=3-3]
			\arrow[from=2-2, to=3-2]
			\arrow[from=3-2, to=3-3]
			\arrow[from=2-3, to=2-4]
			\arrow[from=1-3, to=1-4]
			\arrow[from=1-4, to=2-4]
			\arrow[from=3-3, to=3-4]
			\arrow[from=2-4, to=3-4]
		\end{tikzcd}\]
		The map $R\Gamma_{c,\cal{D}}(X,(i_\ast\ZZ)^D_\RR) \to R\Gamma_{c,\cal{D}}(X,\ZZ^D_\RR)$ is given in degree $0$ by a map $H^0(G_x,\ZZ)_\RR=\RR \to H^0_{c,\cal{D}}(X,\ZZ^D_\RR)=\RR$ which is computed as the boundary map $\delta$ in the following snake diagram
		\[
		\begin{tikzcd}
			& 0 \ar{r}
			& H^{-1}_{c,\cal{D}}(X,\ZZ^D_\RR) \ar{r} \arrow[d, hook]
			& H^{-1}_{c,\cal{D}}(X,(j_!\ZZ)^D_\RR) \ar{r} \arrow[d, hook]
			& \RR \arrow[d, equal] 
			&
			&\\
			& 0 \ar{r}
			& \mathrm{CH}_0(X,1)_\RR \ar{r} \ar{dd}
			& \mathrm{CH}_0(U,1)_\RR \arrow[r, "\ord_v"] \ar{dd}
			& \RR \ar{r} \ar{dd}
			&  0
			& ~\\
			&
			&
			& ~
			&
			&
			\ar[r, phantom, ""{coordinate, name=Y}]
			&~\\
			~& \ar[l, phantom, ""{coordinate, name=Z}] 0 \ar{r}
			& (\displaystyle\prod_{\sigma : K \to \CC} \RR)^{G_\RR} \arrow[r, equal]  \arrow[d, two heads,"\Sigma"]
			& (\displaystyle\prod_{\sigma : K \to \CC} \RR)^{G_\RR} \ar{r} \arrow[d, two heads]
			& 0 \ar{r} \ar{d}
			& 0
			& \\
			&
			& \ar[from=uuuurr, "\delta", crossing over, rounded corners,
			to path=
			{ -- ([xshift=2ex]\tikztostart.east)
				-| (Y) [near end]\tikztonodes
				-| (Z) [near end]\tikztonodes
				|- ([xshift=-2ex]\tikztotarget.west)
				-- (\tikztotarget)}
			] H^0_{c,\cal{D}}(X,\ZZ^D_\RR)=\RR \ar{r}
			& H^0_{c,\cal{D}}(X,(j_!\ZZ)^D_\RR) \ar{r}
			& 0 \ar{r} 
			& 0
			&
		\end{tikzcd}
		\]
		Let $f\in \mathrm{CH}_0(U,1)_\RR$ such that $\ord_v(f)=1$; we also have $\ord_w(f)=0$ for $w\in U$ by definition. Then \[\delta(1)=\Sigma\circ\mathrm{Log}(f)=-\log N(v)\] by the product formula. From \cref{pairing_deligne}, we obtain the commutativity of the upper left rectangle in the following commutative diagram.
		\[\begin{tikzcd}[ampersand replacement=\&]
			{R\Gamma(X,i_\ast M)_\RR \otimes^L R\Gamma_{c,\cal{D}}(X,(i_\ast M)^D_\RR)} \& {R\Gamma_{c,\cal{D}}(X,\ZZ^D_\RR)} \& {\RR[0]} \\
			{R\Gamma(X,i_\ast M)_\RR \otimes^L R\Gamma(X,(i_\ast M)^D)_\RR} \& {R\Gamma(X,\ZZ^c_X)_\RR} \\
			{R\Gamma(G_x,M)_\RR\otimes^L R\Gamma(G_x, M^\vee)_\RR} \& {R\Gamma(G_x,\ZZ)_\RR} \& {\RR[0]}
			\arrow[from=1-1, to=1-2]
			\arrow["\simeq"', from=1-1, to=2-1]
			\arrow[from=1-2, to=2-2]
			\arrow[from=2-1, to=2-2]
			\arrow["\simeq", from=3-1, to=2-1]
			\arrow[from=3-1, to=3-2]
			\arrow[from=3-2, to=2-2]
			\arrow["{\tau^\geq 0}", from=1-2, to=1-3]
			\arrow["{\tau^\geq 0}", from=3-2, to=3-3]
			\arrow["{-\log N(v)}"', from=3-3, to=1-3]
			\arrow[bend right = 60, from=3-2, to=1-2]
		\end{tikzcd}\]

		The perfectness of the pairing reduces to that of the natural pairing
		\[
		R\Gamma(G_x,M)_\RR\otimes^LR\Gamma(G_x,M^\vee)_\RR \to \RR
		\]
		coming from $M\otimes^L M^\vee \to \ZZ$. It was shown in \cite[Paragraph 4.2.4]{AMorin21} that this latter pairing is perfect\footnote{This also reduces by Artin induction, after showing compatibility with induction, to the perfectness for $M=\ZZ$.}.
		
		\item Let $Y=\Spec(\cal{O}')$ be the spectrum of an order in a number field with a finite dominant morphism $\pi:Y\to X$, and suppose $F=\pi_\ast G$ for a $\ZZ$-constructible sheaf $G$ on $Y$. Denote $\pi':Y(\CC)\to X(\CC)$ the induced morphism. The counit $\pi_\ast \ZZ^c_Y \to \ZZ^c_X$ is sent under $\alpha^\ast$ to the counit $\pi'_\ast \overline{\QQ}^\times[1] \to \overline{\QQ}^\times[1]$, and we have the canonical maps $\overline{\QQ}^\times[1] \to \CC^\times[1] \xrightarrow{\log} \RR[1]$. We obtain formally a morphism $R\Gamma_{c,\cal{D}}(Y,\ZZ^D_\RR) \to R\Gamma_{c,\cal{D}}(X,\ZZ^D_\RR)$ and isomorphisms $R\Gamma_{c,\cal{D}}(Y,G^D_\RR) \xrightarrow{\simeq} R\Gamma_{c,\cal{D}}(X,F^D_\RR)$ making the following diagram of pairing commute
		\[\begin{tikzcd}[ampersand replacement=\&]
		{R\Gamma(Y,G)_\RR\otimes^LR\Gamma_{c,\cal{D}}(Y,G^D_\RR)} \& {R\Gamma_{c,\cal{D}}(Y,\ZZ^D_\RR)} \& {\RR[0]} \\
		{R\Gamma(X,F)_\RR\otimes^LR\Gamma_{c,\cal{D}}(X,F^D_\RR)} \& {R\Gamma_{c,\cal{D}}(X,\ZZ^D_\RR)} \& {\RR[0]}
		\arrow[from=1-2, to=2-2]
		\arrow[from=1-1, to=1-2]
		\arrow[from=2-1, to=2-2]
		\arrow["\simeq", from=1-1, to=2-1]
		\arrow["{\tau{\geq 0}}", from=1-2, to=1-3]
		\arrow["{\tau^{\geq 0}}", from=2-2, to=2-3]
		\arrow[from=1-3, to=2-3]
		\end{tikzcd}\]
		The rightmost map is determined by the following commutative diagram (coming from the defining long exact cohomology sequences)
		\[\begin{tikzcd}[ampersand replacement=\&]
			{(\displaystyle\prod_{\tau:K(Y)\to\CC} \RR)^{G_\RR}} \& {H^0_{c,\cal{D}}(Y,\ZZ^D_\RR)=\RR} \\
			{(\displaystyle\prod_{\sigma:K\to\CC} \RR)^{G_\RR}} \& {H^0_{c,\cal{D}}(X,\ZZ^D_\RR)=\RR}
			\arrow[from=1-2, to=2-2]
			\arrow["\Sigma",two heads, from=2-1, to=2-2]
			\arrow["\Sigma",two heads, from=1-1, to=1-2]
			\arrow[ from=1-1, to=2-1]
		\end{tikzcd}\]
		where the left map sums components corresponding to embeddings of $K(Y)$ restricting to the same embedding of $K$. Thus the rightmost map is the identity and the pairing for $F$ is perfect if and only if the pairing for $G$ is perfect.
	\end{itemize}
\end{proof}

\begin{rmk}
	Using duality for Deligne cohomology, it should be possible to reduce the theorem to the duality theorem \cite[thm. 4.4]{AMorin21} similarly to how we reduced Artin-Verdier duality for $F^D$ to Artin-Verdier duality for $F$.
\end{rmk}
\section{The \texorpdfstring{$L$}{L}-function of \texorpdfstring{$F^D$}{F D}}
For any scheme $S$, we will denote $\nu:\Sh(S_{proet})\to \Sh(S_{et})$ the natural morphism of topoi; the left adjoint $\nu^\ast$ is fully faithful \cite[5.1.2]{Bhatt15}.

\begin{defi}
	For each closed point $x$ of $X$, let $l_x$ be a prime number coprime to the residual characteristic at $x$ and $\varphi$ be the \emph{geometric} frobenius in $G_x$. Denote $-\widehat{\otimes}\QQ_l=(R\lim (-\otimes^L\ZZ/l^n\ZZ))\otimes \QQ$ the \emph{completed} tensor product with $\QQ_l$, with the derived limit computed on the proétale site. 
	
	Let $F$ be a $\ZZ$-constructible sheaf on $X$. We define the $L$-function of $F^D$ by the Euler product
	\[
	L_X(F^D,s)=\prod_{x\in X_0} \det\left(I-N(x)^{-s}\varphi|\left(\nu^\ast(i_x^\ast F^D)\right)\widehat{\otimes}\QQ_{l_x}\right)^{-1}
	\]
\end{defi}
We will show that the $L$-function doesn't depend on the choice of the prime numbers $(l_x)_{x\in X_0}$.

\subsection{Explicit computation}

In this section we compute more explicitely the $L$-function of $F^D$. Denote $\pi:Y\to X$ the normalization of $X$ and let $i:x \to X$ be a closed point. Denote $\bar{x}=\Spec(\kappa(x)^{sep})$, $G_x=\Gal(\kappa(x)^{sep}/\kappa(x))$, $\cal{O}_x^h$ (resp. $\cal{O}_x^{sh}$) the henselian (resp. strict henselian) local ring at $x$. For each closed point $y\in Y$ above $x$, we will consider similarly $G_y=\Gal(\kappa(y)^{sep}/\kappa(y))$, $\cal{O}_y^h$, $\cal{O}_y^{sh}$, and moreover $K_y^h=\mathrm{Frac}(\cal{O}_y^h)$ (resp. $K_y^{sh}$) the henselian local field at $y$ (resp. its maximal unramified extension). Fix an embedding $K_y^{h} \hookrightarrow K^{sep}$. This determines an inertia group $I_y$ inside $G_K$ which is the absolute Galois groups of $K_y^{sh}$. If $G$ is a topological group we will denote $\underline{G}$ the associated condensed group.

Let us fix some conventions. If $L(M,s)$ is an $L$-function attached to some object $M$\footnote{which kind of $L$-function we will try to make clear each time from the context}, defined by an Euler product over closed points of $X$, we will denote $L_x(M,s)$ for the local factor at a closed point $x\in X$. If $N$ is a discrete $G_x$-module with free of finite type underlying abelian group, or a rational or $l$-adic $G_x$-representation of finite dimension, denote
\[
Q_x(N,s):= \det(I-N(x)^{-s}\varphi| N)^{-1}
\]
with $\varphi$ the geometric Frobenius at $x$.

\begin{thm}\label{explicit_expression_L_function}
	Let $F$ be a $\ZZ$-constructible sheaf on $X$. The local factor at $x$ of the $L$-function of $F^D$ is
	\[
	L_x(F^D,s)=\frac{\displaystyle\prod_{\pi(y)=x}Q_y((F_\eta^\vee)^{I_y},s+1)Q_x(F_x^\vee,s)}{\displaystyle\prod_{\pi(y)=x}Q_y((F_\eta^\vee)^{I_y},s)}
	\]
\end{thm}
The proof is divided into the following several lemmas below. We mention first some important consequences:

\begin{cor}\label{L_function_pushforward}
	\hspace{2em}
	\begin{itemize}
		\item If $F$ is locally constant around the regular closed point $v$, then $I_v$ acts trivially on $F_\eta$, $F_\eta=F_v$ and we find that the local factor at $v$ of $F^D$ equals the local factor at $v$ of the Artin $L$-function of $F_{\eta}\otimes \QQ$ at $s+1$: \footnote{Notice that a rational representation of a finite group is self-dual because its character takes real values hence is its own conjugate}
		\[
		L_v(F^D,s)=L_v(F_{\eta}\otimes \QQ,s+1)
		\]
		\item Each local factor is well-defined, independently of the choice of a prime number $l$ coprime to the residual characteristic.
		\item The $L$-function of $F^D$ differs from the Artin $L$-function $L_K(F_{\eta}\otimes \QQ,s+1)$ by a finite number of factors and is thus well-defined
		\item $L_X(F^D,s)$ is meromorphic.
	\end{itemize}
\end{cor}

Denote $T$ the torus over $K$ with character group $Y:=F_\eta/tor$. For any prime $l$, define the rational $l$-adic Tate module of $T$ as
\[
V_l(T):=(\lim T[l^n])\otimes \QQ
\]
It is a finite dimensional $l$-adic representation of $G_K$. We know that
\[
V_l(T) \simeq Y^\vee \otimes \QQ_l(1) = F_\eta^\vee \otimes \QQ_l(1)
\]
as $l$-adic representations. 

\begin{defi}[{\cite[§ 8]{Fontaine92}}]
	The $L$-function $L_K(T,s)$ of $T$ is the $L$-function of the $1$-motive $[0 \to T]$ over $K$, defined by the Euler product:
	\[
	L_K(T,s)=\prod_{x\in X_0}\det(I-N(x)^{-s}\varphi|V_{l_x}(T)^{I_x})^{-1}
	\]
	where $l_x$ is a prime number coprime to the residual characteristic at $x$. 
\end{defi}
By the above, the $L$-function of $T$ is also the Artin $L$-function at $s+1$ of $Y\otimes \QQ=F_\eta \otimes \QQ$, and it doesn't depend on the choice of the family $(l_x)$.
\begin{prop}
	Suppose that $v$ is a regular closed point. Let $M$ be a discrete $D_v$-module of finite type. There is a canonical $G_v$-equivariant isomorphism
	\[
	M^{I_v}\otimes \QQ \xrightarrow{\simeq} M_{I_v}\otimes \QQ
	\]
\end{prop}
\begin{proof}
	The action of $D_v$ on $M$ factors through a finite quotient $G$; denote $H$ the image of $I_v$ in $G$. Then the canonical composite map
	\[
	f:M^H \to M \to M_H
	\]
	is $G$-equivariant hence also $G/H$-equivariant, and if $N$ denotes the norm $\sum_{h\in H}h$ we have $fN=[H]\mathrm{Id}$ and $Nf=[H]\mathrm{Id}$. The result follows.
\end{proof}

\begin{cor}
	\begin{itemize}
		\item[]
		\item Suppose that the canonical map $F_v\to F_\eta^{I_v}$ is an isomorphism for a regular closed point $v$. Then the local factor at $v$ of $F^D$ is $L_v(F_\eta\otimes \QQ,s+1)=L_v(T,s)$.
		\item Suppose that $X$ is regular. We have $L_X((g_\ast Y)^D,s)=L_K(T,s)=L_K(Y\otimes \QQ,s+1)$.
		\item Suppose that $X$ is regular and denote $\cal{T}^0$ the connected (lft) Néron model of $T$ on $X$, seen as an étale sheaf on $X$. We can define an $L$-function $L_X(\cal{T}^0,s)$ for $\cal{T}^0$ with local factor at $i:x\to X$ given by $Q_x((\nu^\ast i^\ast \cal{T}^0)\widehat{\otimes} \QQ_l,s)$. Then
		\[
		L_X(\cal{T}^0,s)=L_K(T,s)^{-1}
		\]
	\end{itemize}
\end{cor}

\begin{proof}
	We have a canonical $G_v$-equivariant isomorphism: 
	\[
	(F_\eta^\vee)^{I_v}\otimes \QQ=\Hom_{I_v}(F_\eta\otimes \QQ, \QQ)=\Hom_{Ab}((F_\eta)_{I_v}\otimes \QQ, \QQ) \simeq \Hom_{Ab}(F_\eta^{I_v}\otimes \QQ, \QQ)\simeq F_v^\vee \otimes \QQ
	\]
	hence the local factor at $v$ of the $L$-function of $F^D$ simplifies to $L_v(F_\eta\otimes \QQ,s+1)$.
	
	The second point is an immediate consequence of the first. Let us prove the last point. We have \cite[2.2]{Geisser21}
	\[
	\cal{T}^0=R\cal{H}om_X(\tau^{\leq 1}Rg_\ast Y,\GG_m)
	\]
	as étale sheaves. Since $R^1 g_\ast Y$ is skyscraper with finite stalks (because $Y$ is free as an abelian group), $R^1g_\ast Y$ is killed by an integer $N$ and thus
	\[
	(\nu^\ast i^\ast \cal{T}^0)\widehat{\otimes} \QQ_l= (\nu^\ast i^\ast (g_\ast Y)^D)\widehat{\otimes} \QQ_l[-1].
	\]
	The result follows.
\end{proof}

We now prove \cref{explicit_expression_L_function}.
\begin{lem}
	Let $F$ be an étale sheaf on $X$, and let $i:x \to X$ be a closed point. There is a fiber sequence
	\[
	R\cal{H}om_x(F_x,\ZZ) \to i^\ast F^D \to  \prod_{\pi(y)=x} \mathrm{ind}^{G_x}_{G_y} R\Hom_{\Spec(K_y^{sh})}(F_\eta,\GG_m[1])
	\]
\end{lem}

\begin{rmk}
	We abused notation by identifying what should be the right term with its underlying complex of abelian groups.
\end{rmk}

\begin{proof}
	Denote $X(x):=\Spec(\cal{O}_x^{h})$, $\eta(x):=\Spec(\cal{O}_x^{h})\times_X \Spec(K)$ and $f:X(x)\to X$, $g:\eta \to X$ the canonical morphisms. We consider the cartesian diagram
	\[\begin{tikzcd}[ampersand replacement=\&]
	{\eta(x)} \& {X(x)} \\
	\eta \& X
	\arrow["g", from=2-1, to=2-2]
	\arrow["f", from=1-2, to=2-2]
	\arrow["{f'}"', from=1-1, to=2-1]
	\arrow["{g'}", from=1-1, to=1-2]
	\end{tikzcd}\]
	Denote again $i:x\to X(x)$ the closed immersion. Since $f$ is (ind-)étale, we have 
	\[
	i^\ast F^D=i^\ast f^\ast R\cal{H}om_X(F,\GG_X[1])=i^\ast R\cal{H}om_{X(x)}(F,\GG_{X(x)}[1])
	\]
	By \cref{generic_point_strict_henselization}, we have
	\[
	\eta(x)=\coprod_{{\pi(y)=x}}\coprod_{\Gal(\kappa(y)/\kappa(x))}\Spec(K_y^{sh})
	\]
	and thus
	\[
	(Rg'_\ast \GG_m)_{\bar{x}}=R\Gamma(\eta(x),\GG_m)=\prod_{\pi(y)=x} \mathrm{ind}^{G_x}_{G_y} R\Gamma(K_y^{sh},\GG_m)=\prod_{\pi(y)=x} \mathrm{ind}^{G_x}_{G_y} (K_y^{sh})^{\times}
	\]
	hence $g'_\ast \GG_m=Rg'_\ast \GG_m$. Moreover there is by definition a fiber sequence
	\[
	\GG_{X(x)} \to g'_\ast \GG_m \to i_\ast \ZZ
	\]
	so we obtain a fiber sequence
	\[
	R\cal{H}om_x(F_x,\ZZ) \to i^\ast F^D \to i^\ast Rg'_\ast R\cal{H}om_{\eta(x)}(g'^\ast F,\GG_m[1])
	\]
	Using again \cref{generic_point_strict_henselization}, we see that the right term is the following complex of abelian groups with its natural $G_x$-module structure:
	\begin{align*}
		(Rg'_\ast R\cal{H}om_{\eta(x)}(g'^\ast F,\GG_m[1]))_{\bar{x}} = \prod_{\pi(y)=x} \mathrm{ind}^{G_x}_{G_y} R\Hom_{\Spec(K_y^{sh})}(F_\eta,\GG_m[1])
	\end{align*}
\end{proof}

From now on, we suppose that $F$ is a $\ZZ$-constructible sheaf on $X$.
\begin{lem}
	We have
	\[
	(\nu^\ast R\cal{H}om_x(F_x,\ZZ))\widehat{\otimes} \QQ_l = (\nu^\ast \cal{H}om_x(F_x,\ZZ))\otimes \QQ_l
	\]
\end{lem}
\begin{proof}
	We have
	\[
	\nu^\ast R\cal{H}om_x(F_x,\ZZ)=R\cal{H}om_{x_{proet}}(\nu^\ast F_x,\ZZ)
	\]
	The functor $-\otimes^L \ZZ/l^n\ZZ$ is the cofiber of the map $(-) \xrightarrow{l^n} (-)$ so it commutes with exact functors, and we find
	\[
	R\cal{H}om_{x_{proet}}(\nu^\ast F_x,\ZZ) \widehat{\otimes} \ZZ_l = R\cal{H}om_{x_{proet}}(\nu^\ast F_x, R\lim \ZZ/l^n\ZZ)
	\]
	The constant sheaf functor is exact so the transition maps $\ZZ/l^{n+1} \ZZ \to \ZZ/l^n\ZZ$ are surjective, hence we have $R\lim \ZZ/l^n\ZZ=\lim \ZZ/l^n\ZZ=:\ZZ_l$ \cite[3.1.10]{Bhatt15}.
	
	Under the identification $x_{proet}\simeq \underline{G_x}\text{-}\mathrm{Cond}(Set)$, the sheaf $\cal{E}xt^i_{x_{proet}}(\nu^\ast F_x, \ZZ_l)$ is  the condensed abelian group $\cal{E}xt^i_{\mathrm{Cond(Ab)}}(\nu^\ast F_x,\ZZ_l)$ with its natural $\underline{G_x}$-action; since both $\nu^\ast F_x$ and $\ZZ_l$ are locally compact abelian groups\footnote{The former is even discrete}, we find $\cal{E}xt^i_{x_{proet}}(\nu^\ast F_x, \ZZ_l)=0$ for $i\geq 2$ \cite[Remark after 4.9]{Clausen19}. Let us show that $\cal{E}xt^1_{x_{proet}}(\nu^\ast F_x, \ZZ_l)$ is killed by some integer $N$. We have $R\cal{H}om_{x_{proet}}(\nu^\ast F_x,\ZZ_l)=R\lim R\cal{H}om_{x_{proet}}(\nu^\ast F_x,\nu^\ast\ZZ/l^n\ZZ)=R\lim \nu^\ast R\cal{H}om_{x}(F_x,\ZZ/l^n\ZZ)$ so there is a short exact sequence
	\[
	0 \to R^1\lim \nu^\ast\cal{H}om_x(F_x,\ZZ/l^n\ZZ) \to \cal{E}xt^1_{x_{proet}}(\nu^\ast F_x, \ZZ_l) \to \lim \nu^\ast\cal{E}xt^1_x(F_x,\ZZ/l^n\ZZ) \to 0
	\]
	The underlying abelian group of the $G_x$-module $\cal{H}om_x(F_x,\ZZ/l^n\ZZ)$ is $\Hom_{Ab}(F_x,\ZZ/l^n\ZZ)$, so the transition maps $\cal{H}om_x(F_x,\ZZ/l^{n+1}\ZZ)\to \cal{H}om_x(F_x,\ZZ/l^n\ZZ)$ are eventually surjective: if $F_x$ is torsion-free this is clear, and if $F_x$ is torsion then the groups are constant for $n$ greater than the $l$-adic valuation of the order of $F_x$. Thus the left term is zero by \cref{mittag_leffler}. Let $k$ be the $l$-adic valuation of $(F_x)_{tor}$; then $\cal{E}xt^1_x(F_x,\ZZ/l^n\ZZ) = \cal{E}xt^1_x((F_x)_{tor},\ZZ/l^n\ZZ)$ is killed by $N=l^k$, and therefore so is $\cal{E}xt^1_{x_{proet}}(\nu^\ast F_x, \ZZ_l)$.
	
	From the previous point, it follows by tensoring with $\QQ$ that
	\[
	R\cal{H}om_{x_{proet}}(\nu^\ast F_x,\ZZ) \widehat{\otimes} \QQ_l = \cal{H}om_{x_{proet}}(\nu^\ast F_x,\ZZ_l)\otimes \QQ
	\]
	
	There are canonical maps
	\[
	 \cal{H}om_{x_{proet}}(\nu^\ast F_x,\ZZ_l)\otimes \QQ \to \cal{H}om_{x_{proet}}(\nu^\ast F_x,\QQ_l) \leftarrow \cal{H}om_{x_{proet}}(\nu^\ast F_x,\ZZ)\otimes \QQ_l
	\]
	We have $\mathrm{Ab}(\mathrm{Sh}(x_{proet})=\mathrm{Ab}(\mathrm{Sh}((BG_x)_{proet})=\underline{G_x}-\mathrm{Cond(Ab)}=\ZZ[\underline{G_x}]-\mathrm{Mod}$, the category of modules under the condensed ring $\ZZ[\underline{G_x}]$ \cref{bg_proet}. To conclude it suffices to show that both maps are isomorphisms. Let $G$ be a finite quotient of $G_x$. We check it first for $F_x=\ZZ[G]$. This is easily done using next lemma.
	
	By a standard argument it thus suffices to show that $\nu^\ast F_x$ is (globally) of finite presentation as a $\ZZ[\underline{G}]$-module for some finite quotient $G$ of $G_x$; this is clear as $F_x$ is discrete and of finite type as an abelian group.
\end{proof}

\begin{lem}
	Let $G$ be a finite quotient of $G_x$. Then $\ZZ[\underline{G}]$ is a left $\underline{G_x}$-module such that 
	\[
	\cal{H}om_{x_{proet}}(\ZZ[\underline{G}],-)=\ZZ[\underline{G}]\otimes -
	\]
\end{lem}
\begin{proof}
	We have $\ZZ[\underline{G}]=\underline{\ZZ[G]}=\oplus_G \underline{\ZZ}$, thus
	\[
	\cal{H}om_{x_{proet}}(\ZZ[\underline{G}],-)=\oplus_G 	\cal{H}om_{x_{proet}}(\ZZ,-) = \oplus_G \mathrm{Id} = \ZZ[\underline{G}]\otimes -
	\]
\end{proof}

Denote $R\cal{\Hom}_{Y_{proet}}$ the enriched $R\Hom$ on the proétale site of a scheme, a complex of condensed abelian groups with underlying complex of abelian groups $R\Hom_Y$.
\begin{lem}
	We have
	\[
	 (\nu^\ast R\Hom_{\Spec(K_y^{sh})}(F_\eta,\GG_m[1]))\widehat{\otimes} \ZZ_l = R\underline{\Hom}_{\Spec(K_y^{sh})_{proet}}(\nu^\ast F_\eta,\ZZ_l(1)[2]))
	\]
\end{lem}

\begin{proof}
We have 
\[
\nu^\ast R\Hom_{\Spec(K_y^{sh})}(F_\eta,\GG_m[1])=R\underline{\Hom}_{\Spec(K_y^{sh})_{proet}}(\nu^\ast F_\eta,\nu^\ast \GG_m[1])
\]
Thus we find
\[
\nu^\ast R\Hom_{\Spec(K_y^{sh})}(F_\eta,\GG_m[1])\widehat{\otimes} \ZZ_l = \underline{R\Hom}_{\Spec(K_y^{sh})_{proet}}(\nu^\ast F_\eta,(\nu^\ast \GG_m[1])\widehat{\otimes} \ZZ_l)
\]
Since $l$ is invertible on $\Spec(K_y^{sh})$, we have $\GG_m\otimes^L \ZZ/l^n\ZZ=\mu_{l^n}[1]$ and thus 
\[
(\nu^\ast \GG_m[1])\widehat{\otimes} \ZZ_l = R\lim \nu^\ast \mu_{l^n}[2]=\ZZ_l(1)[2]
\]
\end{proof}

\begin{lem}
	Let $H$ be an open normal subgroup of $I_y$ acting trivially on $F_\eta$, and let $G:=I_y/H$. Denote $R\underline{\Hom}_{\underline{G}}$ the enriched $R\Hom$ between condensed $\underline{G}$-modules. There is a fiber sequence
	\[
	R\underline{\Hom}_{\underline{G}}(\nu^\ast F_\eta,\ZZ_l(1))) \to R\underline{\Hom}_{\Spec(K_y^{sh})_{proet}}(\nu^\ast F_\eta,\ZZ_l(1))) \to R\underline{\Hom}_{\underline{G}}(\nu^\ast F_\eta,\ZZ_l))[-1]
	\]
\end{lem}
\begin{proof}
	We have $\Spec(K_y^{sh})_{proet}=(BI_y)_{proet}=\underline{I_y}\text{-}\mathrm{Cond}(Set)$ by \cref{bg_proet}. The right adjoint to the forgetful functor $D(\underline{I_y}\text{-}\mathrm{Cond}(Ab)) \to D(\underline{G}\text{-}\mathrm{Cond}(Ab))$ is $R\underline{\Gamma}(H,-):=R\underline{\Hom}_{\underline{H}}(\ZZ,-)$, so we have
	\[
	R\underline{\Hom}_{\Spec(K_y^{sh})_{proet}}(\nu^\ast F_\eta,\ZZ_l(1))= R\underline{\Hom}_{\underline{G}}(\nu^\ast F_\eta,R\underline{\Gamma}(H,\ZZ_l(1)))
	\]
	We compute
	\[
	R\underline{\Gamma}(H,\ZZ_l(1))=R\lim R\underline{\Gamma}(H,\nu^\ast\mu_{l^n}) = R\lim \nu^\ast R\Gamma(H,\mu_{l^n})
	\]
	Denote $F$ the finite extension of $K_y^{sh}$ corresponding to $H$. Since $l$ is prime to the residual characteristic of $K_y^{sh}$, all $l^n$-th roots of unity are contained in $K_y^{sh}$ and thus also in $H$. Hensel's lemma together with the short exact sequence $0 \to \cal{O}_{F}^\times \to F^\times \to \ZZ \to 0$ then gives $H^i(H,\mu_{l^n})=\mu_{l^n},\ZZ/l^n\ZZ$ for $i=0,1$, while $H^i(H,\mu_{l^n})=0$ for $i\geq 2$ because $I_y$ is of cohomological dimension $1$ and therefore also $H$ \cite[8.11 (b), 5.10]{Harari20}. Thus we have a fiber sequence of discrete $G$-modules
	\[
	\mu_{l^n} \to R\Gamma(H,\mu_{l^n}) \to \ZZ/l^n\ZZ [-1]
	\]
	The transition maps on both sides are surjective. Applying the fully faithful exact functor $\nu^\ast$, we can compute the $R\lim$ using that the transition maps are again surjective and we obtain a fiber sequence in $D(\underline{G}\text{-}\mathrm{Cond}(Ab))$:
	\[
	\ZZ_l(1) \to R\underline{\Gamma}(H,\ZZ_l(1)) \to \ZZ_l[-1]
	\]
\end{proof}

Since $K_y^{sh}$ has all $l^n$-th roots of unity, as $\underline{G}$-modules we have $\ZZ_l\simeq \ZZ_l(1)$.
\begin{lem}
	We have
	\begin{align*}
	R\underline{\Hom}_{\underline{G}}(\nu^\ast F_\eta,\ZZ_l)\otimes \QQ& = \nu^\ast \Hom_G(F_\eta,\QQ) \otimes_\QQ \QQ_l = \nu^\ast (F_\eta^\vee)^{I_y} \otimes \QQ_l\\
	R\underline{\Hom}_{\underline{G}}(\nu^\ast F_\eta,\ZZ_l(1))\otimes \QQ &= \nu^\ast (F_\eta^\vee)^{I_y} \otimes \QQ_l(1)
	\end{align*}
\end{lem}

\begin{proof}
	The proof of both statements is similar, so we treat only the first one. 
	
	We first show that $\underline{\Ext}^i_{\underline{G}}(\nu^\ast F_\eta,\ZZ_l)$ is killed by some integer $N$ for all $i\geq 1$. There is a spectral sequence giving short exact sequences
	\[
	0 \to R^1\lim \nu^\ast \Ext^{i-1}_G(F_\eta,\ZZ/l^n\ZZ) \to \underline{\Ext}^i_{\underline{G}}(\nu^\ast F_\eta,\ZZ_l) \to \lim \nu^\ast \Ext^{i}_G(F_\eta,\ZZ/l^n\ZZ) \to 0
	\]
	Denote $\cal{H}om_G$ is the internal Hom for discrete $G$-modules. There is moreover a spectral sequence giving a long exact seqence
	\[
	\cdots \to H^i(G,\cal{H}om_G(F_\eta,\ZZ/l^n\ZZ)) \to \Ext^{i}_G(F_\eta,\ZZ/l^n\ZZ) \to  H^{i-1}(G,\cal{E}xt^1_G(F_\eta,\ZZ/l^n\ZZ)) \to \cdots.
	\]
	For $i\geq 2$, since 
	\[
		\cal{H}om_G(F_\eta,\ZZ/l^n\ZZ)=\Hom_{Ab}(F_\eta,\ZZ/l^n\ZZ) ~~\text{and}~~ \cal{E}xt^1_G(F_\eta,\ZZ/l^n\ZZ)=\Ext^1_{Ab}(F_\eta,\ZZ/l^n\ZZ)\\
	\]are of finite type, both $H^i(G,\cal{H}om_G(F_\eta,\ZZ/l^n\ZZ))$ and $H^{i-1}(G,\cal{E}xt^1_G(F_\eta,\ZZ/l^n\ZZ))$ are finite and killed by $[G]$, thus $\Ext^{i}_G(F_\eta,\ZZ/l^n\ZZ)$ is killed by $[G]^2$. For $i=1$, $H^1(G,\cal{H}om_G(F_\eta,\ZZ/l^n\ZZ))$ is finite killed by $[G]$ and $H^{0}(G,\cal{E}xt^1_G(F_\eta,\ZZ/l^n\ZZ))$ is finite killed by $[(F_\eta)_{tor}]$, so $\Ext^{1}_G(F_\eta,\ZZ/l^n\ZZ)$ is finite killed by $[G][(F_\eta)_{tor}]$. Finally, for $i=0$ we have that $\Hom_G(F_\eta,\ZZ/l^n\ZZ)$ is a finite group.
	
	Suppose first that $i\geq 2$. Then $\Ext^{i-1}_G(F_\eta,\ZZ/l^n\ZZ)$ is a system of finite groups so it satisfies the Mittag-Leffler condition, thus $R^1\lim \nu^\ast \Ext^{i-1}_G(F_\eta,\ZZ/l^n\ZZ)=0$ by \cref{mittag_leffler,mittag_leffler_discrete}. On the other hand, $\Ext^{i}_G(F_\eta,\ZZ/l^n\ZZ)$ is a system of finite group killed by $[G]^2$, thus $\underline{\Ext}^i_{\underline{G}}(\nu^\ast F_\eta,\ZZ_l)=\lim \nu^\ast \Ext^{i}_G(F_\eta,\ZZ/l^n\ZZ)$ is killed by $[G]^2$. We now treat the case $i=1$. Let us show that $R^1\lim \nu^\ast \Hom_G(F_\eta,\ZZ/l^n\ZZ)=0$. Again this is a system of finite groups, which thus satisfies the Mittag-Leffler condition, and we can apply \cref{mittag_leffler,mittag_leffler_discrete}. It follows that $\underline{\Ext}^1_{\underline{G}}(\nu^\ast F_\eta,\ZZ_l)=\lim \nu^\ast \Ext^{1}_G(F_\eta,\ZZ/l^n\ZZ)$ is killed by $[G][(F_\eta)_{tor}]$.
	
	We deduce that
	\[
	R\underline{\Hom}_{\underline{G}}(\nu^\ast F_\eta,\ZZ_l)\otimes \QQ = \underline{\Hom}_{\underline{G}}(\nu^\ast F_\eta,\ZZ_l) \otimes \QQ
	\]
	
	There are canonical maps
	\[
	\underline{\Hom}_{\underline{G}}(\nu^\ast F_\eta,\ZZ_l) \otimes \QQ \to \underline{\Hom}_{\underline{G}}(\nu^\ast F_\eta,\QQ_l) \leftarrow \underline{\Hom}_{\underline{G}}(\nu^\ast F_\eta,\QQ) \otimes_\QQ \QQ_l = \nu^\ast \Hom_G(\nu^\ast F_\eta,\QQ) \otimes_\QQ \QQ_l
	\]
	Since $F_\eta$ is a finite type, hence finite presentation abelian group and $G$ is finite, $F_\eta$ is a finite presentation $G$-module so $\nu^\ast F_\eta$ is (globally) of finite presentation as a $\underline{G}$-module\footnote{Note that as $G$ is finite, $\ZZ[\underline{G}]=\underline{\ZZ[G]}$}. By a standard argument, to show that the canonical maps are isomorphisms we reduce to the case of $\ZZ[\underline{G}]$. But $\underline{\Hom}_{\underline{G}}(\ZZ[\underline{G}],-)$ is the forgetful functor $U:\underline{G}-\mathrm{Cond(Ab)} \to \mathrm{Cond(Ab)}$, thus everything follows from the identifications
	\[
	(U \ZZ_l)\otimes \QQ = \ZZ_l\otimes \QQ = \QQ_l = U \QQ_l = \QQ \otimes_\QQ \QQ_l = U\QQ \otimes_\QQ \QQ_l
	\]
	
	The second equality in each case follows from the identity $\Hom_G(F_\eta,\QQ)=(F_\eta^\vee)^{G}\otimes \QQ=(F_\eta^\vee)^{I_y}\otimes \QQ$.
\end{proof}

\subsection{Functoriality}

\begin{prop}
	Let $0\to F \to G \to H \to 0$ be a short exact sequence of $\ZZ$-constructible sheaves on $X$. Then
	\[
	L_X(G^D,s)=L_X(F^D,s)L_X(H^D,s)
	\]
\end{prop}
\begin{proof}
	If $i:v\to X$ is a closed point, the short exact sequence gives a fiber triangle
	\[
	(\nu^\ast i^\ast H^D)\widehat{\otimes}\QQ_l \to (\nu^\ast i^\ast G^D)\widehat{\otimes}\QQ_l \to (\nu^\ast i^\ast F^D)\widehat{\otimes}\QQ_l
	\]
	Local factors are multiplicative with respect to short exact sequnces, hence
	\[
	L_v(G^D,s)=L_v(F^D,s)L_v(H^D,s)
	\]
\end{proof}

\begin{prop}
Let $\pi:Y\to X$ be a finite morphism between spectra of orders in number fields and let $F$ be a $\ZZ$-constructible sheaf on $Y$. We have
\[
L_X((\pi_\ast F)^D,s)=L_Y(F^D,s)
\]
\end{prop}

\begin{proof}
	This follows readily from \cref{explicit_expression_L_function} using the compatibility of local factors with induction \cite[VII.10.4 (iv) and its proof]{NeukirchANT}.
\end{proof}

We mention the following consequence of \cite{AMorin21}:
\begin{prop}
	Let $X=\Spec(\cal{O})$ be the spectrum of an order in a number field $K$, with open subscheme $j:U \to X$. Let $\omega$ be the number of roots of units in $K$, $\Delta_{K}$ its discriminant, $r_1$ and $r_2$ respectively the number of real and complex places of $K$, and $R_U$ the regulator introduced in \cite{AMorin21} for irreducible $1$-dimensional arithmetic schemes. We have
	\[
	L_X^\ast((j_!\ZZ)^D,0)=\frac{2^{r_1}(2\pi)^{r_2}[\mathrm{CH}_0(U)]R_U}{\omega \sqrt{|\Delta_{K}|}}
	\]
\end{prop}

\begin{proof}
	Denote $f:V\to X$ the regular locus of $U$ seen as an open subscheme of $X$ and $Z$ its closed complement in $U$. Denote also $\pi:Y\to X$ the normalization of $X$, $f':V\to Y$ the inclusion and $K$ the function field of $X$ and $Y$. We have $f=\pi f'$ and $\pi$ induces an isomorphism between $\pi^{-1}(V)$ and $V$. Finally, denote $T$ the closed complement of $V$ in $Y$. We have for the $L$-functions of $\ZZ$-constructible sheaves:
	\[
	L_X(j_!\ZZ)=\zeta_U=\zeta_Y \frac{\prod_{v\in Z} \zeta_v}{\prod_{w\in T}\zeta_w}
	\]
	Hence by \cite[thm. C]{AMorin21}, we get the relation
	\[
	\frac{[\mathrm{CH}_0(U)]R_U}{\omega}=\frac{h_K R_K\prod_{w\in T}\log N(w)}{\omega \prod_{v\in Z}\log N(v)}
	\]
	which simplifies to
	\[
	[\mathrm{CH}_0(U)]R_U=\frac{h_K R_K\prod_{w\in T}\log N(w)}{\prod_{v\in Z}\log N(v)}
	\]
	On the other hand, we have similarly
	\[
	L_X((j_!\ZZ)^D)=L_Y(\ZZ^D) \frac{\prod_{v\in Z}\zeta_v}{\prod_{w\in T}\zeta_w}
	\]
	hence
	\begin{align*}
	L_X^\ast((j_!\ZZ)^D,0)= \zeta_Y^\ast(1) \frac{\prod_{v\in Z}\zeta^\ast_v(0)}{\prod_{w\in T}\zeta^\ast_w(0)}&=\frac{2^{r_1}(2\pi)^{r_2}h_K R_K\prod_{w\in T}\log N(w)}{\omega \sqrt{|\Delta_K|}\prod_{v\in Z}\log N(v)}\\
	&=\frac{2^{r_1}(2\pi)^{r_2}[\mathrm{CH}_0(U)]R_U}{\omega \sqrt{|\Delta_K|}}
	\end{align*}
\end{proof}

\begin{rmk}
	When $X$ is singular we do not necessarily have $L_X(\GG_X[1],s)=\zeta_X(s+1)$, as can be seen from \cref{explicit_expression_L_function}. This formula is to be compared with the formula for $\zeta_U^\ast(1)$:
	\[
	\zeta_U^\ast(1)=\frac{2^{r_1}(2\pi)^{r_2}h_K R_K}{\omega \sqrt{|\Delta_K|}}\frac{\prod_{v\in Z} 1/(1-N(v)^{-1})}{\prod_{w\in T} 1/(1-N(w)^{-1})}
	=\frac{2^{r_1}(2\pi)^{r_2}[\mathrm{CH}_0(U)] R_U}{\omega \sqrt{|\Delta_K|}}\frac{\prod_{v\in Z} \log N(v)/(1-N(v)^{-1})}{\prod_{w\in T} \log N(w)/(1-N(w)^{-1})}
	\]
	and also with the formula from \cite{Poonen20}.
\end{rmk}

\begin{rmk}
	We could have also obtained this using \cref{special_value_thm}. Suppose that $U$ is a proper subscheme of $X$. Denote $S_f$ the finite places in the normalization $\pi:Y\to X$ that are not above $U$, and $Z$ the closed complement of $U$. We can compute explicitly $R((j_! \ZZ)^D)$, by considering the following snake diagram:
	\[\begin{tikzpicture}[baseline= (a).base]
	\node[scale=.85] (a) at (0,0){
		\begin{tikzcd}
		&
		& 
		& 0 \ar{r} \arrow[d, hook]
		& (H^{1}(X,(j_!\ZZ))_\RR)^\vee \arrow[d, hook]   
		&
		&\\
		& 0 \ar{r}
		& 0 \ar{r} \ar{dd}
		& \mathrm{CH}_0(X,1)_\RR \arrow[r, equal] \ar{dd}[near start]{\mathrm{Log}}
		& \mathrm{CH}_0(X,1)_\RR \ar{r} \ar{dd}[near start]{\mathrm{Log}}
		&  0
		& ~\\
		&
		&
		& ~
		&
		&
		\ar[r, phantom, ""{coordinate, name=Y}]
		&~\\
		~& \ar[l, phantom, ""{coordinate, name=Z}] 0 \ar{r}
		& \displaystyle\prod_{v\in S_f}\RR \times \displaystyle\prod_{v\in Z}\displaystyle\prod_{\pi(w)=v} \RR \ar{r} \arrow[d, equal]
		& \displaystyle\prod_{v\in S_f}\RR \times \displaystyle\prod_{v\in Z}\displaystyle\prod_{\pi(w)=v} \RR \times (\displaystyle\prod_\sigma \RR)^{G_\RR} \ar{r} \arrow[d, two heads,"\Sigma \times \prod_v \Sigma \times \Sigma"]
		& (\displaystyle\prod_\sigma \RR)^{G_\RR} \ar{r} \arrow[d, two heads]
		& 0
		& \\
		&
		& \ar[from=uuuurr, crossing over, rounded corners,
		to path=
		{ -- ([xshift=2ex]\tikztostart.east)
			-| (Y) [near end]\tikztonodes
			-| (Z) [near end]\tikztonodes
			|- ([xshift=-2ex]\tikztotarget.west)
			-- (\tikztotarget)}
		]\displaystyle\prod_{v\in S_f}\RR \times \displaystyle\prod_{v\in Z}\displaystyle\prod_{\pi(w)=v} \RR \ar{r}{\Sigma \times \prod_v \Sigma}
		& \RR \times \displaystyle\prod_{v\in Z} \RR \ar{r}
		& 0 
		&
		&
		\end{tikzcd}
	};
	\end{tikzpicture}\]
	If $U=X$, we get a similar diagram with $H^0(X,\ZZ)_\RR^\vee$ instead of $0$ at the bottom, enabling us to compute $R(\ZZ^D)$. We then obtain the formula using \cref{explicit_characteristic}.
\end{rmk}

\section{The Weil-étale Euler characteristic and the special value theorem}

\subsection{Definitions}

\begin{defi}
	Let $F$ be a $\ZZ$-constructible sheaf on $X$. We define Weil-Arakelov cohomology wtih coefficients in $F^D$ as the complex:
	\[
	R\Gamma_{ar,c}(X,F^D_\RR):=R\Gamma_{c,\cal{D}}(X,F^D_\RR)[-1] \oplus R\Gamma_{c,\cal{D}}(X,F^D_\RR)
	\]
\end{defi}
The determinant of Weil-Arakelov cohomology has a canonical trivialization
\[
\det_\RR R\Gamma_{ar,c}(X,F^D_\RR) \xrightarrow{\simeq} \det_\RR(R\Gamma_{c,\cal{D}}(X,F^D_\RR))^{-1}\otimes \det_\RR(R\Gamma_{c,\cal{D}}(X,F^D_\RR)) \xrightarrow{\simeq} \RR
\]
We now put an integral structure on the determinant of Weil-Arakelov cohomology. Let $F$ be a tamely ramified red or blue sheaf. Recall that $\Lie_X(F^D)_\RR=R\Hom_{G_\RR,X(\CC)}(\alpha^\ast F,\CC[1])$. Consider the map
\[
R\Gamma_{W,c}(X,F^D)_\RR = R\Hom(R\Gamma(X,F),\RR[-1])\oplus R\Gamma_{c,B}(X,F^D)_\RR \xrightarrow{p} R\Gamma_{c,B}(X,F^D)_\RR \xrightarrow{\mathrm{Log}} R\Hom_{G_\RR,X(\CC)}(\alpha^\ast F,\CC[1])
\]
By the duality \cref{duality_theorem} and the remark after \cref{defi_deligne_compact}, its mapping fiber is 
\[
R\Hom(R\Gamma(X,F),\RR[-1]) \oplus R\Gamma_{c,\cal{D}}(X,F^D_\RR) \simeq R\Gamma_{ar,c}(X,F^D_\RR)
\]
We thus have a distinguished triangle
\[
R\Gamma_{ar,c}(X,F^D_\RR) \to R\Gamma_{W,c}(X,F^D)_\RR \to \Lie_X(F^D)_\RR
\]
and we find by taking determinants a natural trivialization 
\begin{align*}
	\lambda:\Delta_X(F^D)_\RR = \det_\RR(R\Gamma_{W,c}(X,F^D)_\RR) \otimes \det_\RR(\Lie_X(F^D)_\RR)^{-1} &\xrightarrow{\simeq} \det_\RR(R\Gamma_{ar,c}(X,F^D_\RR) ) \xrightarrow{\simeq}\RR
\end{align*}

\begin{defi}
	Let $F$ be a tamely ramified red or blue sheaf.  The Weil-étale Euler characteristic of $F^D$ is the positive real number $\chi_X(F^D)$ such that
	\[
	\lambda(\Delta_X(F^D))=(\chi_X(F^D))^{-1}\ZZ \hookrightarrow \RR
	\]
\end{defi}

\begin{rmk}
	If we let $\cup\theta$ be the map given by the composition of projections and inclusions
	\[
	\cup\theta :R\Gamma_{ar,c}(X,F_\RR^D) \to R\Gamma_{c,\cal{D}}(X,F^D_\RR) \to R\Gamma_{ar,c}(X,F_\RR^D)[1]
	\]
	then $\cup\theta$ induces a long exact sequence
	\[
	\cdots \to H^{i-1}_{ar,c}(X,F^D_\RR) \xrightarrow{\cup\theta} H^i_{ar,c}(X,F^D_\RR) \xrightarrow{\cup\theta} H^{i+1}_{ar,c}(X,F^D_\RR) \to \cdots
	\]
	which gives a trivialization
	\[\lambda':\bigotimes_{i\in \ZZ} (\det_\RR  H^i_{ar,c}(X,F^D_\RR))^{(-1)^{i}}\xrightarrow{\simeq} \RR\]
	which coincides with $\lambda$ under the isomorphism $\det_\RR R\Gamma_{ar,c}(X,F^D_\RR)\simeq \bigotimes_{i\in \ZZ} (\det_\RR  H^i_{ar,c}(X,F^D_\RR))^{(-1)^{i}}$. We can thus define alternatively $\chi_X(F^D)$ such that
	\[
	\lambda\left(\bigotimes_{i\in \ZZ} (\det_\ZZ  H^i_{W,c}(X,F^D))^{(-1)^{i}}\otimes \bigotimes_{i\in \ZZ} (\det_\ZZ  H^i\Lie_X(F^D))^{(-1)^{i+1}}\right)=(\chi_X(F^D))^{-1}\ZZ \hookrightarrow \RR
	\]
\end{rmk}

\begin{thm}
	The Weil-étale Euler characteristic is multiplicative with respect to red-to-blue short exact sequences of tamely ramified sheaves.
\end{thm}

\begin{proof}
	Let $0\to F \to G \to H \to 0$ be a red-to-blue short exact sequence of tamely ramified sheaves. We let
	\begin{align*}
	&T_1: R\Gamma_{W,c}(X,G^D) \to R\Gamma_{W,c}(X,H^D) \to R\Gamma_{W,c}(X,F^D) \to\\
	&T_2: \Lie_X(G^D) \to \Lie_X(H^D) \to \Lie_X(F^D) \to 
	\end{align*}
	be the two natural triangles; the first one is \emph{not} distinguished. For $R=\ZZ,\RR$, let $g_R$ be a determinant functor on $\mathrm{Gr}^b(\mathrm{Mod}_R^{ft})$ with values in graded $R$-lines extending the usual determinant functor on projective $R$-modules of finite type. The graded cohomology functor \[H:D_{perf}(R) \to \mathrm{Gr}^b(\mathrm{Mod}_R^{ft})\] together with the assignment that sends a distinguished triangle $T:X\xrightarrow{u} Y \to Z \to $ to the exact sequence 
	\[
	0 \to \ker(H(u)) \to H(X) \to H(Y) \to H(Z) \xrightarrow{\partial} \ker(H(u))[1] \to 0\] induced by the long exact sequence, induces a pullback functor $H^\ast$ on Picard groupoids of determinants. We put $f_R=H^\ast g_R$. Notice that $H(T)$ makes sense for any triangle (not necessarily distinguished) that induces a long exact cohomology sequence. There is a base change isomorphism $\gamma:(g_\ZZ(-))_\RR \xrightarrow{\simeq} g_\RR((-)_\RR)$, which induces a base change isomorphism $H^\ast(\gamma):(f_\ZZ(-))_\RR \xrightarrow{\simeq} f_\RR((-)_\RR)$. By \cref{weil_etale_well_defined}, $H(T_1)$ is an exact sequence in $\mathrm{Gr}^b(\mathrm{Mod}_\ZZ^{ft})$, while $H(T_2)$ is because $T_2$ is distinguished. On the other hand, $T_{1,\RR}$ and $T_{2,\RR}$ are both distinguished triangles and the structure of exact sequence on $H(T_{1,\RR})$ and $H(T_{2,\RR})$ coming from the distinguished triangles coincides with the structure coming by base change from $H(T_1)$ and $H(T_2)$. We get a commutative diagram (see \cite[thm. 6.3 and its proof]{AMorin21}):
	\[
	\begin{tikzcd}[ampersand replacement=\&, column sep = 4cm]
		\parbox{8cm}{\centering $f_\ZZ(R\Gamma_{W,c}(X,H^D)) \otimes f_\ZZ(R\Gamma_{W,c}(X,F^D)) \otimes f_\ZZ(\Lie_X(H^D))^{-1}\otimes f_\ZZ(\Lie_X(F^D))^{-1}$} \& \parbox{4cm}{\centering $f_\ZZ(R\Gamma_{W,c}(X,G^D)) \otimes f_\ZZ(\Lie_X(G^D))^{-1}$} \\
			\parbox{8cm}{\centering $(f_\ZZ(R\Gamma_{W,c}(X,H^D)))_\RR \otimes (f_\ZZ(R\Gamma_{W,c}(X,F^D))_\RR \otimes (f_\ZZ(\Lie_X(H^D))^{-1})_\RR\otimes (f_\ZZ(\Lie_X(F^D))^{-1})_\RR $} \& 	\parbox{4cm}{\centering $(f_\ZZ(R\Gamma_{W,c}(X,G))_\RR \otimes (f_\ZZ(\Lie_X(G))_\RR$} \\
		{f_\RR(\Delta_X(H^D)_\RR)\otimes f_\RR(\Delta_X(F^D)_\RR)} \& {f_\RR(\Delta_X(F^D)_\RR)} \\
		\RR\otimes\RR \& \RR
		\arrow["{g_\ZZ(HT_1)\otimes g_\ZZ(HT_2)^{-1}}", from=1-1, to=1-2]
		\arrow[from=1-1, to=2-1]
		\arrow[from=1-2, to=2-2]
		\arrow["{(g_\ZZ(HT_1))_\RR\otimes (g_\ZZ(HT_2)^{-1})_\RR}", from=2-1, to=2-2]
		\arrow["{H^\ast(\gamma)\otimes H^\ast(\gamma)\otimes H^\ast(\gamma)^{-1}\otimes H^\ast(\gamma)^{-1}}", from=2-1, to=3-1]
		\arrow["{H^\ast(\gamma)\otimes H^\ast(\gamma)^{-1}}"', from=2-2, to=3-2]
		\arrow["{g_\RR(H(T_{1,\RR}))\otimes g_\RR(H(T_{2,\RR})^{-1})=f_\RR(T_{1,\RR})\otimes f_\RR(T_{2,\RR})^{-1}}", from=3-1, to=3-2]
		\arrow[from=3-1, to=4-1]
		\arrow[from=3-2, to=4-2]
		\arrow["\mathrm{mult}", from=4-1, to=4-2]
	\end{tikzcd}\]
	Under the multiplication map, the image of $x\ZZ \otimes y\ZZ \subset \RR \otimes \RR$ is $xy\ZZ$.
\end{proof}

\begin{prop}
	Let $Y=\Spec(\cal{O}')$ be the spectrum of an order in a number field with a finite dominant morphism $\pi:Y\to X$. Suppose that $K(Y)/K$ is tamely ramified, and let $F$ be a tamely ramified red or blue sheaf on $Y$. Then
	\[
	\chi_X((\pi_\ast F)^D)=\chi_Y(F^D).
	\]
\end{prop}

\begin{proof}
	This follows from the isomorphism
	\begin{align*}
		R\Gamma_{W,c}(Y,F^D)&\xrightarrow{\simeq} R\Gamma_{W,c}(X,(\pi_\ast F)^D)\\
		\Lie_Y(F^D) &\xrightarrow{\simeq} \Lie_X((\pi_\ast F)^D)\\
		R\Gamma_{c,\cal{D}}(Y,F^D_\RR) &\xrightarrow{\simeq} R\Gamma_{c,\cal{D}}(X,(\pi_\ast F)^D_\RR)\\
	\end{align*}
	and the isomorphism of distinguished triangles
	\[\begin{tikzcd}[ampersand replacement=\&]
		{R\Gamma_{ar,c}(Y,F^D_\RR)} \& {R\Gamma_{W,c}(Y,F^D)_\RR} \& {\Lie_Y(F^D)_\RR} \& {} \\
		{R\Gamma_{ar,c}(X,(\pi_\ast F)^D_\RR)} \& {R\Gamma_{W,c}(X,(\pi_\ast F)^D)_\RR} \& {\Lie_X((\pi_\ast F)^D)_\RR} \& {}
		\arrow[from=2-1, to=2-2]
		\arrow["\simeq"', from=1-1, to=2-1]
		\arrow["\simeq"', from=1-2, to=2-2]
		\arrow[from=1-1, to=1-2]
		\arrow[from=1-2, to=1-3]
		\arrow[from=2-2, to=2-3]
		\arrow[from=1-3, to=1-4]
		\arrow[from=2-3, to=2-4]
		\arrow["\simeq"', from=1-3, to=2-3]
	\end{tikzcd}\]
\end{proof}

We now use the multplicativity of the Weil-étale Euler characteristic to extend it to arbitrary tamely ramified $\ZZ$-constructible sheaves:
\begin{defi}
	Let $F$ be a tamely ramified $\ZZ$-constructible sheaf, let $j:U\to X$ be an open subscheme such that $F_{|U}$ is locally constant and let $i:Z\to X$ be the closed complement. The sheaves $j_!F_{|U}$ and $i_\ast i^\ast F$ are tamely ramified and respectively red and blue, and we define the Weil-étale Euler characteristic of $F^D$ as
	\[
	\chi_X(F^D)=\chi_X((j_! F_{|U})^D)\chi_X((i_\ast i^\ast F)^D)
	\]
	The definition doesn't depend on the choice of $U$. The Weil-étale Euler characteristic is multiplicative with respect to short exact sequences of tamely ramified sheaves. Let $Y=\Spec(\cal{O}')$ be the spectrum of an order in a number field with a finite dominant morphism $\pi:Y\to X$ such that $K(Y)/K$ is tamely ramified, and let $F$ be a tamely ramified $\ZZ$-constructible sheaf on $Y$. Then
	\[
	\chi_X((\pi_\ast F)^D)=\chi_Y(F^D)
	\]
\end{defi}
\begin{proof}
	This follows formally from the previous results by using the open-closed decomposition lemma, see \cite[§ 6.5]{AMorin21}.
\end{proof}
\subsection{Computations of the Euler characteristic, and the special value theorem}\label{subsec:computations_special_value_thm}

We first give an explicit expression of our Euler characteristic which doesn't involve Weil-étale cohomology anymore. This will allow us to give an expression valid for any tamely ramified $\ZZ$-constructible sheaf.

\begin{prop}
	Let $F$ be a tamely ramified red or blue sheaf. Let $R_1(F^D)$ be the absolute value of the determinant of the pairing
	\[
	H^{-1}_{c,\cal{D}}(X,F^D_\RR)\times H^{1}(X,F)_\RR \to \RR
	\]
	and $R_0(F^D)$ the absolute value of the determinant of the pairing
	\[
	H^0_{c,\cal{D}}(X, F^D_\RR) \times H^{0}(X,F)_\RR \to \RR
	\] in the following bases : pick bases modulo torsion of $H^{-1}(\Lie_X(F^D))$, of $H^i(X,F)$ for $i=0,1$ and of $H^{i}_{c,B}(X,F^D)$ for $i=-1,0$. Then pick any $\RR$-bases of $H^i_{c,\cal{D}}(X,F^D_\RR)$ compatible\footnote{In the following sense: denote $\cal{E}^\bullet$ the exact sequence seen as an acyclic complex, then the choice of bases gives an element of $\det_\RR \cal{E}^\bullet$ which is required to be sent to $1$ under $\det_\RR \cal{E}^\bullet\xrightarrow[\simeq]{\det(0)} \RR$} with the chosen bases in the following exact sequence
	\[
	0 \to H^{-1}_{c,\cal{D}}(X,F^D_\RR) \to H^{-1}_{c,B}(X,F^D)_\RR \to H^{-1}(\Lie_X(F^D))_\RR \to H^0_{c,\cal{D}}(X, F^D_\RR) \to H^{0}_{c,B}(X,F^D)_\RR \to 0
	\]
	We have
	\[
	\chi_X(F^D)=\frac{[H^0(X,F)_{tor}][H^0_{c,B}(X,F^D)_{tor}]}{[H^1(X,F)_{tor}][H^{-1}_{c,B}(X,F^D)_{tor}][H^{0}(\Lie_X(F^D))]}\frac{R_1(F^D)}{R_0(F^D)}
	\]
\end{prop}

\begin{rmk}
	We can reformulate the part about determinants of pairings. Consider an acyclic chain complex of $\RR$-vector spaces $A^\bullet$ with a $\ZZ$-lattice $M^i\subset A^i$ given for each $i$. We define $\det_{\ZZ,\RR} A^\bullet$ to be the absolute value of the image of the vector corresponding to bases of the lattices under the canonical isomorphism $\det_\RR A^\bullet \xrightarrow[\det_\RR(0)]{\simeq} \RR$; put another way, it is the positive real number such that 
		\[
		\bigotimes_{i\in \ZZ} (\det_\ZZ M^i)^{(-1)^i} \subset (\bigotimes_{i\in \ZZ} (\det_\ZZ M^i)^{(-1)^i})_\RR \simeq \det_\RR A^\bullet
		\]
		corresponds to $(\det_{\ZZ,\RR} A^\bullet) \ZZ \subset \RR$ under the canonical isomorphism $\det_\RR A^\bullet \xrightarrow[\det_\RR(0)]{\simeq} \RR$. Consider the following complex
		\[
		A^\bullet :~~ 0 \to (H^1(X,F)_\RR)^\vee \to H^{-1}_{c,B}(X,F^D)_\RR \to H^{-1}(\Lie_X(F^D))_\RR \to (H^0(X,F)_\RR)^\vee \to H^{0}_{c,B}(X,F^D)_\RR \to 0
		\]
		with first non-zero term in degree $-1$ and endowed with the natural lattices. Then we have
		\[
		\frac{R_1(F^D)}{R_0(F^D)}=\det_{\ZZ,\RR} A^\bullet
		\]	
\end{rmk}

\begin{proof}
	Since determinants depend only on cohomology and are multiplicative in short exact sequences, we can write $\Delta_X(F^D)$ as an alternating tensor product of the determinants of the torsion and torsion-free parts of the cohomology groups:
	\[
	\Delta_X(F^D)=\Delta_X(F^D)_{tor}\Delta_X(F^D)/tor
	\]
	We have $(\Delta_X(F^D)_{tor})_\RR = \det_\RR(0)=\RR$ canonically hence by \cite[lem. A.1]{AMorin21} the contribution of $\Delta_X(F^D)_{tor}$ inside $\Delta_X(F^D)_\RR$ is sent under the trivialisation to
	\[
	1/\chi_X(F^D)_{tor}:=\frac{\prod [H^i \Lie_X(F^D)_{tor}]^{(-1)^i}}{\prod [H^i_{W,c}(X,F^D)_{tor}]^{(-1)^i}}
	\]
	We obtain the first part of the claimed expression by the computation of the involved cohomology groups in \cref{weil-etale_perfect,lie_perfect}.
	
	It remains to show that $\Delta_X(F^D)/tor=\frac{R_1(F^D)}{R_0(F^D)}$. We will use \cite[lem. A.2]{AMorin21}. Fix basis vectors of $H^i(X,F)/tor$ for $i=0,1$, of $(H^{-1}\Lie_X(F^D))/tor$ and of $H^i_{c,B}(X,F^D)$ for $i=-1,0$. The image of $\Delta_X(F^D)/tor$ inside 
	\begin{align*}
	\Delta_X(F^D)_\RR&:= \det_\RR R\Gamma_{W,c}(X,F^D)_\RR \otimes_\RR (\det_\RR \Lie_X(F^D)_\RR)^{-1}\\
	& \simeq \det_\RR (D_{F,\RR}[1]) \otimes_\RR \det_\RR R\Gamma_{c,B}(X,F^D)_\RR \otimes_\RR (\det_\RR \Lie_X(F^D)_\RR)^{-1} \\
	& \simeq \det_\RR (D_{F,\RR}[1]) \otimes_\RR \det_\RR R\Gamma_{c,\cal{D}}(X,F^D_\RR)\\
	& \simeq \bigotimes_i \left(\det_\RR H^{i+1}(D_{F,\RR})\right)^{(-1)^i} \otimes_\RR \bigotimes_{i} \left(\det_\RR H^{i}_{c,\cal{D}}(X,F^D_\RR)\right)^{(-1)^i}
	\end{align*}
	is a certain product of basis vectors and dual basis vectors, obtained from the basis vectors of $H^i_{c,B}(X,F^D)/tor$, $H^i\Lie_X(F^D)/tor$ and $H^i(X,F)/tor$ through the identity $H^i(D_{F,\RR})=\Hom(H^{2-i}(X,F),\RR)$ and the long exact sequence associated to the fiber sequence
	\[
	R\Gamma_{c,\cal{D}}(X,F^D_\RR) \to R\Gamma_{c,B}(X,F^D)_\RR \to \Lie_X(F^D)_\RR
	\]
	Finally, the trivialisation is obtained by combining together, for each $i\in \ZZ$, the terms related by the duality theorem \cref{duality_theorem}:
	\begin{align*}
	\left(\det_\RR H^{i}_{c,\cal{D}}(X,F^D_\RR)\right)^{(-1)^i} \otimes \left(\det_\RR H^{i+2}(D_{F,\RR})\right)^{(-1)^{i+1}} &\xrightarrow{\simeq}  \left(\det_\RR H^{i}_{c,\cal{D}}(X,F^D_\RR)\right)^{(-1)^i} \bigotimes \left(\det_\RR H^{i}_{c,\cal{D}}(X,F^D_\RR) \right)^{(-1)^{i+1}}\\
	& \xrightarrow{\simeq} \RR
	\end{align*}
	The isomorphism $H^{i}_{c,\cal{D}}(X,F^D_\RR) \xrightarrow{\simeq} H^{i+2}(D_{F,\RR})$ is non-trivial only for $i=-1,0$ in which case its determinant in the bases obtained is exactly (up to sign) $R_1$, resp. $R_0$. By \cite[lem. A.2]{AMorin21}, the contribution of $\Delta_X(F^D)/tor$ is thus sent under the trivialisation to
	\[
	\frac 1 {\chi_X(F^D)/tor} := R_0/R_1
	\]
\end{proof}

\begin{cor}\label{first_formula_characteristic}
	Let $F$ be a tamely ramified $\ZZ$-constructible sheaf. With the same notations, we have
	\[
	\chi_X(F^D)=\frac{[H^0(X,F)_{tor}][H^0_{c,B}(X,F^D)_{tor}]}{[H^1(X,F)_{tor}][H^{-1}_{c,B}(X,F^D)_{tor}][H^{0}(\Lie_X(F^D))]}\frac{R_1(F^D)}{R_0(F^D)}
	\]
\end{cor}
\begin{proof}
	Let $j:U\to X$ be an open subscheme such that $F_{|U}$ is locally constant and let $i:Z\to X$ be the closed complement. Put $F_U=j_!F_{|U}$ and $F_Z=i_\ast i^\ast F$. Then $F_U$ is red, $F_Z$ is blue and there is a short exact sequence $0 \to F_U \to F \to F_Z \to 0$. Consider the two following acyclic complexes of abelian groups, with first non-zero term respectively in degree $-1$ and $0$:
	\begin{align*}
	A^\bullet : 0 &\to H^{-1}_{c,B}(X,(F_Z)^D) \to H^{-1}_{c,B}(X,F^D) \to H^{-1}_{c,B}(X,(F_U)^D)\\
	& \to H^{0}_{c,B}(X,(F_Z)^D) \to H^{0}_{c,B}(X,F^D) \to H^{0}_{c,B}(X,(F_U)^D) \to I \to 0\\
	B^\bullet : 0 & \to H^0(X,F_U) \to H^0(X,F) \to H^0(X,F_Z)\\
	& \to H^1(X,F_U) \to H^1(X,F) \to H^1(X,F_Z) \to J \to 0
	\end{align*}
	where $I$ is the image of $H^{0}_{c,B}(X,F_U^D)$ in $H^{1}_{c,B}(X,F_Z^D)$ and $J$ is the image of $H^1(X,F_Z)$ in $H^2(X,F_U)$. Since $F_Z$ is blue , $H^{1}_{c,B}(X,F_Z^D)$ and $H^1(X,F_Z)$ are finite so $I$ and $J$ are also finite. Moreover Artin-Verdier duality gives a commutative square
	\[\begin{tikzcd}[ampersand replacement=\&]
	{H^{0}_{c,B}(X,(F_U)^D)} \& {H^{1}_{c,B}(X,(F_Z)^D)} \\
	{H^2(X,F_U)^\ast} \& {H^1(X,F_Z)^\ast}
	\arrow["\simeq", from=1-1, to=2-1]
	\arrow["\simeq", from=1-2, to=2-2]
	\arrow["\partial", from=1-1, to=1-2]
	\arrow["\partial", from=2-1, to=2-2]
	\end{tikzcd}\]
	whence $I\simeq J^\ast$. 
	
	Since $F_Z$ is supported on a finite closed subscheme, we have $\Lie_X((F_Z)^D)=0$ hence an isomorphism $\Lie_X(F^D)\xrightarrow{\simeq}\Lie_X((F_U)^D)$ which gives a canonical isomorphism
	\[
	\det_\ZZ \Lie_X(F^D) \left(\det_\ZZ \Lie_X((F_U)^D)\right)^{-1} \xrightarrow[\tau]{\simeq} \ZZ
	\]
	On the other hand, since $A^\bullet$ and $B^\bullet$ are acyclic, we obtain canonical trivializations $\det_\ZZ A^\bullet \xrightarrow[\det(0)]{\simeq} \ZZ$ and $\det_\ZZ B^\bullet \xrightarrow[\det(0)]{\simeq} \ZZ$. Consider now the line
	\[
	\Delta:= \det_\ZZ A^\bullet \det_\ZZ B^\bullet \det_\ZZ \Lie_X(F^D) \left(\det_\ZZ \Lie_X((F_U)^D)\right)^{-1}
	\]
	It obtains a canonical trivialization $\alpha: \Delta \xrightarrow{\det 0 \otimes \det 0 \otimes \tau } \ZZ\otimes \ZZ \otimes \ZZ \xrightarrow{\simeq} \ZZ$ through the previous remarks.
	
	For a distinguished triangle $X\to Y \to Z \to X[1]$, denote $i:(\det Y)^{-1} \det X \det Z\to 1$ the structural isomorphism to the unit. From the octahedral diagram
	\[\begin{tikzcd}[ampersand replacement=\&]
		{R\Gamma_{c,\cal{D}}(X,(F_Z)^D_\RR)} \& {R\Gamma_{c,\cal{D}}(X,F^D_\RR)} \& {R\Gamma_{c,\cal{D}}(X,(F_U)^D_\RR)} \\
		{R\Gamma_{c,B}(X,(F_Z)^D)_\RR} \& {R\Gamma_{c,B}(X,F^D)_\RR} \& {R\Gamma_{c,B}(X,(F_U)^D)_\RR} \\
		{\Lie_X((F_Z)^D)_\RR=0} \& {\Lie_X(F)_\RR} \& {\Lie_X((F_U)^D)_\RR}
		\arrow["\simeq", from=1-1, to=2-1]
		\arrow["\simeq", from=3-2, to=3-3]
		\arrow[from=2-1, to=2-2]
		\arrow[from=1-1, to=1-2]
		\arrow[from=1-2, to=2-2]
		\arrow[from=1-2, to=1-3]
		\arrow[from=1-3, to=2-3]
		\arrow[from=2-2, to=2-3]
		\arrow[from=2-3, to=3-3]
		\arrow[from=2-2, to=3-2]
		\arrow[from=2-1, to=3-1]
		\arrow[from=3-1, to=3-2]
	\end{tikzcd}\]
	we find using the associativity axiom of determinant functors \cite[3.1]{Breuning11} a commutative diagram
	\begin{equation}\label{associativity_diagram}
		\begin{tikzcd}[ampersand replacement=\&]
		{\left(\det_\RR R\Gamma_{c,\cal{D}}(X,F^D_\RR)\right)^{-1}\det_\RR R\Gamma_{c,\cal{D}}(X,(F_U)^D_\RR) \det_\RR R\Gamma_{c,\cal{D}}(X,(F_Z)^D_\RR)} \& \RR \\
		\parbox{8cm}{\centering $\left(\det_\RR R\Gamma_{c,B}(X,F^D)_\RR\right)^{-1}  \det_\RR R\Gamma_{c,B}(X,(F_U)^D)_\RR$ \phantom{a} $\det_\RR R\Gamma_{c,B}(X,(F_Z)^D)_\RR \det_\RR \Lie_X(F^D)_\RR$ \phantom{a} $\left(\det_\RR \Lie_X((F_U)^D)_\RR\right)^{-1}$}
		 \& {\RR\otimes_\RR \RR}
		\arrow["i",from=1-1, to=1-2]
		\arrow[from=2-2, to=1-2]
		\arrow["\simeq", from=1-1, to=2-1]
		\arrow["i_\RR \otimes \tau_\RR", from=2-1, to=2-2]
	\end{tikzcd}
\end{equation}
	There are isomorphims
	\begin{align}
		\label{isomA} \det_\RR A^\bullet_\RR &\simeq \left(\det_\RR R\Gamma_{c,B}(X,F^D)_\RR\right)^{-1} \det_\RR R\Gamma_{c,B}(X,(F_U)^D)_\RR\det_\RR R\Gamma_{c,B}(X,(F_Z)^D)_\RR\\
		\label{isomB} \det_\RR B^\bullet_\RR &\simeq  \left(\det_\RR R\Gamma(X,F)_\RR\right)^{(-1)} \det_\RR R\Gamma(X,F_U)_\RR \det_\RR R\Gamma(X,F_Z)_\RR
	\end{align}
	 under which the trivialization $\det(0)_\RR=\det_\RR(0)$ on the left corresponds to $i_\RR$ on the right.
	 By combining \cref{associativity_diagram,isomA,isomB} we find a commutative diagram of isomorphisms
	 \[\begin{tikzcd}[ampersand replacement=\&]
	 	{\Delta_\RR} \& \parbox{8cm}{\centering $\left(\det_\RR R\Gamma_{c,\cal{D}}(X,F^D_\RR)\right)^{(-1)} \det_\RR R\Gamma_{c,\cal{D}}(X,(F_U)^D_\RR) \det_\RR R\Gamma_{c,\cal{D}}(X,(F_Z)^D_\RR)$ \phantom{a} $\left(\det_\RR R\Gamma(X,F)_\RR\right)^{(-1)} \det_\RR R\Gamma(X,F_U)_\RR\det_\RR R\Gamma(X,F_Z)_\RR$} \\
	 	\RR \& {\RR}
	 	\arrow["\simeq", from=1-1, to=1-2]
	 	\arrow["{i\cdot i_\RR}", from=1-2, to=2-2]
	 	\arrow[equals, from=2-2, to=2-1]
	 	\arrow["\alpha"', from=1-1, to=2-1]
	 \end{tikzcd}\]
  	
  	We now describe how to produce another trivialization of $\Delta_\RR$. By \cref{duality_theorem}, there is an isomorphism $\psi_{(-)}:R\Gamma_{c,\cal{D}}(X,(-)^D_\RR) \xrightarrow{\simeq} R\Gamma(X,(-))_\RR^\vee$. If $L$ is a graded $\RR$-line, there is a canonical isomorphism $L\otimes L^{-1} \xrightarrow{\delta} \RR$ and we take the trivialization
  	\begin{align*}
  		& \begin{split}
  			\left(\det_\RR R\Gamma_{c,\cal{D}}(X,F^D_\RR)\right)^{(-1)} \det_\RR R\Gamma_{c,\cal{D}}(X,(F_U)^D_\RR) \det_\RR R\Gamma_{c,\cal{D}}(X,(F_Z)^D_\RR)\\ \left(\det_\RR R\Gamma(X,F)_\RR\right)^{(-1)} \det_\RR R\Gamma(X,F_U)_\RR\det_\RR R\Gamma(X,F_Z)_\RR
  		\end{split}\\
  		&\\
  		& \xrightarrow{((\det\psi_F)^t)^{-1}\det\psi_{F_U} \det\psi_{F_Z}}  \begin{split}
  			\det_\RR R\Gamma(X,F)_\RR \left(\det_\RR R\Gamma(X,F_U)_\RR\right)^{(-1)} \left(\det_\RR R\Gamma(X,F_Z)_\RR\right)^{(-1)}\\ \left(\det_\RR R\Gamma(X,F)_\RR\right)^{(-1)} \det_\RR R\Gamma(X,F_U)_\RR\det_\RR R\Gamma(X,F_Z)_\RR
  		\end{split} \\
  	 &\xrightarrow{\delta\otimes \delta \otimes \delta} \RR \otimes_\RR \RR \otimes_\RR \RR \\
  	 &\xrightarrow{\mathrm{mult}} \RR.
  	\end{align*}
  	Consider the following diagram:
  	\[\begin{tikzcd}[ampersand replacement=\&,column sep=0cm]
  		{\parbox{8cm}{\centering $\left(\det_\RR R\Gamma_{c,\cal{D}}(X,F^D_\RR)\right)^{(-1)} \det_\RR R\Gamma_{c,\cal{D}}(X,(F_U)^D_\RR) \det_\RR R\Gamma_{c,\cal{D}}(X,(F_Z)^D_\RR)$ \phantom{a} $\left(\det_\RR R\Gamma(X,F)_\RR\right)^{(-1)} \det_\RR R\Gamma(X,F_U)_\RR\det_\RR R\Gamma(X,F_Z)_\RR$}} \\
  		{\RR\otimes_\RR \RR} \& {\parbox{8cm}{\centering $\det_\RR R\Gamma(X,F) \left(\det_\RR R\Gamma(X,F_U)\right)^{(-1)} \left(\det_\RR R\Gamma(X,F_Z)\right)^{(-1)}$ \phantom{a} $\left(\det_\RR R\Gamma(X,F)_\RR\right)^{(-1)} \det_\RR R\Gamma(X,F_U)_\RR\det_\RR R\Gamma(X,F_Z)_\RR$}} \\
  		\RR \& {\RR \otimes_\RR \RR \otimes_\RR \RR}
  		\arrow["{\mathrm{mult}}", from=2-1, to=3-1]
  		\arrow["{((i_\RR)^t)^{-1}\otimes i_\RR}"', from=2-2, to=2-1]
  		\arrow[from=2-2, to=3-2]
  		\arrow["{((\det\psi_F)^t)^{-1}\det\psi_{F_U} \det\psi_{F_Z}}"{description}, from=1-1, to=2-2]
  		\arrow["{i\otimes i_\RR}", from=1-1, to=2-1]
  		\arrow["{i\cdot i_\RR}"', bend right, from=1-1, to=3-1]
  		\arrow["\mathrm{mult}", from=3-2, to=3-1]
  	\end{tikzcd}\]
  	From the isomorphism of fiber sequences
  	\[\begin{tikzcd}[ampersand replacement=\&]
  		{ R\Gamma_{c,\cal{D}}(X,(F_Z)^D_\RR)} \& {R\Gamma_{c,\cal{D}}(X,F^D_\RR)} \& { R\Gamma_{c,\cal{D}}(X,(F_U)^D_\RR)} \\
  		{R\Gamma(X,F_Z)_\RR^\vee} \& {R\Gamma(X,F)_\RR^\vee} \& {R\Gamma(X,F_U)_\RR^\vee}
  		\arrow[from=2-1, to=2-2]
  		\arrow[from=2-2, to=2-3]
  		\arrow["\simeq", from=1-1, to=2-1]
  		\arrow["\simeq", from=1-2, to=2-2]
  		\arrow["\simeq"', from=1-3, to=2-3]
  		\arrow[from=1-1, to=1-2]
  		\arrow[from=1-2, to=1-3]
  	\end{tikzcd}\]
  	we deduce that the upper triangle commutes, while the lower square commutes formally. Thus the trivialization by duality equals the trivialization $i\cdot i_\RR$.  	

	Under those equal trivializations, the image of $\Delta$ inside $\RR$ is $\ZZ$ on the one hand and is computed with \cite[prop. A.3]{AMorin21} on the other hand. We get an identity
	\[
	1 = \frac{[H^0(X,F)_{tor}][H^0_{c,B}(X,F^D)_{tor}]}{[H^1(X,F)_{tor}][H^{-1}_{c,B}(X,F^D)_{tor}][H^{0}(\Lie_X(F^D))]}\frac{R_1(F^D)}{R_0(F^D)}\frac{[J]}{[I]}\frac{1}{\chi_X((F_U)^D)\chi_X((F_Z)^D)}
	\]
	and thus, because $[I]=[J]$:
	\[
	\chi_X(F^D):=\chi_X((F_U)^D)\chi_X((F_Z)^D)=\frac{[H^0(X,F)_{tor}][H^0_{c,B}(X,F^D)_{tor}]}{[H^1(X,F)_{tor}][H^{-1}_{c,B}(X,F^D)_{tor}][H^{0}(\Lie_X(F^D))]}\frac{R_1(F^D)}{R_0(F^D)}
	\]
\end{proof}

\begin{defi}\label{defi_discriminant}
	Let $F$ be a tamely ramified $\ZZ$-constructible sheaf. We put
	\begin{align*}
	& r_1(F)=\log_2([\Ext^1_{G_\RR,X(\CC)}(\alpha^\ast F,2i\pi\ZZ)])\\
	& r_2(F)=\rank_\ZZ \Hom_{G_\RR,X(\CC)}(\alpha^\ast F,2i\pi\ZZ)
	\end{align*}
	For each archimedean place $v$ of $K$, fix a corresponding embedding $\sigma_v$. There is an isomorphism\footnote{Non-canonical; the particular choice here is justified by the later computation of $\Disc(F)$ and $R(F^D)$ for the case $F=\ZZ$}
	\[
	\Hom_{G_\RR,X(\CC)}(\alpha^\ast F,\RR) \xrightarrow{\simeq} \prod_v \Hom_{G_v}(F_v,\ZZ)_\RR
	\]
	given by $(\phi_\sigma) \mapsto (\phi_{\sigma_v})_{v~\text{real}},(2\phi_{\sigma_v})_{v~\text{complex}}$.	Using the isomorphism $\Lie_X(F^D)_\RR \simeq \Hom_{G_\RR,X(\CC)}(\alpha^\ast F,\CC[1])$ (\ref{lie_tensor_R}), the short exact sequence $0 \to 2i\pi\RR \to \CC\xrightarrow{\Re} \RR \to 0$ and the above isomorphism, we can consider the following acyclic complex of $\RR$-vector spaces with lattices:
	\[
	D^\bullet :~~ 0 \to \Hom_{G_\RR,X(\CC)}(\alpha^\ast F,2i\pi\ZZ)_\RR \to H^{-1}\Lie_X(F^D)_\RR \to \prod_v \Hom_{G_v}(F_v,\ZZ)_\RR \to 0
	\]
	(with first non-zero term in degree $-1$). We denote
	\[
	\Disc(F)=(2\pi)^{r_2(F)}\det_{\ZZ,\RR} D^\bullet
	\]
\end{defi}

\begin{rmk}
\begin{itemize}
	\item[]
	\item If $\widetilde{D}^\bullet$ is the complex 
	\[
	0 \to \Hom_{G_\RR,X(\CC)}(\alpha^\ast F,i\ZZ)_\RR \to H^{-1}\Lie_X(F^D)_\RR \to \prod_v \Hom_{G_v}(F_v,\ZZ)_\RR \to 0
	\]
	with similar degree conventions, then $\Disc(F)=\det_{\ZZ,\RR} \widetilde{D}^\bullet$.
	\item The quantity $\Disc(F)$ only depends on $F_\eta$.
\end{itemize}
\end{rmk}

We now reformulate the previous expression obtained to make it closer to a formula looking like the analytic class number formula.
\begin{prop}\label{explicit_characteristic}
	Let $F$ be a tamely ramified $\ZZ$-constructible sheaf. Denote 
	\[N_2:=\Ker\left(\Ext^2_{G_\RR,X(\CC)}(\alpha^\ast F,2i\pi\ZZ)\to H^1_{c,B}(X,F^D)\right)\]
	and let $\GG_X=[g_\ast \GG_m \to \oplus_{x\in X_0}i_{x,\ast}\ZZ]$ denote Deninger's dualizing complex. Let $R(F^D):=\det_{\ZZ,\RR} B^\bullet$ with $B^\bullet$ the exact sequence of $\RR$-vector spaces with lattices
	\[
	B^\bullet : ~~ 0 \to (H^1(X,F)_\RR)^\vee \to \Hom_X(F,\GG_X)_\RR \to \prod_v \Hom_{G_v}(F_v,\ZZ)_\RR \to (H^0(X,F)_\RR)^\vee \to \Ext^1_X(F,\GG_X)_\RR \to 0
	\]
	with first non-zero term in degree $-1$. We have
	\[
	\chi_X(F^D)=\frac{(2\pi)^{r_2(F)}2^{r_1(F)}[H^0(X,F)_{tor}][\Ext^1_X(F,\GG_X)_{tor}]}{[H^1(X,F)_{tor}][\Hom_X(F,\GG_X)_{tor}][\Ext^1_{G_K^t}(F_\eta,\cal{O}_{K^t})][N_2]\Disc(F)}R(F^D)
	\]
\end{prop}

\begin{proof}
	We have $\ZZ^c_X\simeq \GG_X[1]$ \cite{Nart89}.	Consider the following biacyclic double complex (numbers on the side indicate indexing, zeros are ommited):
	\[\begin{tikzpicture}[baseline= (a).base]
	\node[scale=.85] (a) at (0,0){\begin{tikzcd}[column sep=small]
	&& {C^\bullet} & {D^\bullet} && {E^\bullet} \\
	&& {\Hom_{G_\RR,X(\CC)}(\alpha^\ast F,2i\pi \ZZ)_\RR} & {\Hom_{G_\RR,X(\CC)}(\alpha^\ast F,2i\pi \ZZ)_\RR} &&& {-1} \\
	{A^\bullet:} & {(H^1(X,F)_\RR)^\vee} & {H^{-1}_{c,B}(X,F^D)_\RR} & {H^1\Lie_X(F^D)_\RR} & {(H^0(X,F)_\RR)^\vee } & {H^{0}_{c,B}(X,F^D)_\RR} & 0 \\
	{B^\bullet:} & {(H^1(X,F)_\RR)^\vee} & {H^{-1}(X,F^D)_\RR} & {\prod_v \Hom_{G_v}(F_v,\ZZ)_\RR} & {(H^0(X,F)_\RR)^\vee } & {H^{0}(X,F^D)_\RR} & 1 \\
	& {-1} & 0 & 1 & 2 & 3
	\arrow[Rightarrow, no head, from=3-5, to=4-5]
	\arrow[Rightarrow, no head, from=3-2, to=4-2]
	\arrow[from=3-3, to=4-3]
	\arrow[from=4-2, to=4-3]
	\arrow[from=3-2, to=3-3]
	\arrow[from=3-3, to=3-4]
	\arrow[from=3-4, to=3-5]
	\arrow[from=3-5, to=3-6]
	\arrow[from=4-5, to=4-6]
	\arrow["\simeq", from=3-6, to=4-6]
	\arrow[from=4-3, to=4-4]
	\arrow[from=4-4, to=4-5]
	\arrow[from=3-4, to=4-4]
	\arrow[Rightarrow, no head, from=2-3, to=2-4]
	\arrow[from=2-3, to=3-3]
	\arrow[from=2-4, to=3-4]
	\end{tikzcd}};
	\end{tikzpicture}\]
	If $C^{\bullet,\bullet}$ is any biacyclic double complex of $\RR$-vector spaces with lattices, we find by applying \cref{trivialization_double_complex} to both the rows and the columns that the following diagram is commutative:
	\[\begin{tikzcd}[ampersand replacement=\&,column sep = 1.5cm]
	{\bigotimes_{i,j}(\det_\RR C^{i,j})^{(-1)^{i+j}}} \& {\det_\RR \mathrm{Tot}C} \& {\bigotimes_{j}(\det_\RR C^{\bullet,j})^{(-1)^{j}}} \& {\bigotimes_{j}(\RR)^{(-1)^{j}}} \\
	{} \& {\bigotimes_{i}(\det_\RR C^{i,\bullet})^{(-1)^{i}}} \& {\bigotimes_{i}(\RR)^{(-1)^{i}}} \& \RR
	\arrow["\simeq", from=1-1, to=1-2]
	\arrow[from=1-2, to=1-3]
	\arrow[from=1-2, to=2-2]
	\arrow["{\otimes_j\det 0}", from=1-3, to=1-4]
	\arrow[from=2-3, to=2-4]
	\arrow[from=1-4, to=2-4]
	\arrow["{\otimes_i\det 0}", from=2-2, to=2-3]
	\arrow["{\det 0}"{description}, from=1-2, to=2-4]
	\end{tikzcd}\]
	Fix bases $(u_k^{(i,j)})$ of the lattices inside the $C^{i,j}$, and denote $(u_k^{i,j})^{-1}$ the dual basis vectors. By looking at the image under the composite morphism $\bigotimes_{i,j}(\det_\RR C^{i,j})^{(-1)^{i+j}} \to \RR$ of $\otimes_{i,j} \bigwedge_k ((u_k^{(i,j)})^{(-1)^{i+j}})$, we find the relation
	\[
	\prod_i (\det_{\ZZ,\RR} C^{i,\bullet})^{(-1)^i}=\prod_j (\det_{\ZZ,\RR} C^{\bullet,j})^{(-1)^j}
	\]
	Hence in our case we obtain
	\[
	\det_{\ZZ,\RR} A^\bullet (\det_{\ZZ,\RR} B^\bullet)^{-1}=\det_{\ZZ,\RR} C^\bullet (\det_{\ZZ,\RR} D^\bullet)^{-1} (\det_{\ZZ,\RR} E^\bullet)^{-1}
	\]
	Thus it remains to compute $\det_{\ZZ,\RR} C^\bullet~ (\det_{\ZZ,\RR} E^\bullet)^{-1}$. Consider the following acyclic complex of abelian groups:
	\[G^\bullet:~~\begin{tikzcd}[row sep=tiny]
	0 & {\Hom_{G_\RR,X(\CC)}(\alpha^\ast F,2i\pi\ZZ)} & {H^{-1}_{c,B}(X,F^D)} & {H^{-1}(X,F^D)} \\
	{} & {\Ext^1_{G_\RR,X(\CC)}(\alpha^\ast F,2i\pi\ZZ)} & {H^{0}_{c,B}(X,F^D)} & {H^{0}(X,F^D)} \\
	{} & {N_2} & 0 && {}
	\arrow[from=1-1, to=1-2]
	\arrow[from=1-2, to=1-3]
	\arrow[from=1-3, to=1-4]
	\arrow[from=2-2, to=2-3]
	\arrow[from=2-1, to=2-2]
	\arrow[from=2-3, to=2-4]
	\arrow[from=3-1, to=3-2]
	\arrow[from=3-2, to=3-3]
	\end{tikzcd}\]
	with first non-zero term in degree $-1$. The last term is a subgroup of $\Ext^2_{G_\RR,X(\CC)}(\alpha^\ast F,2i\pi\ZZ)$ and hence is finite $2$-torsion. Then $\det_\ZZ G^\bullet\xrightarrow[\det(0)]{\simeq} \ZZ$ and $G^\bullet_\RR=C^\bullet\oplus E^\bullet[-3]$, thus we find 
	\[
	\det_{\ZZ,\RR}C^\bullet(\det_{\ZZ,\RR}E^\bullet)^{-1}=\det_{\ZZ,\RR} G_\RR = \prod_i [G^i_{tor}]
	\]
\end{proof}

\begin{prop}\label{special_value_Z}
	Supose that $X$ is regular. We have
	\begin{align*}
	&\sum (-1)^i i\cdot \dim_\RR H^i_{ar,c}(X,\ZZ^D_\RR) = -1\\
	&\chi_X(\ZZ^D)=\frac{2^{r_1}(2\pi)^{r_2}h_K R_K}{\omega \sqrt{|\Delta_K|}}
	\end{align*}
	and thus the special value formula
	\begin{align*}
		\ord_{s=0} L_X(\ZZ^D)&= \ord_{s=1} \zeta_X = \sum (-1)^i i\cdot \dim_\RR H^i_{ar,c}(X,\ZZ^D_\RR)\\
		L^\ast_X(\ZZ^D,0)&=\zeta_X^\ast(1)=\chi_X(\ZZ^D)
	\end{align*}
\end{prop}

\begin{proof}
	The first formula is immediate from the previous computations.
	
	We use \cref{explicit_characteristic}. We have $\Hom_{G_\RR,X(\CC)}(\ZZ,2i\pi\ZZ)=(2i\pi\ZZ)^{r_2}$, which is of rank $r_2(\ZZ)=r_2$, $\Ext^1_{G_\RR,X(\CC)}(\ZZ,2i\pi\ZZ)=(\ZZ/2\ZZ)^{r_1}$ so that $r_1(\ZZ)=r_1$, $\Ext^2_{G_\RR,X(\CC)}(\ZZ,2i\pi\ZZ)=0$ hence $N_2=0$, and $\Lie_X(\ZZ^D)=\cal{O}_K[1]$ so in particular $H^0\Lie_X(\ZZ^D)=0$. Thus using \cref{weil_etale_Z,cohomology_Z} we find
	\[
	\chi_X(\ZZ^D)=\frac{2^{r_1}(2\pi)^{r_2}h_K R(\ZZ^D)}{\omega \Disc(\ZZ^D)}
	\]
	
	By the choices made in \cref{defi_discriminant}, $\Disc(\ZZ^D)$ is computed as following: fix a basis $(\alpha_i)$ of $\cal{O}_K$. Choose an ordering of the archimedean places, real places first and complex places second, and denote $\sigma_i$ the choice of an embedding for each place $v_i$ and $\sigma_{r_1+r_2+i}=\overline{\sigma_{r_1+i}}$. Thus $\sigma_1,\ldots,\sigma_{r_1}$ are the real embeddings and $\sigma_{r_1+1},\ldots,\sigma_{r_1+r_2}$ the chosen complex embeddings for each complex place. By unravelling the definitions, we see that we are asked to compute the determinant of the matrix
	\[
		\begin{pmatrix}
		\Re(\sigma_1(\alpha_1)) & \cdots & \Re(\sigma_1(\alpha_n)) \\
		\vdots  & \ddots & \vdots  \\
		\Re(\sigma_{r_1+r_2}(\alpha_1)) & \cdots & \Re(\sigma_{r_1+r_2}(\alpha_n)) \\
		\Im(\sigma_{r_1+1}(\alpha_1)) & \cdots & \Im(\sigma_{r_1+1}(\alpha_n))\\
		\vdots  & \ddots & \vdots  \\
		\Im(\sigma_{r_1+r_2}(\alpha_1)) & \cdots & \Im(\sigma_{r_1+r_2}(\alpha_n))\\
		\end{pmatrix}
	\]
	Observe that it is obtained from the matrix $(\sigma_j(\alpha_i))_{i,j=1,\ldots,n}$ by the transformation on the rows corresponding to complex embeddings given schematically by
		\[
		\begin{pmatrix}
		\overline{\sigma(\alpha)}\\
		\sigma(\alpha)
		\end{pmatrix}
		\xrightarrow{R_1\leftarrow R_1+R_2}
		\begin{pmatrix}
		2\Re(\sigma(\alpha))\\
		\sigma(\alpha)
		\end{pmatrix}
		\xrightarrow{R_2\leftarrow R_2-\frac{1}{2} R_1}
		\begin{pmatrix}
		2\Re(\sigma(\alpha))\\
		i\Im(\sigma(\alpha))
		\end{pmatrix}
		\xrightarrow{R_2\leftarrow 1/i R_21}
		\begin{pmatrix}
		2\Re(\sigma(\alpha))\\
		\Im(\sigma(\alpha))
		\end{pmatrix},
		\]
		plus some permutation of the rows. Since $\left(\det(\sigma_j(\alpha_i))_{i,j}\right)^2=\Delta_K$ and the sign of $\Delta_K$ is $(-1)^{r_2}$ (Brill's theorem \cite[Lemma 2.2]{Washington97}), we obtain
		\[
		\Disc(\ZZ^D) = |(\frac{1}{i})^{r_2}\det (\sigma_j(\alpha_i))_{i,j} | = sqrt{|\Delta_K|}
		\]
		so that
		\[
		\chi_X(\ZZ^D)=\frac{2^{r_1}(2\pi)^{r_2}h_K R(\ZZ^D)}{\omega \sqrt{|\Delta_K|}}
		\]
		
		It remains to determine $R(\ZZ^D)$. By definition, this is $\det_{\ZZ,\RR} B^\bullet$ where $B^\bullet$ is the acyclic complex with first non-zero term in degree $0$
		\[
		0 \to (\cal{O}_K^\times)_\RR \to (\prod_v \ZZ)_\RR \to (\ZZ_\RR)^\vee \to 0
		\]
		Let us make the maps explicit. Denote $|\cdot|_v$ the normalized absolute value for the archimedean place $v$: $|\cdot|_v=|\cdot|$ if $v$ is real and $|\cdot|_v=|\cdot|^2$ if $v$ is complex. Denote also $L$ the logarithmic embedding $L:=\prod_v \log|\sigma_v(\cdot)|_v$. We have a commutative diagram
		\[\begin{tikzcd}[ampersand replacement=\&]
		\& {(\prod_{\sigma} \RR)^{G_\RR}} \\
		{(\cal{O}_K^\times)_\RR} \& {\prod_v \RR} \& \RR
		\arrow["\simeq", from=1-2, to=2-2]
		\arrow["{\prod_\sigma \log|\sigma(\cdot)|}", from=2-1, to=1-2]
		\arrow["\Sigma", from=1-2, to=2-3]
		\arrow["\Sigma"', from=2-2, to=2-3]
		\arrow["{L}"', from=2-1, to=2-2]
		\end{tikzcd}\]
		where the middle isomorphism comes from \cref{defi_discriminant} and $\Sigma$ is the sum of components map (see the discussion after \cref{defi_deligne_compact}). Here $\RR=H^0_{c,\cal{D}}(X,\ZZ^D_\RR)=H^0(X,\ZZ)_\RR^\vee$ obtains the canonical basis because the perfect pairing $H^0_{c,\cal{D}}(X,\ZZ^D_\RR)\times H^0(X,\ZZ)_\RR = \RR\times \RR \to \RR$ is given by the multiplication map. The ordering on the places gives an ordered  basis on $(\prod_v \ZZ)_\RR$. Let $x$ be the vector $(1,0,\ldots,0)\in (\prod_v \ZZ)_\RR$. We have $\Sigma(x)=1$. Denote $(\varepsilon_i)$ a system of fundamental units of $\cal{O}_K^\times$ and . Then $\det_{\ZZ,\RR} B^\bullet$ is the absolute value of the determinant of the matrix $P$ expressing $(L(\varepsilon_1),L(\varepsilon_2),\ldots,L(\varepsilon_{r_1+r_2-1}),x)$ in the basis of $\prod_v \ZZ)_\RR$. Put $N_i=1$ if $\sigma_i$ is real and $N_i=2$ if $\sigma_i$ is complex. Then we have,
		\[
		P =		\begin{pmatrix}
		N_1\log|\sigma_1(\varepsilon_1)| & \cdots & N_1\log|\sigma_{1}(\varepsilon_{s-1})| & 1 \\
		N_2\log|\sigma_2(\varepsilon_1)| & \cdots & N_2\log|\sigma_{2}(\varepsilon_{s-1})| & 0 \\
		\vdots  & \ddots & \vdots & \vdots  \\
		N_s\log|\sigma_s(\varepsilon_1)| & \cdots & N_s\log|\sigma_{s}(\varepsilon_{s-1})| & 0 \\
		\end{pmatrix}
		\]
		By definition, the regulator $R_K$ is the determinant of any $(s-1)\times (s-1)$ minor of the matrix $(N_i\log|\sigma_i(\varepsilon_j)|)_{i=1,\ldots,s;j=1,\ldots,s-1}$; thus by developping with respect to the last column we find
		\[
			R(\ZZ^D) = |\det P| = R_K.
		\]
		
		We obtain finally
		\[
		\chi_X(\ZZ^D)=\frac{2^{r_1}(2\pi)^{r_2}h_K R_K}{\omega \sqrt{|\Delta_K|}}
		\]
\end{proof}

\begin{rmk}
	For $F=\ZZ$, the complexes $\Lie_X(\ZZ^D)$, $R\Gamma_{W,c}(X,\ZZ^D)$, $R\Gamma_{c,\cal{D}}(X,\ZZ^D_\RR)$ and $R\Gamma_{ar,c}(X,\ZZ^D_\RR)$ are equal (up to a shift) to the complexes $R\Gamma_{dR}(X/\ZZ)/F^1$, $R\Gamma_{W,c}(X,\ZZ(1))$, $R\Gamma_c(X,\RR(1))$ and $R\Gamma_{ar,c}(X,\widetilde{\RR}(1))$ of \cite{Flach18}, in which the above formula was already obtained.
\end{rmk}

\begin{prop}\label{special_value_point}
	Let $i:x\to X$ be the inclusion of a closed point, $M$ a discrete $G_x$-module of finite type. Let $\chi_{W,c}$ denote the Weil-étale Euler characteristic of a $\ZZ$-constructible sheaf on $X$ from \cite{AMorin21} and let $R(M)$ be the absolute value of the determinant of the pairing
	\[
	H^0(G_x,M)_\RR \times H^0(G_x,M^\vee)_\RR \to \RR 
	\]
	in bases modulo torsion. Then
	\begin{align*}
	&\sum (-1)^i i \cdot \dim_\RR H^i_{ar,c}(X,(i_\ast M)^D_\RR) = - \rank_\ZZ H^0(G_x,M^\vee)\\
	&\chi_X((i_\ast M)^D)=\frac{[H^0(G_x,M)]_{tor}}{[H^1(G_x,M)]R(M)\log(N(v))^{\rank H^0(G_x,M)}}=\chi_{W,c}(i_\ast M)=\chi_{W,c}(i_\ast M^\vee)
	\end{align*}
	We have the special value formula 
	\begin{align*}
		\ord_{s=0}L_X((i_\ast M)^D)& = \sum (-1)^i i \cdot \dim_\RR H^i_{ar,c}(X(,i_\ast M)^D_\RR)\\
		L_X^\ast((i_\ast M)^D,0)& =\chi_X((i_\ast M)^D).
	\end{align*}
\end{prop}

\begin{proof}
	The computation of $\sum (-1)^i i \cdot \dim_\RR H^i_{ar,c}(X(,i_\ast M)^D_\RR)$ is straightforward. Let us compute $\chi_X((i_\ast M)^D)$. Using \cref{explicit_characteristic} we find
	\[
	\chi_X((i_\ast M)^D)=\frac{[H^0(G_x,M)_{tor}]R((i_\ast M)^D)}{[H^1(G_x,M)]}
	\]
	Here, the complex $B^\bullet$ is
	\[
	(H^0(X,i_\ast M)_\RR)^\vee \xrightarrow{\simeq} H^0(X,i\ast M^\vee)_\RR
	\]
	with first term in degree $2$, so $R((i_\ast M)^D)=\det_{\ZZ,\RR} B^\bullet$ is the inverse of the absolute value of the determinant of the pairing
	\[
	H^0(X,i_\ast M)_\RR\times H^0(X,i_\ast M^\vee) \to \RR
	\]
	Hence from the proof of \cref{duality_theorem} we see that
	\[
	R((i_\ast M)^D)=\frac 1 {R(M)\log(N(v))^{\rank H^0(G_x,M)}}
	\]
	which shows the first equality. The second equality is \cite[cor. 6.8]{AMorin21}. Let us show the last equality. If $M$ is torsion, then $M^\vee=M^\ast[-1]$ so by \cite[prop. 6.23]{AMorin21} both terms equal $1$. We can thus suppose that $M$ is torsion-free. Let $n=\rank_\ZZ H^0(G_x,M)$. Then $\rank_\ZZ \Hom_{G_x}(M,\ZZ)=n$ by the perfectness of the above pairing, $M^\vee = \Hom_\ZZ(M,\ZZ)$, and
	\[
	\chi_{W,c}(i_\ast M)=\frac 1 {[H^1(G_x,M)]R(M)\log(N(v))^n}
	\]
	\[
	\chi_{W,c}(i_\ast M^\vee)=\frac 1 {[H^1(G_x,M^\vee)]R(M^\vee) \log(N(v))^n}
	\]
	whence the result from the duality isomorphism $H^1(G_x,M)\simeq \Ext^1_{G_x}(M,\ZZ)^\ast = H^1(G_x,M^\vee)^\ast$ and the identification $R(M)=R(M^\vee)$.
	
	Finally, since $L_X((i_\ast M)^D,s)=L_X(i_\ast M^\vee,s)$ is the $L$-function of a (complex of) $\ZZ$-constructible sheaves, from \cite[thm. 6.24]{AMorin21} we find
	\begin{align*}
		&\ord_{s=0} L_X((i_\ast M)^D,s) = - \rank_\ZZ H^0(G_x,M^\ast)\\
		& L_X^\ast((i_\ast M)^D,0)=\chi_{W,c}(i_\ast M^\vee)
	\end{align*}
\end{proof}

\begin{thm}\label{special_value_constructible}
	Let $F$ be a constructible sheaf. Then
	\[\sum (-1)^i i\cdot \dim_\RR H^i_{ar,c}(X,F^D_\RR)=0\]
	In particular, we have the vanishing order formula $\ord_{s=0} L_X(F^D,s)=\sum (-1)^i i\cdot \dim_\RR H^i_{ar,c}(X,F^D_\RR)$. If $F$ is moreover tamely ramified, then
	\[
	\chi_X(F^D)=1.
	\]
	In particular, we have the special value formula $L_X^\ast(F^D,0)=\pm\chi_X(F^D)$.
\end{thm}

\begin{proof}
	The vanishing order formula is immediate; let us show the special value formula. If $F$ is supported on a finite closed subscheme, we reduce to the case $F=i_\ast M$ where $i:x\to X$ is a closed point and $M$ is a finite discrete $G_x$-module. In that case, \cref{special_value_point} together with \cite[prop. 6.23]{AMorin21} give $\chi_X((i_\ast M)^D)=1$. By dévissage, we reduce to checking that $\chi_X((g_\ast M)^D)=1$ for $M$ a tamely ramified finite discrete $G_K$-module. Let $L$ be a tamely ramified extension of $K$ such that $G_L$ acts trivially on $M$, and let $G_{L/K}:=\Gal(L/K)$. For $R$ a ring, denote $K_0'(R)$ the Grothendieck group of finitely generated left $R$-modules. The map $\theta : M\in G_{L/K}\text{-}\mathrm{Mod} \mapsto \chi_X((g_\ast M)^D)$ is multiplicative: indeed, if $0\to M ' \to M \to M'' \to 0$ is an exact sequence of $G_{L/K}$-modules then we have an exact sequence $0 \to g_\ast M' \to g_\ast M \to g_\ast M'' \to N \to 0$ where $N\subset R^1g_\ast M'$ is constructible supported on a finite closed subset. Then $\chi_X(N)=1$ so $\chi_X(g_\ast M)=\chi_X(g_\ast M')\chi_X(g_\ast M'')$. Thus $\theta$ factors through $K_0'(\ZZ[G_{L/K}])$ and takes values in $\RR^\times_{>0}$ which is torsion-free. By \cite[Corollary 1]{Swan63}, the kernel of $K_0'(\ZZ[G_{L/K}]) \to K_0'(\QQ[G_{L/K}])$ is finite; since $M\otimes \QQ=0$, this implies that the class of $M$ is a torsion element in $K_0'(\ZZ[G_{L/K}])$ so $\theta(M)=\chi_X(g_\ast M)=1$.
\end{proof}

\begin{rmk}
	\begin{itemize}
		\item[]
		\item In \cite{AMorin21}, a similar formula for the Weil-étale Euler characteristic $\chi_X(F)$ of a constructible sheaf $F$ is proven by reduction to Tate's formula for the Euler characteristic of a global field (see \cite[I.5.1, II.2.13]{ADT} and \cite[6.23]{AMorin21} for the reduction). The method of the above proof applies \emph{mutatis mutandis} to $\chi_X(F)$ and thus gives an alternative proof of Tate's formula.
		\item If $F$ is a tamely ramified constructible sheaf, we have $R\Gamma_{ar,c}(X,F^D_\RR)=0$, $R\Gamma_{W,c}(X,F^D)\simeq R\Gamma_{c,B}(X,F^D)$, $\Lie_X(F^D)=\Ext^1_{G_K^t}(F_\eta,\cal{O}_{K^t})[0]$ and the two latter complexes are bounded with finite cohomology groups. Thus their Euler characteristics
		\[
		\chi_{W,X}(F^D):=\prod_{i\in \ZZ}[H^i_{W,c}(X,F^D)]^{(-1)^i}
		\]
		and 
		\[
		\chi_{L,X}(F^D):=[\Ext^1_{G_K^t}(F_\eta,\cal{O}_{K^t})]
		\]
		are well-defined and we have $\chi_X(F^D)=\chi_{W,X}(F^D)/\chi_{L,X}(F^D)$. On the other hand, using Artin-Verdier duality, the morphism of fiber sequences \cref{morphism_of_fiber_sequences} and \cite[prop. 6.23]{AMorin21}, we find
		\[
		\chi_{W,X}(F^D)=[F_\eta]^{[K:\QQ]}
		\]
		Thus the above proposition implies that
		\[
		[\Ext^1_{G_K^t}(F_\eta,\cal{O}_{K^t})]=[F_\eta]^{[K:\QQ]}
		\]
	\end{itemize} 
\end{rmk}

\begin{thm}[Special values theorem]\label{special_value_thm}
	Let $F$ be a $\ZZ$-constructible sheaf. We have the vanishing order formula
	\[
	\mathrm{ord}_{s=0}L_X(F^D,s)=\sum (-1)^i i\cdot \dim_\RR H^i_{ar,c}(X,F^D_\RR).
	\]
	If $F$ is tamely ramified, we have the special value formula
	\[
	L_X^\ast(F^D,0)=\pm\chi_X(F^D)
	\]
\end{thm}

\begin{proof}
	This follows from \cref{special_value_Z,special_value_point,special_value_constructible} by Artin induction.
\end{proof}

\appendix
\section{Duality for the Tate cohomology of finite groups}

Let $G$ be a finite group. The functor $R\hat{\Gamma}(G,-)$ is lax monoidal, so for any complex of $G$-modules $M$ we can construct the pairing
\[
R\hat{\Gamma}(G,M^\vee)\otimes^L R\hat{\Gamma}(G,M)\to R\hat{\Gamma}(G,\ZZ)\to \hat{H}^0(G,\ZZ)[0]=\ZZ/|G|\ZZ[0] \hookrightarrow \QQ/\ZZ[0]
\]
where $M^\vee=R\cal{H}om_G(M,\ZZ)$. In the following, we will also denote $M^\ast=\cal{H}om_G(M,\QQ/\ZZ)$ the Pontryagin dual.
\begin{thm}\label{duality_tate_finite_groups}
	Let $G$ be a finite group. The natural pairing
	\[
	R\hat{\Gamma}(G,M^\vee)\otimes^L R\hat{\Gamma}(G,M)\to \QQ/\ZZ[0]
	\]
	is perfect for any bounded complex $M$ of $G$-modules with finite type cohomology groups.
\end{thm}

\begin{proof}
	By definition, showing that the pairing is perfect is showing that the adjoint map $R\hat{\Gamma}(G,M^\vee) \to R\hat{\Gamma}(G,M)^\ast$ is an isomorphism. Passing to cohomology, this is equivalent (using {\cite[\href{https://stacks.math.columbia.edu/tag/0FP2}{Tag 0FP2}]{stacks-project}} and the injectivity of $\QQ/\ZZ$) to showing that the cup product pairing
	\[
	\hat{H}^i(G,M)\times \hat{H}^{-i}(G,M^\vee)\to \hat{H}^{0}(G,\ZZ) \to \QQ/\ZZ
	\]
	is perfect for each $i\in \ZZ$. Notice that if $M$ is induced, then $M^\vee$ is also induced; moreover, if $M_0$ is a finite type abelian group, $\ind^G M_0$ also is. Thus the usual dimension shifting argument reduces to checking that the pairing is perfect in just one degree. We proceed by Artin induction:
	\begin{itemize}
		\item In an exact sequence $0 \to M \to P \to Q \to 0$ or more generally a fiber sequence $M \to P \to Q$, if the theorem is true for two out of the three terms then it is true for the third. 
		\item We can filter a bounded complex with finite type cohomology groups by its truncations, which reduces to the case of a $G$-module $M$ of finite type.
		\item Let $H$ be a proper subgroup of $G$, let $M$ be a torsion-free finite type discrete $H$-module and let us consider the induced $G$-module $\ind_H^G M$. We want to show compatibility of the pairing in degree $0$ for $\ind_H^G M$ and the pairing in degree $0$ for $M$.
		
		We have $M^\vee=\Hom_\ZZ(M,\ZZ)$. Denote $\pi_\ast$ for the induction functor $\ind_H^G$ and $\pi^\ast$ for restriction to $H$. The functor $\pi_\ast$ is a right adjoint of $\pi^\ast$ and the finiteness of $G$ also makes $\pi_\ast$ into a left adjoint of $\pi^\ast$. Let us denote $\eta:\pi_\ast \pi^\ast \to \mathrm{id}$ the counit of the latter adjunction; the counit $\eta_\ZZ:\pi_\ast\ZZ \simeq \oplus_{G/H} \ZZ \to \ZZ$ is the sum map. As a right adjoint to a strict monoidal functor, $\pi_\ast$ has a lax monoidal structure $\pi_\ast(-) \otimes \pi_\ast(-) \xrightarrow{\mathrm{lax}} \pi_\ast (-\otimes -)$. Consider the following commutative diagram:
		\[\begin{tikzcd}[ampersand replacement=\&]
			{\pi_\ast M \otimes (\pi_\ast M)^\vee} \&\& \ZZ \\
			{\pi_\ast M \otimes \pi_\ast(M^\vee)} \& {\pi_\ast(M\otimes M^\vee)} \& {\pi_\ast \ZZ}
			\arrow["\simeq", from=2-1, to=1-1]
			\arrow["{\mathrm{ev}}", from=1-1, to=1-3]
			\arrow["{\mathrm{lax}}"', from=2-1, to=2-2]
			\arrow["{\pi_\ast(\mathrm{ev})}"', from=2-2, to=2-3]
			\arrow["{\eta=\sum}"', from=2-3, to=1-3]
		\end{tikzcd}\]
		We can apply the lax-monoidal functor $\hat{H}^0$ to the above to obtain the following diagram (using Shapiro's lemma's identifications):
			\[\begin{tikzpicture}[baseline= (a).base]
			\node[scale=.85] (a) at (0,0){
				\begin{tikzcd}[ampersand replacement=\&,column sep=tiny]
					{\hat{H}^0(G,\pi_\ast M)\otimes \hat{H}^0(G,(\pi_\ast M)^\vee)} \& {\hat{H}^0(G,\pi_\ast M \otimes (\pi_\ast M)^\vee)} \&\& {\hat{H}^0(G,\ZZ)=\ZZ/|G|\ZZ} \& {\QQ/\ZZ} \\
					{\hat{H}^0(G,\pi_\ast M)\otimes \hat{H}^0(G,\pi_\ast (M^\vee))} \& {\hat{H}^0(G,\pi_\ast M \otimes \pi_\ast(M^\vee))} \& {\hat{H}^0(G,\pi_\ast(M\otimes M^\vee))} \& {\hat{H}^0(G,\pi_\ast \ZZ)} \\
					{\hat{H}^0(H,M) \otimes \hat{H}^0(H,M^\vee) } \&\& {\hat{H}^0(H,M\otimes M^\vee) } \& {\hat{H}^0(H,\ZZ)=\ZZ/|H|\ZZ } \& {\QQ/\ZZ}
					\arrow["\simeq", from=2-2, to=1-2]
					\arrow["{\mathrm{ev}_\ast}", from=1-2, to=1-4]
					\arrow["{\mathrm{lax}_\ast}"', from=2-2, to=2-3]
					\arrow["{(\pi_\ast(\mathrm{ev}))_\ast}"', from=2-3, to=2-4]
					\arrow["{[G:H]}"', from=2-4, to=1-4]
					\arrow["\cup", from=1-1, to=1-2]
					\arrow["\simeq", from=2-1, to=1-1]
					\arrow["\cup", from=2-1, to=2-2]
					\arrow["\simeq", from=3-1, to=2-1]
					\arrow["\cup", from=3-1, to=3-3]
					\arrow["\simeq"', from=3-3, to=2-3]
					\arrow["\simeq"', from=3-4, to=2-4]
					\arrow["{\mathrm{ev}_\ast}"', from=3-3, to=3-4]
					\arrow[hook, from=3-4, to=3-5]
					\arrow[hook, from=1-4, to=1-5]
					\arrow[Rightarrow, no head, from=3-5, to=1-5]
					\arrow["{(2)}"{description}, draw=none, from=2-2, to=1-4]
					\arrow["{(3)}"{description}, draw=none, from=3-1, to=2-3]
					\arrow["{(1)}"{description}, draw=none, from=2-1, to=1-2]
					\arrow["{(4)}"{description}, draw=none, from=3-3, to=2-4]
					\arrow["{(5)}"{description}, draw=none, from=3-4, to=1-5]
				\end{tikzcd}
	};
	\end{tikzpicture}\]
		Square $(1)$ commutes by functoriality of the cup product, square $(2)$ commutes by functoriality of $\hat{H}^0(G,-)$, square $(3)$ commutes by inspection, square $(4)$ commutes by functoriality in Shapiro's lemma and square $(5)$ commutes because $(e^{\frac{2i\pi}{|G|}})^{[G:H]}=e^{\frac{2i\pi}{|H|}}$. As the bottom left square commutes, the whole diagram commutes and the pairing for $M$ is perfect if and only if the pairing for $\mathrm{ind}_H^G M$ is perfect.
		
		\item For $M=\ZZ$ we have $M^\vee=\ZZ$. The map $R\hat{\Gamma}(G, M^\vee)\to R\hat{\Gamma}(G,M)^\ast$ corresponds in degree $0$ to the adjoint map to the cup product pairing $\ZZ/|G|\ZZ \otimes \ZZ/|G|\ZZ \to \ZZ/|G|\ZZ \to \QQ/\ZZ$ induced by the multiplication $\ZZ \otimes \ZZ \to \ZZ$. Thus it is perfect.
		\item For $M$ finite as an abelian group, we have $R\mathcal{H}om_{G}(M,\QQ)=0$ so the fiber sequence $\ZZ \to \QQ \to \QQ/\ZZ$ gives $M^\vee\simeq M^\ast[-1]$. For any bounded complex $M$, the complex $M\otimes \QQ$ is a complex of $\QQ$-vector spaces hence cohomologically trivial. The map $\QQ/\ZZ[-1] \to \ZZ$ thus induces an isomorphism of pairings
		\[\begin{tikzcd}
			{R\hat{\Gamma}(G,M^\ast)[-1]\otimes^L R\hat{\Gamma}(G,M)} & {R\hat{\Gamma}(G,\QQ/\ZZ)[-1]} & {\hat{H}^{-1}(G,\QQ/\ZZ)[0]} & {\QQ/\ZZ[0]} \\
			{R\hat{\Gamma}(G,M^\vee)\otimes^L R\hat{\Gamma}(G,M)} & {R\hat{\Gamma}(G,\ZZ)} & {\hat{H}^{0}(G,\ZZ)[0]} & {\QQ/\ZZ[0]} \\
			& {}
			\arrow[from=1-1, to=1-2]
			\arrow[from=1-2, to=1-3]
			\arrow[from=1-3, to=1-4]
			\arrow[from=2-3, to=2-4]
			\arrow[from=2-2, to=2-3]
			\arrow[from=2-1, to=2-2]
			\arrow["\simeq"', from=1-1, to=2-1]
			\arrow["\simeq"', from=1-2, to=2-2]
			\arrow[Rightarrow, no head, from=2-4, to=1-4]
			\arrow["\simeq"', from=1-3, to=2-3]
		\end{tikzcd}\]
		It follows that it is enough to prove that the map $R\hat{\Gamma}(G, M^\ast)[-1]\to R\hat{\Gamma}(G,M)^\ast$ is an isomorphism in degree $0$, i.e. that the cup-product pairing
		\[
			\hat{H}^{-1}(G,M^\ast)\times \hat{H}^0(G,M)\to \hat{H}^{-1}(G,\QQ/\ZZ)\to \QQ/\ZZ
		\]
		induced by the Pontryagin duality pairing $M^\ast\otimes M\to \QQ/\ZZ$ is perfect. For $g\in G$, $g$ acts on $M^\ast$ as the transpose of $g^{-1}$. Thus, on $M^\ast$, $N$ acts as $N^t$ and the family $(1-g)_{g\in G}$ is a permutation of the family $((1-g)^t)_{g\in G}$.
		In the perfect pairing $M^\ast\otimes M\to \QQ/\ZZ$, we have $\prescript{\perp}{}{\left(\bigcap_{g\in G}\Ker(1-g)\right)}=\sum_{g\in G} \mathrm{Im}((1-g)^t)$ and $\Ker(N)^\perp=\mathrm{Im}(N^t)$, so this pairing induces a perfect pairing between the subquotients:
		\[
		\Ker(N^t)/\sum_{g\in G} \mathrm{Im}((1-g)^t) \times \bigcap_{g\in G}\Ker(1-g)/\mathrm{Im}(N) \to \QQ/\ZZ
		\]
		This is exactly what we had to prove.
	\end{itemize}
\end{proof}

\section{The maximal tamely ramified extension of a number field}
In this section we discuss the tame Galois group of a number field. The results are certainly known but we did not find a convenient reference.

Let $K$ be the field of fractions of a henselian DVR with finite residue field. A finite extension $L/K$ is called tamely ramified if the ramification index $e_{L/K}$ is prime to the residual characteristic. We have the following properties:
\begin{prop}[{\cite[7.8, 7.9, 7.10]{NeukirchANT}}]
	\begin{enumerate}
		\item[]
		\item If $M/L/K$ is a tower of finite extensions, then $M/K$ is tamely ramified if and only if $M/L$ and $L/K$ are.
		\item The composite of two tamely ramified finite extensions of $K$ is tamely ramified.
		\item The maximal tamely ramified extension $K^t$ is defined as the composite of all finite tamely ramified extensions of $K$ inside $K^{sep}$. Its finite subextensions are tamely ramified.
	\end{enumerate}
\end{prop}

Let $K$ be a number field. For each finite place $v$ of $K$, which corresponds to a closed point of $X=\Spec(\cal{O}_K)$, denote $K_v$ the henselian local field at $v$. Let $K^{sep}$ be a separable closure of $K$ and for each finite place $v$, choose an embedding $K_v\hookrightarrow K^{sep}$; this determines a place $\overline{v}$ of $K^{sep}$ above $v$ and gives $K^{sep}$ the structure of a separable closure of $K_v$. Denote $G_K$, resp. $G_{K_v}$ the absolute Galois group of $K$ resp. $K_v$. The previous proposition is adapted to the global case by working simultaneously at all places, as following: a finite extension $L/K$ is called tamely ramified if for all finite places $v$ of $K$ and $w$ of $L$ above $v$, the finite extension $L_w/K_v$ is tamely ramified. We then have:
\begin{cor}
	\begin{enumerate}
		\item[]
		\item If $M/L/K$ is a tower of finite extensions, then $M/K$ is tamely ramified if and only if $M/L$ and $L/K$ are.
		\item The composite of two tamely ramified finite extensions of $K$ is tamely ramified.
		\item The maximal tamely ramified extension $K^t$ is defined as the composite of all finite tamely ramified extensions of $K$ inside $K^{sep}$. Its finite subextensions are tamely ramified.
	\end{enumerate}
\end{cor}
\begin{proof}
	This follows from the previous propositions by taking henselizations, using the two following observations:
	\begin{itemize}
		\item For any choice of finite place in any of the three fields in a tower $M/L/K$, there exists a compatible system of places above and below it.
		\item If $L/K$ and $M/K$ are two finite extensions inside $K^{sep}$ and $z$ is a finite place of the composite extension $LM$ mapping to places $w$, $w'$ of $L$ and $M$ then $(LM)_z=L_w M_{w'}$; indeed, $(LM)_z$ is the smallest henselian field inside $K^{sep}$ containing $LM$ so $(LM)_z\subset L_w M_{w'}$, and the other inclusion is immediate.
	\end{itemize}
\end{proof}
The extension $K^t/K$ is Galois; we define the tame Galois group of $K$ as $G_K^t:=\mathrm{Gal}(K^t/K)$.

For a finite place $v$ of $K$, denote $P_v\subset D_v$ the wild ramification and decomposition subgroup of the place $\bar{v}$ of $K^{sep}$. There is an identification $D_v=G_{K_v}$ under which we have $P_v=\Gal(K^{sep}/K_v^t)$. We now characterise the tame Galois group in terms of the wild inertia subgroups $P_v$.

\begin{prop}\label{tame_galois_group_and_wild_ramification}
	Let $N$ be the smallest closed normal subgroup of $G_K$ containing $P_v$ for all finite places $v$. Then $N=\Gal(K^{sep}/K^t)$, and consequently $G_K^t=G_K/N$.
\end{prop}

\begin{proof}
	We first show $N\subset \Gal(K^{sep}/K^t)$; since the latter is normal and closed it suffices to show that the elements of $P_v$, for any finite place $v$, fix $K^t$. Let $L/K$ be a finite tamely ramified extension. The place $\bar{v}$ of $K^{sep}$ determines a unique place $w$ of $L$ above $v$. The following diagram commutes
	
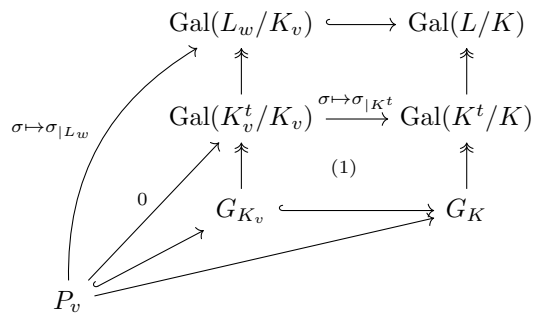
\[\begin{tikzcd}[ampersand replacement=\&]
	{} \& {\Gal(L_w/K_v)} \& {\Gal(L/K)} \\
	\& {\Gal(K_v^t/K_v)} \& {\Gal(K^t/K)} \\
	{} \& {G_{K_v}} \& {G_K} \\
	{P_v}
	\arrow[two heads, from=3-3, to=2-3]
	\arrow[two heads, from=2-3, to=1-3]
	\arrow[from=4-1, to=3-3]
	\arrow[two heads, from=3-2, to=2-2]
	\arrow[hook, from=4-1, to=3-2]
	\arrow["0", from=4-1, to=2-2]
	\arrow[hook, from=3-2, to=3-3]
	\arrow["{\sigma \mapsto \sigma_{|K^t}}", from=2-2, to=2-3]
	\arrow[two heads, from=2-2, to=1-2]
	\arrow[hook, from=1-2, to=1-3]
	\arrow["{\sigma\mapsto \sigma_{|L_w}}", bend left=30, from=4-1, to=1-2]
	\arrow["{(1)}"{description}, draw=none, from=3-2, to=2-3]
	\end{tikzcd}\]
	Indeed, the only nontrivial part is square $(1)$. If $L/K$ is a finite tamely ramified extension and $w$ is the place of $L$ induced by $\bar{v}$ then $L\subset L_w \subset K_v^t$; therefore there is an inclusion $K^t\subset K_v^t$ and the commutativity of $(1)$ follows. Since the diagram commutes, the elements of $P_v$ are sent to $0$ in $\Gal(L/K)$ so they fix any finite tamely ramified extension of $K$, hence also $K^t$.
	
	We now show $\Gal(K^{sep}/K^t)\subset N$. Since $N$ is normal, it suffices to show that for $N\subset U$ with $U$ an open normal subgroup, we have $\Gal(K^{sep}/K^t)\subset U$. This amounts to showing that the finite extension $L/K$ corresponding to an open normal subgroup $U$ containing $N$ is tamely ramified. Let $v$ be a finite place of $K$ and denote $w$ the place of $L$ induced by $\bar{v}$. Then $L_w=LK_v$ is the fixed field of $U\cap G_{K_v}$, which contains $P_v$, so $L_w\subset K_v^t$ is tamely ramified.
\end{proof}

\section{Miscellaneous results on proétale cohomology and condensed mathematics}

In this section we collect some results on condensed mathematics; these are certainly already known to experts.

\begin{defi}[{\cite[3.2.1]{Bhatt15}}]
	An object $F$ of a topos $\cal{X}$ is called weakly contractible if every surjection $G \to F$ has a section. We say that $\cal{X}$ is locally weakly contractible if it has enough weakly contractible coherent objects, i.e., each $X\in \cal{X}$ admits a surjection $\cup_i Y_i \to X$ with $Y_i$ a coherent weakly contractible object.
\end{defi}

The proétale topos of a scheme is locally weakly contractible \cite[4.2.8]{Bhatt15}; in particular the condensed topos (i.e. the proétale topos of a geometric point) is locally weakly contractible, and (the sheaves represented by) extremally disconnected profinite sets are a suitable family of weakly contractible objects.

\begin{lem}\label{mittag_leffler}
Let $\cal{X}$ be a locally contractible topos. If $(A_i)$ is a projective system of abelian group objects satisfying the Mittag-Leffler condition, then $R\lim A_i= \lim A_i$.
\end{lem}

\begin{proof}
	Let $Y$ be a weakly contractible object. As $\Gamma(Y,-)$ is exact and preserves injectives, we find $\Gamma(Y,-)=R\Gamma(Y,-)$ and thus $(R\lim A_i)(Y)=R\lim( A_i(Y))$. Moreover, $(A_i(Y))$ is a projective system of abelian groups satisfying the Mittag-Leffler condition, again by exactness and commutation with arbitrary limits of $\Gamma(Y,-)$; so we are done.
\end{proof}

\begin{lem}\label{mittag_leffler_discrete}
	Let $(A_i)$ be a family of discrete abelian groups satisfying the Mittag-Leffler condition, and let $\nu^\ast:Ab\to \mathrm{Cond}(Ab)$ be the constant sheaf functor. Then $(\nu^\ast A_i)$ satisfies the Mittag-Leffler condition.
\end{lem}

\begin{proof}
	The functor $\nu^\ast$ is exact and agrees with the functor $\underline{(-)}:Ab(\mathrm{Top}) \to \mathrm{Cond}(Ab)$ on discrete abelian groups \cite[4.2.12]{Bhatt15}; the latter functor is limit-preserving. An intersection of a decreasing family of abelian groups is a limit of the associated (discrete) topological abelian groups, so we are done.
\end{proof}

\begin{prop}[{\cite[6.1.17]{Bhatt15}}]\label{finite_proet_right_adjoint}
	Let $\pi:Y\to X$ be a finite morphism of finite presentation. Then $\pi_{proet,\ast}:D(Y_{proet})\to D(X_{proet})$ has a right adjoint. 
\end{prop}
\begin{proof}
	Since both source and target triangulated categories are compactly generated, it suffices to show that $\pi_{proet,\ast}$ commutes with direct sums. This can be checked on $w$-contractible affines. Since evaluation of sheaves of abelian groups on the proétale site at $w$-contractible affines commutes with colimits, this is easily shown using \cite[2.4.10]{Bhatt15}.
\end{proof}

In the following we follow \cite{SGAIV} and call a topology a family of sieves satisfying the usual axioms, and a pretopology a family of covers satisfying the usual axioms.

\begin{prop}
	The canonical topology on condensed sets induces the proétale topology (i.e. generated by finite jointly surjective families) on profinite sets under the Yoneda embedding.
\end{prop}

\begin{proof}
	Denote $J_{proet}$, resp. $J_{ind}$ the proétale topology resp. the topology induced from the canonical topology on condensed sets. Since condensed sets are covered under the canonical topology by profinite sets (under the fully faithful, limit-preserving Yoneda embedding), we conclude by \cite[Exp. IV, 1.2.1]{SGAIV} that $\mathrm{Sh}(Top^{pf},J_{proet})=\mathrm{Cond}(Set)\simeq \mathrm{Sh}(Top^{pf},J_{ind})$. Since the topology is caracterised by its category of sheaves \cite[Exp. II, 4.4.4]{SGAIV} we find $J_{proet}=J_{ind}$.
\end{proof}

\begin{defi}[{\cite[4.3.1]{Bhatt15}}]
	Let $G$ be a compactly generated topological group. The pro-étale site $BG_{proet}$ of $G$ is defined as the site of profinite continuous $G$-sets	with covers given by finite jointly surjective families.
\end{defi}

\begin{prop}[{\cite[Corollary 2]{Flach08}}]\label{bg_proet}
	Let $G$ be a compactly generated topological group. We have
	\[
	\mathrm{Sh}(BG_{proet})=\underline{G}\text{-}\mathrm{Cond}(Set)
	\]
	If moreover $G=\Gal(k^{sep}/k)$ for a field $k$, then
	\[
	\Sh(\Spec(k)_{proet})=\underline{G}{-}\mathrm{Cond}(Set)
	\]
\end{prop}

\begin{proof}
	The Yoneda embedding from profinite sets to condensed sets if fully faithful, and it can be extended to a right adjoint (hence limit-preserving) faithful functor $\underline{(-)}$, which is moreover fully faithful when restricted to compactly generated topological space \cite[1.7]{Clausen19}. The action of $G$ on a profinite set $X$ is encoded by a map $G\times X\to X$, where the product is computed in topological spaces. The product of a profinite set and a compactly generated space is compactly generated\cite[7.2]{Rezk18}\footnote{In the above reference, compactly generated spaces are called $k$-spaces}, so this is the same as a map $\underline{G}\times \underline{X} \to \underline{X}$. Thus we get a fully faithful embedding of profinite continuous $G$-sets in $\underline{G}\text{-}\mathrm{Cond}(Set)$. The canonical topology on $\underline{G}$-condensed sets is obtained by forgetting the $\underline{G}$-structure, and similarly for the proétale topology on profinite continuous $G$-sets, so that the canonical topology on the former induces the proétale topology on the latter by the previous proposition. By \cite[Exp. IV, 1.2.1]{SGAIV} it remains only to show that a $\underline{G}$-condensed set is covered by profinite $G$-sets; for any cover of the underlying condensed set by profinite sets $S_i$, $G\times S_i$ (where $G$ acts on the first factor) is a cover by profinite $G$-sets by an adjunction argument.
\end{proof}

\section{Strictly henselian local rings of singular schemes}

Let $X$ be an integral arithmetic scheme with finite normalization $\pi:Y\to X$, and let $x\in X$ be a closed point. We want to understand the fiber above the generic point of $X$ of the strictly henselian local ring at $x$ in terms of the fibers of strictly henselian local rings of points of $Y$ above $x$. We can localize at $x$ and we are thus in the following setting: let $(A,\fr{m})$ be a Noetherian integral local ring with fraction field $K$, perfect residue field $k$ and integral closure $B$. Can we describe $A^{sh}\otimes_A K$ in terms of $B$ and its strict henselizations at various maximal ideals?

We first have
\begin{lem}[{\cite[\href{https://stacks.math.columbia.edu/tag/07QQ}{Tag 07QQ}]{stacks-project}}]
	\[A^{sh}\otimes_A K = \prod_{i=1}^n \kappa(\fr{p}_i)\]
	where $\fr{p}_i$ are the prime ideals of $A^{sh}\otimes_A K$ above $(0)$; moreover each $k(\fr{p}_i)/k$ is separable algebraic.
\end{lem}
\begin{proof}
	$A^{sh}$ is Noetherian and flat over $A$ so it has finitely many minimal primes, which are exactly the primes lying above $(0)$. Moreover it is reduced as $A$ is reduced. Therefore $A^{sh}\otimes_A K$ is Noetherian and has finitely many prime ideals, all minimals, so $A^{sh}\otimes_A K$ is an Artinina ring; since it is reduced, its local rings are fields.
\end{proof}
The lemma says that $A^{sh}\otimes_A K$ is the total rings of fraction of $A^{sh}$. The total ring of fractions of the normalization of $A^{sh}$ is the same, so we want to determine the normalization of $A^{sh}$:
\begin{lem}[{\cite[\href{https://stacks.math.columbia.edu/tag/0CBM}{Tag 0CBM}]{stacks-project}}]
$B':=B\otimes_A A^{sh}$ is the normalization of $A^{sh}$
\end{lem}
\begin{proof}
	Normalization commutes with étale maps\footnote{It is easy to see the commutation with localization, so this reduces to commutation with standard étale maps} and filtered colimits, and $A^{sh}$ is a filtered colimit of étale $A$-algebras.
\end{proof}
By the previous lemma the minimal primes of $A^{sh}$ and $B'$ are in bijection. Moreover, as $B'$ is finite over $A^{sh}$, $B'$ is a finite product of strictly henselian local rings each finite over $A^{sh}$. Since $B'$ is normal, each local ring is moreover normal, hence integral. We now see that each minimal prime ideal is contained in a unique maximal ideal of $B'$, and vice-versa. Denote $\fr{q}_i$ the minimal prime of $B'$ corresponding to $\fr{p}_i$ and $\fr{M}_i$ the maximal ideal containing $\fr{q}_i$. We deduce
\[
\kappa(\fr{p}_i) = \kappa(\fr{q}_i) = \mathrm{Frac}((B')_{\fr{M}_i})
\]

\begin{lem}[{\cite[\href{https://stacks.math.columbia.edu/tag/08HV}{Tag 08HV}]{stacks-project}}]
	Denote $\fr{n}=\fr{M}_i\cap B$. We have
	\[
	(B')_{\fr{M}_i}=B_{\fr{n}}^{sh}
	\]
\end{lem}
\begin{proof}
	$B'$ is a filtered colimit of étale $B$-algebras and so is $B_{\fr{n}}^{sh}$. Since there is a natural morphism of $B$-algebras $B' \to B_{\fr{n}}^{sh}$, $B_{\fr{n}}^{sh}$ is also a filtered colimit of étale $B'$-algebras. Since $B_{\fr{n}}^{sh}$ is a stricly henselian local ring, it must be the strict henselian local ring of $B'$ at any prime above $\fr{n}$.
\end{proof}
Thus $\kappa(\fr{q}_i)=\mathrm{Frac}((B')_{\fr{M}_i})=\mathrm{Frac}(B_{\fr{n}}^{sh})$. Because $B'$ is a colimit of étale $B$-algebras, the maximal ideal $\fr{n}$ of $B$ is above $\fr{m}$. Moreover we have:
\begin{lem}[{\cite[\href{https://stacks.math.columbia.edu/tag/0C25}{Tag 0C25}]{stacks-project}}]
	The fiber above $\fr{n}$ in $B'$ is isomorphic to $\Hom_{k}(\kappa(\fr{n}),k^{sep})$
\end{lem}
\begin{proof}
	We have
	\[
	B'\otimes_{B} \kappa(\fr{n})=A^{sh}\otimes_A \kappa(\fr{n})=A^{sh}\otimes_A k \otimes_k \kappa(\fr{n})=k^{sep}\otimes_k \kappa(\fr{n})
	\]
	hence the result.
\end{proof}

Combining everything, we obtain
\begin{prop}\label{generic_point_strict_henselization}
	Let $(A,\fr{m})$ be a Noetherian integral local ring with fraction field $K$, perfect residue field $k$ and integral closure $B$ finite over $A$. The total ring of fractions of $A^{sh}$ is:
	\[
	A^{sh}\otimes_A K = \prod_{\fr{n}} \prod_{\Hom_k(\kappa(\fr{n}),k^{sep})} \mathrm{Frac}(B_{\fr{n}}^{sh})
	\]
	where $\fr{n}$ goes through the finitely many maximal ideals of $B$.
\end{prop}

\section{A lemma on determinants of total complexes}

\begin{lem}\label{trivialization_double_complex}
	Let $\cal{A}$ be an abelian category, let $C$ be a double complex of objects of $\cal{A}$ with uniformly bounded and acyclic rows and columns and let $\det$ be a determinant functor $\Ch^b(\cal{A})_{qis} \to \cal{P}$. We denote by $\mathbf{1}$ the unit of the Picard category $\cal{P}$. Define the filtration by rows on $\mathrm{Tot}C$ by $F^n_r \mathrm{Tot}C=\mathrm{Tot}(\sigma_r^{\leq n} C)$,\footnote{Here $r$ stands for \emph{rows}.} where $(\sigma_r^{\leq n} C)^{p,q}=C^{p,q}$ if $q\leq n$ and $0$ otherwise. The filtration by rows induces a commutative diagram
	\[\begin{tikzcd}[ampersand replacement=\&,column sep=1.5cm]
	{\det \mathrm{Tot} C} \& {\mathbf{1}} \\
	{\bigotimes_j (\det C^{\bullet,j})^{(-1)^j}} \& {\bigotimes_j (\mathbf{1})^{(-1)^j}}
	\arrow["\det0", from=1-1, to=1-2]
	\arrow[from=1-1, to=2-1]
	\arrow["{\otimes_j (\det 0)^{(-1)^j}}"', from=2-1, to=2-2]
	\arrow[from=1-2, to=2-2]
	\end{tikzcd}\]
\end{lem}
\begin{proof}
	The $n$-th graded piece of the filtration by rows of $\mathrm{Tot}C$ is $C^{\bullet,n}[-n]$.	Since the rows are exact, the map $0:\mathrm{Tot}C \to 0$ is a quasi-isomorphism; moreover it induces quasi-isomorphisms between the graded pieces of the filtrations by rows on $\mathrm{Tot}C$ and the trivial filtration on $0$ because columns are exact. We conclude with \cite[1.7]{Knudsen02} and \cite[2.3]{Breuning05}.
\end{proof}

\todos

\printbibliography

\end{document}